\documentclass[aap]{imsart}

\RequirePackage{amsthm,amsmath,amsfonts,amssymb}
\RequirePackage[numbers]{natbib}

\startlocaldefs

\usepackage{xcolor}
\usepackage{yfonts}
\usepackage{enumerate}
\usepackage{accents}

\newcommand{\Tint}{\mathbb{T}^{{\rm in}}}
\newcommand{\Tstr}{\mathbb{T}^{{\rm bs}}}
\usepackage{mathrsfs}
\RequirePackage[colorlinks,citecolor=blue,urlcolor=blue]{hyperref}

\definecolor{c8b8b8b}{RGB}{139,139,139}
\definecolor{c999999}{RGB}{153,153,153}
\definecolor{c9b9b9b}{RGB}{155,155,155}
\definecolor{c808080}{RGB}{128,128,128}
\definecolor{cfa796b}{RGB}{250,121,107}
\definecolor{cfb786a}{RGB}{251,120,106}
\definecolor{cfb7b6a}{RGB}{251,123,106}

\definecolor{c8b8b8b}{RGB}{139,139,139}
\definecolor{c999999}{RGB}{153,153,153}
\definecolor{c9b9b9b}{RGB}{155,155,155}
\definecolor{c808080}{RGB}{128,128,128}
\definecolor{cff0000}{RGB}{255,0,0}
\definecolor{c00ff00}{RGB}{0,255,0}
\definecolor{c14fe00}{RGB}{0,255,0}

\definecolor{c0000ff}{RGB}{0,0,255}
\definecolor{c008080}{RGB}{0,128,128}

\definecolor{cfb786a}{RGB}{251,120,106}
\definecolor{c009b33}{RGB}{0,155,51}
\definecolor{c009933}{RGB}{0,153,51}
\definecolor{c8b8b8b}{RGB}{139,139,139}
\definecolor{c999999}{RGB}{153,153,153}

\def \globalscale {1.000000}

\makeatletter
\def\timenow{\@tempcnta\time
	\@tempcntb\@tempcnta
	\divide\@tempcntb60
	\ifnum10>\@tempcntb0\fi\number\@tempcntb
	\multiply\@tempcntb60
	\advance\@tempcnta-\@tempcntb
	:\ifnum10>\@tempcnta0\fi\number\@tempcnta}
\makeatother

\usepackage[ddmmyyyy,hhmmss]{datetime}
\usepackage{mathtools}
\usepackage{euscript}

\usepackage{tikz}
\usetikzlibrary{intersections,shapes,trees,arrows,spy,positioning,decorations,arrows.meta,decorations.pathreplacing,patterns,calc}
\numberwithin{equation}{section}

\theoremstyle{plain}
\newtheorem{theo}{Theorem}[section]
\newtheorem{prop}[theo]{Proposition}

\newtheorem{lemme}[theo]{Lemma}

\newtheorem{cor}[theo]{Corollary}

\theoremstyle{remark}
\newtheorem{remark}[theo]{Remark}

\newcommand{\mofe}[1]{\color{black}{#1}\color{black}}

\newcommand{\modifbis}[1]{\color{black}{#1}\color{black}}

\renewcommand{\thesection}{\arabic{section}}

\DeclareFontFamily{U}{mathx}{\hyphenchar\font45}
\DeclareFontShape{U}{mathx}{m}{n}{
	<5> <6> <7> <8> <9> <10>
	<10.95> <12> <14.4> <17.28> <20.74> <24.88>
	mathx10
}{}
\DeclareSymbolFont{mathx}{U}{mathx}{m}{n}
\DeclareFontSubstitution{U}{mathx}{m}{n}
\DeclareMathAccent{\widecheck}{0}{mathx}{"71}

\def\E{\mathbb{E}}
\def\N{\mathbb{N}}
\def\P{\mathbb{P}}

\def\R{\mathbb{R}}
\def\Z{\mathbb{Z}}

\newcommand\Cs{{\mathcal C}}

\newcommand{\Db}{{\mathbb D}}
\newcommand{\cF}{{\mathcal F}}
\newcommand{\Ps}{{\mathcal P}}
\newcommand{\cL}{{\mathcal L}}
\newcommand{\cA}{{\mathcal A}}

\newcommand{\cM}{{\mathcal M}}

\newcommand{\rW}{{w}}
\newcommand\1{{\textbf{1}}}
\def\dd{\textnormal{d}}
\newcommand{\median}[1]{{\mathrm{median}\left\{{#1} \right\}}}
\newcommand{\supp}{{\mathrm{supp}}}
\newcommand{\cV}{{\mathcal V}}

\newcommand{\ind}[1]{{\1_{\{#1\}} }}

\newcommand{\wht}{{\mathbb{T}^{\rm in}}}

\newcommand{\wct}{{\mathbb{T}^{\rm bs}}}
\newcommand{\wctn}[1]{{\mathbb{T}^{\rm bs}_{#1}}}

\newcommand{\waitbds}[1]{{\mathbb{T}^{\rm bs}_{#1}}}
\newcommand{\waitbdsu}{{\mathbb{T}^{\rm bs}}}
\newcommand{\waitint}[1]{{\mathbb{T}^{\rm in}_{#1}}}

\newcommand{\rin}{{\mathring{r}}}
\newcommand{\rbs}{\bar{r}}
\newcommand{\tin}{{\tau}}
\newcommand{\tstop}{t_{\rm in}}
\newcommand{\tur}{{T}^{\rbs}}
\newcommand{\rur}{{R}^{\rbs}}
\newcommand{\uur}{{U}^{\rbs}}
\newcommand{\repsb}{\bar{r}_{\star}}
\newcommand{\turb}{{T}^{\repsb}}
\newcommand{\rurb}{{R}^{\repsb}}
\newcommand{\uurb}{{U}^{\repsb}}
\newcommand{\repsek}{\rbs}
\newcommand{\Nstuck}{N_{\dagger}}

\newcommand{\whtns}[1]{{{T}^{\rm in}_{#1}}}

\newcommand{\wctns}[1]{{{T}^{\rm bs}_{#1}}}

\newcommand{\wctny}[1]{{{\bar{T}}^{{\rm bs}}_{#1}}}
\newcommand{\whtnl}[1]{{{\hat{T}}^{{\rm in}}_{#1}}}

\newcommand{\renewal}{\mathscr{R}_Y}
\newcommand{\turk}{{T}^{\kappa}}
\newcommand{\rurk}{{R}^{\kappa}}
\newcommand{\uurk}{{U}^{\kappa}}
\newcommand{\turt}{{T}^{\theta}}

\newcommand{\uurt}{{U}^{\theta}}

\newcommand{\turtk}{{T}^{2\xi}}
\newcommand{\rurtk}{{R}^{2\xi}}
\newcommand{\uurtk}{{U}^{2\xi}}

\newcommand{\tu}[1]{{T}^{{#1}}}
\newcommand{\ru}[1]{{R}^{{#1}}}
\newcommand{\uu}[1]{{U}^{{#1}}}

\newcommand{\Yups}[2]{\mro{Y}{#1}{#2}}
\newcommand{\Ylows}[2]{\mru{Y}{#1}{#2}}

\newcommand{\capco}[1]{\mathrm{K}_{#1}}

\newcommand{\fray}{\mathsf{Y}}

\newcommand{\hL}{{L}}

\newcommand{\sel}{{\sigma}}
\newcommand{\mr}{{\mathrm{m}}}

\newcommand{\mru}[3]{{\underline{\mathrm{m}}({#1}_{#2-},{#3}) }}
\newcommand{\mro}[3]{{\overline{\mathrm{m}}({#1}_{#2-},{#3}) }}

\newcommand{\maxsupp}[1]{{\max\supp\,{#1}}}

\newcommand{\cadlag}{c{\`a}dl{\`a}g }

\endlocaldefs

\begin{document}

\begin{frontmatter}
\title{$\Lambda$-Wright--Fisher processes with general selection and opposing environmental effects: fixation and coexistence}
\runtitle{$\Lambda$-Wright--Fisher processes with selection}

\begin{aug}
\author[A]{\fnms{Fernando}~\snm{Cordero}\ead[label=e1]{fcordero@techfak.uni-bielefeld.de}},
\author[B]{\fnms{Sebastian}~\snm{Hummel}\ead[label=e2]{shummel@berkeley.edu}}
\and
\author[C]{\fnms{Gr\'egoire}~\snm{V{\'e}chambre}\ead[label=e3]{vechambre@amss.ac.cn}}
\address[A]{Faculty of Technology, Bielefeld University, Box 100131, 33501 Bielefeld, Germany\printead[presep={,\ }]{e1}}
\address[B]{Department of Statistics, University of California at Berkeley, 367 Evans Hall, Berkeley, CA 94720-3860, U.S.A.\printead[presep={,\ }]{e2}}
\address[C]{Academy of Mathematics and Systems Science, Chinese Academy of Sciences, No. 55, Zhongguancun East Road, Haidian District, Beijing, China\printead[presep={,\ }]{e3}}
\end{aug}

\begin{abstract}
Our results characterize the long-term behavior for a broad class of $\Lambda$-Wright--Fisher processes with frequency-dependent and environmental selection. In particular, we reveal a rich variety \mofe{of } parameter-dependent behaviors and provide explicit criteria to discriminate between them. That includes the situation in which the (entire) boundary is repelling -- a new phenomenon in this context. \mofe{This has } significant biological implications, because \mofe{it means that } selection alone can maintain coexistence. If a boundary point is attractive, we derive polynomial/exponential decay rates for the probability of not being polynomially/exponentially close to \mofe{that } boundary, depending on some weak/strong integrability conditions. Moreover, we provide a handy representation of the fixation probability. \modifbis{In our proofs we make use of Siegmund duality. The dual process can be sandwiched near the boundaries in between transformed L{\'e}vy processes}. In this way we relate the boundary behavior of the dual process to fluctuation properties of these L{\'e}vy processes and shed new light on previously established conditions for attractive/repelling boundary points. Our method allows us to treat models that so far could not be analyzed by means of moment or Bernstein duality. This closes an existing gap in the literature. 
\end{abstract}

\begin{keyword}[class=MSC]
\kwd[Primary ]{92D15}
\kwd{60J25}
\kwd[; secondary ]{60J27}
\end{keyword}

\begin{keyword}
\kwd{$\Lambda$-Wright--Fisher process}
\kwd{frequency-dependent selection} 
\kwd{environmental selection} 
\kwd{Siegmund duality} 
\kwd{fixation probability} 
\kwd{absorption probability} 
\kwd{L{\'e}vy processes} 
\kwd{fluctuation theory}
\kwd{convergence rates}
\end{keyword}

\end{frontmatter}
\tableofcontents

\section{Introduction} \label{s1}
The one dimensional \emph{$\Lambda$-Wright--Fisher process}~$X=(X_t)_{t\geq 0}$ is a L{\'e}vy-type process on the unit interval. Interest in this process stems from its connections to coalescent processes~\citep{Bertoin2003,Bertoin2005,Bertoin2006,Pitman1999,Sa99}, and its application in mathematical population genetics~\citep{Eldon2006,DEP12,birkner_blath_2009,Cordero2020}, where it appears as the type-frequency process of one of two alleles in an essentially infinite population that evolves under \emph{neutrality} (i.e. allele types do not affect offspring rates). The "$\Lambda$" in the name of the process appeals to a finite measure on~$[0,1]$ that drives the process and codes the rate of neutral offspring of various sizes. In many biological settings, evolutionary selection affects small offspring sizes and we then refer to it as \emph{selective drift}~\citep{BLW16,EGT10,foucart2013impact,GS18,CHS19,CF21}. In contrast, the effect of instantaneous change (hurricanes, floods,...) in an environment of a population can {make }one allele type particularly thrive and cause a large progeny. We refer to the latter phenomenon as \emph{environmental selection}. From a modeling perspective, selective drift and environmental selection lead to a variant of the $\Lambda$-Wright--Fisher process with a drift term vanishing at the boundary (so at $0$ and $1$) and an additional jump term, respectively. The resulting variant then exhibits connections to branching-coalescing processes~\citep{CV19,casanova2019effective,Vec2023}.

\smallskip

The boundary behavior of the $\Lambda$-Wright--Fisher process with selection can be peculiar. For example, it has been shown that for some parameter choices, one boundary point is repelling and the process converges almost surely to the other boundary point, when started \mofe{at any point in $(0,1)$ }\citep{foucart2013impact,Griffiths2014,GS18,CHS19}. This is particularly intriguing when one considers the biological implications: one of the two alleles is doomed to extinction no matter its initial frequency. 

\smallskip

Unfortunately, the boundary behavior is only known for a rather restrictive class of parameters. In fact, most previous works consider the simplest \mofe{possible } selective drift (modeling genic selection, i.e. \mofe{when } selection is independent of allele frequencies), which is parametrized by a positive \mofe{constant $\sigma_0^{}$}, and there is no environmental selection~\citep{foucart2013impact,bah2015,Griffiths2014}. To determine the \mofe{long-term } behavior in this case, the strength of selection has to be compared at the boundary to that of neutral evolution. The measure for the former \mofe{is~$\sigma_0^{}$}; the latter turns out to be well-captured by the quantity $\int_{[0,1]}\log(1/(1-r))r^{-2}\Lambda(\dd r)$. The two important cases are distinguished by the sign of \[C(\Lambda,\sigma_0^{})\coloneqq -\int_{[0,1]}\log\Big(\frac{1}{1-r}\Big)\frac{\Lambda(\dd r)}{r^2}+\sigma_0^{}.\] If $C(\Lambda,\sigma_0^{})<0$, neutral evolution dominates selection and both boundary points are attractive; if $C(\Lambda,\sigma_0^{})>0$, selection dominates neutral evolution and one boundary point is attractive and the other one is repelling.

\smallskip

The proof of this \emph{dichotomy} crucially relies on a \emph{moment duality} between~$X$ and the line-counting process $L=(L_t)_{t\geq 0}$ of a certain branching-coalescing structure, called the ancestral selection graph (ASG). The branching-coalescing structure has an interpretation in terms of potential genealogies that are compatible with the model, see~\citep{GS18} for details. In this structure, selection leads to branching, while neutral evolution leads to coalescence. The formal relation between $X$ and $L$ is \begin{equation}
	\forall x\in [0,1],\, n\in \N,\, t\geq 0\qquad \E_x[(1-X_t)^n]=\E_n[(1-x)^{L_t}]. \label{eq:momentduality}
\end{equation}
Heuristically, the proof idea underlying the dichotomy is that if selection dominates neutral evolution, then~$L$ is transient and the {right-hand side }of~\eqref{eq:momentduality} converges to~$0$ for any $x\in (0,1)$. Thus, also the left-hand side converges to~$0$, which is only possible if $X$ converges to~$1$ from any $x\in (0,1)$. \citep{GS18} showed that essentially the same dichotomy holds for drift terms \mofe{of very specific form}. In that setting, only one specific allele is favored by selection, with an intensity that changes with the type composition. Unfortunately, moment duality is only known to hold for this restrictive set of selective drifts; making the class of $\Lambda$-Wright--Fisher processes for which the boundary behavior is understood rather small. A first step to overcome these limitations was carried out by~\citep{Cordero2020}. For \mofe{drifts of the form $x(1-x)\sigma(x)$ for some polynomial~$\sigma$} , they generalized the moment duality to a \emph{Bernstein duality}. This new duality allows the representation of mixed moments of~$X$ in terms of a dual process that again has an interpretation in terms of a branching-coalescing structure. In this setting, they \mofe{determined a constant $C(\Lambda,\sigma)$ (generalizing $C(\Lambda,\sigma_0^{})$ for non constant functions $\sigma$) satisfying that if $C(\Lambda,\sigma)<0$, then } both boundary points are attractive. However, the case $C(\Lambda,\sigma)>0$ has continued to be notoriously difficult to analyze, also with the new duality. 
\smallskip

The moment duality and the corresponding dichotomy was extended in \citep{casanova2019effective,CV19} to situations where both selective drift and environmental selection favor always the same allele. Our aim here is to depart from this restrictive setting and characterize the boundary behavior of $\Lambda$-Wright--Fisher processes for a fairly general class of selective drifts and environments with opposing effects. The case of \mofe{strong } neutral reproductions, i.e. $
	\int_{[0,1]}r^{-1}\Lambda(\dd r)=\infty \text{ or } \Lambda(\{1\})>0,$ \mofe{will be treated in a forthcoming paper using the Bernstein-duality approach}. \modifbis{ Prominent examples for this setting are $\Lambda(\{0\})>0$, i.e. there is a diffusive component, or $\Lambda=\mathrm{Beta}(2-\alpha,\alpha)$ for $1<\alpha<2$}. In the present article we restrict our attention to $\Lambda$ satisfying 
\begin{equation}\label{eq:dust}
	\int_{[0,1]}r^{-1}\Lambda(\dd r)<\infty\text{ and }\Lambda(\{1\})=0,
\end{equation}
i.e. to the case with \emph{dust}~\cite[Thm 3.5]{Berestycki2009}.

\smallskip

Our first main result establishes explicit conditions \mofe{for the attraction or repulsion of the boundary points } for a broad class of $\Lambda$-Wright-Fisher processes with frequency-dependent selective drift and environmental selection that can favor both types. In particular, we complete the picture of the long-time behavior studied in~\citep{GS18,CHS19}, and extend results of~\citep{casanova2019effective,CV19} to environmental selection acting in opposing directions. Moreover, our results reveal parameter regions for the entire boundary to be repelling. In the context of~$\Lambda$-Wright--Fisher processes without mutation or immigration, this seems to be a new phenomenon.

\smallskip

\modifbis{An entirely repelling boundary has intriguing biological implications. It means that genetic variation at a single locus can be maintained by neutral and selective evolution alone; even in temporal homogeneous systems, outside the framework of the deterministic mutation--selection balance, or without any additional mutation or immigration term. In our framework, we show that a frequency-dependent selective drift is necessary to maintain coexistence and observe that already a simple instance of balancing selection can lead to such situations.}

\smallskip

If exactly one boundary point is attractive, one would like to quantify how fast $X$ converges to that point. We show that in this situation\mofe{, depending on some integrability conditions, the probability for $X$ not being exponentially (resp. polynomially) close to that attractive point decays at least exponentially (resp. polynomially)}. If both boundary points are attractive, the decay is always exponential and we provide a representation of the fixation probability that we consider useful for simulations.

\smallskip

To prove our results, we establish a Siegmund duality for~$X$. Siegmund duality has proven to be useful in various contexts~\citep{BLW16,BEH21,CF21,Kolo11,Kolokoltsov2013} and has the advantage to hold in more general settings than the moment duality. \modifbis{In our case, the Siegmund dual can be almost surely sandwiched in between transformed L{\'e}vy processes near the boundaries. Thus, we can employ the fluctuation theory of L{\'e}vy processes to analyze~$X$. In particular, we can give meaning to the mysterious quantity~$C(\Lambda,\sigma)$ as the limit rate of growth of the L{\'e}vy processes. Moreover, in addition to being interesting in its own right, the dual process enjoys in certain parameter settings some other characteristics useful to us. For example, to establish our results we make use of a renewal structure, properties of the stationary measure, and coalescence properties of its flow}. 

\section{Main results} \label{sec:main}
$\Lambda$-Wright Fisher processes provide a broad framework to model neutral evolution. Here we consider an extension that incorporates frequency-dependent selection and environmental selection.\mofe{ It is parametrized by an element $(\Lambda,\mu,\sel)$ of the set $\Ps\coloneqq \cM_f([0,1])\times\cM((-1,1))\times\Cs^1([0,1])$,
where $\cM_f([0,1])$ denotes the finite Borel measures on~$[0,1]$, $\cM((-1,1))$ the sigma-finite Borel measures on~$(-1,1)$, and $\Cs^1([0,1])$ are the continuously differentiable functions on $[0,1]$. We will restrict our attention to the case where $(\Lambda,\mu,\sel)$ belongs to the set
\begin{equation*}
\textstyle \Theta\coloneqq\left\{(\lambda,\nu,f)\in\Ps:0\neq \lambda\text{ satisfies~\eqref{eq:dust}},\,\nu(\{0\})=0,\,\int_{(-1,1)}\lvert r\rvert \nu(\dd r)<\infty\right\},
\end{equation*}}
For $(\Lambda,\mu,\sel)\in \Theta$, the
\emph{$\Lambda$-Wright-Fisher process} with selective drift~$\sel$ and environmental selection measure~$\mu$ is the solution $X=(X_t)_{t\geq 0}$ of the SDE \mofe{
\begin{align}\label{eq:SDEWFP}
	\dd X_t  &=\ \int_{(0,1)^2}r\big(\1_{\{u\leq X_{t-}\}}(1-X_{t-})-\textbf{1}_{\{u>X_{t-}\}} X_{t-} \big)N(\dd t, \dd r,\dd u)\nonumber\\
	&\ +X_t(1-X_t) \sel(X_t)\dd t \,+\,\int_{(-1,1)}rX_{t-}(1-X_{t-})S(\dd t,\dd r), 
\end{align} 
where $N(\dd t,\dd r,\dd u)$ is a Poisson measure on $[0,\infty)\times(0,1)^2$ with intensity $\dd t\times r^{-2}\Lambda(\dd r)\times \dd u$, and $S(\dd t, \dd r)$ is a Poisson measure on $[0,\infty)\times(-1,1)$ with intensity $\dd t \times \mu$; $N$ and $S$ are independent. We refer to the process also as $(\Lambda,\mu,\sel)$-Wright-Fisher process for short. For any $x\in[0,1]$, there is a pathwise unique strong solution to~\eqref{eq:SDEWFP} such that $X_0=x$ almost surely, and $X_t\in[0,1]$ for all $t\geq0$; we show this in Appendix~\ref{sec:ExistenceUniqueness} in a slightly more general framework applying a result of~\cite{li2012strong}}. We write $\P_x$ and $\E_x$, respectively, for the measure induced by~$X$ and its expectation under $X_0=x\in[0,1]$.
\smallskip

The $(\Lambda,\mu,\sel)$-Wright--Fisher process describes the evolution of the type composition of an infinite haploid population with two types, say type $a$ and type $A$, which is subject to neutral reproductions, frequency-dependent selection and environmental shocks. More precisely, $X_t$ represents the proportion of type $a$ individuals at time~$t$. The measure $\Lambda$ drives neutral reproductions; a point $(t,r,u)$ of~$N$ means that at time~$t$ a fraction~$r$ of the population is replaced by the offspring of one individual, its type being~$a$ (resp.~$A$) if $X_{t-}\leq u$ (resp. $X_{t-}>u$). The function $\sel$ models frequency-dependent selection; the strength and direction of selection at time $t$ depends on the modulus and the sign of~$\sel(X_t)$, respectively. The measure $\mu$ encodes the effect of the environment; a point $(t,r)$ of $S$ with $r>0$ (resp. $r<0$) means that at time $t$ conditions are exceptionally favorable for the reproduction of type $a$ (resp. $A$) individuals, so that a proportion~$\lvert r\rvert $ of the type $a$ (resp. $A$) population duplicates (i.e. each reproducing individual has one offspring) and the offspring replaces uniformly at random a proportion $\lvert r\rvert X_{t-}$ (resp. $\lvert r\rvert (1-X_{t-})$) of the population that was alive before the environmental shock.

\smallskip 
\modifbis{Our conditions on $\Lambda$ and $\mu$ prevent neutral reproductions and environmental effects from being too strong. For example, the conditions $\int_{[0,1]}r^{-1}\Lambda(\dd r)<\infty$ and $\int_{(-1,1)}\lvert r\vert \mu(\dd r)<\infty$ imply that the $r$-components of the jumps of $S$ and $N$ are locally summable (in particular, our processes have paths of bounded variation)}. \mofe{They also imply that $\Lambda(\{0\})=0$; without this condition, SDE \eqref{eq:SDEWFP} would have an extra diffusion term, see Appendix~\ref{sec:ExistenceUniqueness}. The condition $\Lambda(\{1\})=0$ prevents the process from jumping directly to the boundary}. \modifbis{Some cases with strong neutral reproductions (i.e. $
	\int_{[0,1]}r^{-1}\Lambda(\dd r)=\infty \text{ or } \Lambda(\{1\})>0,$) can be studied via the ASG. For example, in \citep{CHS19}, they consider the case where $\mu= 0$, $\sigma$ is a polynomial, and the ASG is recurrent (which is slightly stronger than $
	\int_{[0,1]}r^{-1}\Lambda(\dd r)=\infty \text{ or } \Lambda(\{1\})>0$). There, both boundary points are attractive. As mentioned in the introduction, a more general study of the case of strong neutral reproductions will be the object of a forthcoming paper}. \mofe{This is why we focus here on the parameter regime $\Theta$}. In full generality, this case has eluded previous analyses because here the corresponding ancestral processes can be transient, leading to difficulties in applying methods based on \mofe{Bernstein } duality.

\subsection{Boundary classification}\label{sec:main:accessibility}
Our results reveal a rich variety of parameter-dependent long-term behaviors for the model~\eqref{eq:SDEWFP} and allow to discriminate the different boundary behaviors for~$X$ by explicit criteria. 

\smallskip

A \mofe{point } $b\in\{0,1\}$ is said to be attractive for $X$ from $x\in(0,1)$ if $\P_x(\lim_{t\to\infty}X_t=b)>0$. Similarly, $b\in\{0,1\}$ is said to be repelling for $X$ from $x\in(0,1)$ if $\P_x(\lim_{t\to\infty}X_t=b)=0$. We say that $b\in\{0,1\}$ is \textit{attractive (repelling)} for $X$ if $b$ is attractive (repelling) for $X$ from any $x\in(0,1)$. \modifbis{In contrast to the notion of accessibility, the notion of attractiveness does not require that the point must be reached in finite time. Remark~\ref{boundaryfinitetime} below will show that in our case the boundary can only be reached asymptotically}.

\smallskip

To determine if a boundary point is attractive or repelling, one has to compare the neutral forces, parameterized by~$\Lambda$, with the selective forces, parametrised by~$\mu$ and~$\sel$. Define \begin{align*}
	 C_0(\Lambda,\mu,\sel)&\textstyle \coloneqq \left(\sel(0)- \int_{(-1,1)}\log\big(\frac{1}{1-r}\big) \bar{\mu}(\dd r) \right)-\int_{(0,1)}\log\big(\frac{1}{1-r}\big) \frac{\Lambda(\dd r)}{r^2},\\ 
	C_1(\Lambda,\mu,\sel)&\textstyle \coloneqq \left(-\sel(1) -\int_{(-1,1)}\log\big(\frac{1}{1-r}\big) \mu(\dd r)\right)-\int_{(0,1)}\log\big(\frac{1}{1-r}\big) \frac{\Lambda(\dd r)}{r^2},
\end{align*}
where $\bar{\mu}$ denotes the pushforward measure of $\mu$ under $r\mapsto -r$. For $(\Lambda,\mu,\sigma)\in\Theta$, we have $C_0(\Lambda,\mu,\sel),C_1(\Lambda,\mu,\sel)\in[-\infty,\infty)$. 
We note that the quantity $\int_{(0,1)}\log({1}/(1-r)) r^{-2}\Lambda(\dd r)$ has been previously referred to as the \emph{coalescence impact} of $\Lambda$, because it is related to properties of the associated $\Lambda$-coalescent (see, e.g., \citep[Def.~2.18]{CHS19}). 
\mofe{The seeming asymmetry of $C_0(\Lambda,\mu,\sigma)$ and $C_1(\Lambda,\mu,\sigma)$ in $\sigma$ can be explained as follows. Attractiveness of~$1$ for~$X$ is equivalent to attractiveness of~$0$ for $1-X$. Moreover, if $X$ is a $(\Lambda,\mu,\sigma)$-Wright--Fisher process, then $1-X$ is a $(\Lambda,\bar{\mu},\bar{\sigma})$-Wright--Fisher process, where $\bar{\sigma}(y)=-\sigma(1-y)$}. The term inside the bracket in the definition of $C_0(\Lambda,\mu,\sel)$ can be understood as the strength at which selective and environmental forces repel/push (depending on its sign) $X$ from/towards~$0$ when~$X$ is close to that boundary point; the integral with respect to $\Lambda$ represents the strength at which neutral forces push $X$ towards $0$. An analogous interpretation applies to the boundary point $1$ and~$C_1(\Lambda,\mu,\sel)$. This motivates to distinguish the parameter regions
\begin{align*}
	\Theta_0\coloneqq\{(\Lambda,\mu,\sel)\in\Theta: C_0(\Lambda,\mu,\sel)<0, C_1(\Lambda,\mu,\sel)>0\},\\
	\Theta_1\coloneqq\{(\Lambda,\mu,\sel)\in\Theta: C_0(\Lambda,\mu,\sel)>0, C_1(\Lambda,\mu,\sel)<0\},\\
	\Theta_2\coloneqq\{(\Lambda,\mu,\sel)\in\Theta: C_0(\Lambda,\mu,\sel)<0, C_1(\Lambda,\mu,\sel)<0\},\\
	\Theta_3\coloneqq\{(\Lambda,\mu,\sel)\in\Theta: C_0(\Lambda,\mu,\sel)>0, C_1(\Lambda,\mu,\sel)>0\}.
\end{align*}

To formulate our result, we require some control on the big jumps of $N$ and $S$. To this end, define for $\nu\in\cM(A)$, with $A= [-1,1]$ or $A=[0,1]$, and $\gamma\geq 0$, \begin{equation}\label{eq:weakint}
	w_\gamma(\nu)\coloneqq\int_{(1/2,1)}\log\Big(\frac{1}{1-r}\Big)^{1+\gamma} \nu(\dd r).
\end{equation}
\begin{theo}[Boundary classification]\label{thm:accessibility_condition}
	Let $(\Lambda,\mu,\sel)\in\Theta$. Let $X$ be the solution of \eqref{eq:SDEWFP} with $X_0=x\in (0,1)$. \begin{enumerate}
		\item[(0)] If $(\Lambda,\mu,\sel)\in\Theta_0$ and for some $\gamma>0$, $\rW_\gamma(r^{-2}\Lambda(\dd r))<\infty$ and $\rW_\gamma(\mu)<\infty$, then a.s. $\lim_{t\to\infty}X_t=0$, \mofe{i.e. $0$ is attractive and $1$ is repelling.} 
		\item[(1)] If $(\Lambda,\mu,\sel)\in\Theta_1$ and for some $\gamma>0$, $\rW_\gamma(r^{-2}\Lambda(\dd r))<\infty$ and $\rW_\gamma(\bar{\mu})<\infty$, then a.s. $\lim_{t\to\infty}X_t=1$, \mofe{i.e. $1$ is attractive and $0$ is repelling.} 
		\item[(2)] If $(\Lambda,\mu,\sel)\in\Theta_2$, then a.s. $\lim_{t\to\infty}X_t$ exists and belongs to $\{0,1\}$. Moreover, both $0$ and~$1$ are attractive. 
		\item[(3)] If $(\Lambda,\mu,\sel)\in\Theta_3$, then $0$ and~$1$ are repelling. Moreover, $X$ admits a stationary distribution~$\pi_X$ with~$\pi_X(0,1)=1$ and, for any $x \in (0,1)$, \mofe{under $\P_{x}$, $X_t$ } converges in distribution to $\pi_X$. 
	\end{enumerate}
\end{theo}
Theorem~\ref{thm:accessibility_condition} provides, under some integrability conditions on $\Lambda$ and $\mu$, explicit criteria for a boundary point to be attractive/repelling. The proof is provided in Section~\ref{sec:accessibility}.
\smallskip

Theorem~\ref{thm:accessibility_condition} has important biological implications. For $(\Lambda,\mu,\sel)\in\Theta_0$ (resp. $\Theta_1$), type $A$ (resp. type~$a$) almost surely takes over the entire population eventually, regardless of its initial frequency. For $(\Lambda,\mu,\sel)\in\Theta_2$, both types have a positive probability to eventually fixate in the population. In all three cases, it is of practical relevance to know the rate at which fixation (or equivalently extinction) essentially takes place. We provide an answer in Theorem~\ref{thm:survival_probability} below. Moreover, if both boundary points are attractive, it is a natural question how the fixation probabilities depend on the initial distribution. We provide a useful representation for the fixation probability in Theorem~\ref{thm:representation_absorption_probability} below. For $(\Lambda,\mu,\sel)\in\Theta_3$, both types persist in the long run, which leads to interesting biological implications as discussed in the introduction. To the best of our knowledge, this is a new phenomenon, even in cases where there is no environmental selection.

\begin{remark}[\mofe{Balancing selection}] \label{rem:balancing}
	In our framework, balancing selection~\citep{hedrick2007balancing} can be modeled by choosing~$\sigma(x)=1-2x$. Consider the situation with no environmental effects (i.e. $\mu= 0$). \mofe{In this case,} 
\[C_0(\Lambda,\mu,\sel)>0 \ \text{and} \ C_1(\Lambda,\mu,\sel)>0 \Leftrightarrow \, \textstyle 1>\int_{(0,1)}\log(\frac{1}{1-r})\frac{\Lambda(\dd r)}{r^2}. \] 
\mofe{According to Theorem~\ref{thm:accessibility_condition}(3), balancing selection maintains genetic variation in this context}. Thus, our result adds another argument to the long lasting controversy between the `(neo)classical' and `balance' school in explaining the observed abundance of genetic variation in natural populations~\citep{Beatty1987}.
\end{remark}

\begin{remark}[Coexistence requires frequency-dependent selection] \label{coeximplyfreqdepsel}
	If $(\Lambda,\mu,\sel)\in\Theta_3$, then $\sel$ is not constant. To see this, consider $(\Lambda,\mu,\tilde{\sel})\in\Theta$ with $\tilde{\sel}$ constant. Then,
	\begin{align*}
		\textstyle C_0(\Lambda,\mu,\tilde{\sel})+C_1(\Lambda,\mu,\tilde{\sel})=&-2\textstyle \int_{(0,1)}\!\log\big(\frac{1}{1-r}\big) \frac{\Lambda(\dd r)}{r^2}- \int_{(-1,1)}\!\log\big(\frac{1}{1-r^2}\big) \mu(\dd r)<0.
	\end{align*}
	Hence, $C_0(\Lambda,\mu,\tilde{\sel})<0$ or $C_1(\Lambda,\mu,\tilde{\sel})<0$, which implies that $(\Lambda,\mu,\tilde{\sel})\notin\Theta_3$. In particular, without frequency-dependent selection at least one of the two boundary points is attractive in our framework.
\end{remark}
Theorem~\ref{thm:accessibility_condition} connects to the existing literature as follows.
\smallskip

\emph{No environment.} In the absence of environmental shocks (i.e $\mu=0$), Theorem \ref{thm:accessibility_condition} generalizes \citep[Thm.~4.3]{bah2015}, \citep[Thm.~2.14]{CHS19}, \citep[Thm. 1.1]{foucart2013impact}, \citep[Thm. 4.6]{GS18}, and \citep[Thm.~3]{Griffiths2014} for the non-critical case (i.e. for $C_0(\Lambda,\mu,\sigma)\neq 0$). \citep{foucart2013impact,Griffiths2014,bah2015} consider the case of a constant~$\sigma$; \citep{GS18} and \citep{CHS19} consider special cases of frequency-dependent selection. In none of these studies coexistence has been observed. For \citep{foucart2013impact,Griffiths2014,bah2015} this is in line with Remark \ref{coeximplyfreqdepsel}. In \citep{GS18}, $\sel(x)=-\sum_{i\geq0} s_ix^i$, where $(s_i)_{i\in \N_0}$ is a decreasing sequence in~$\R_+$ whose series is absolutely convergent. In this case, $\sel(0)\leq 0$, and hence $C_0(\Lambda,0,\sel)<0$ and $0$ is always attractive. This is proved in \citep{GS18} using moment duality. Unfortunately, existence of a moment dual seems to be the exception rather than the rule. However, for $\sel$ being a general polynomial, \citep{CHS19} established a more involved duality relation that allows the representation of mixed moments of~$X$ as the expectation of a Bernstein polynomial with (random) coefficients driven by a dual process. It was shown that if the Bernstein-dual process to~$X$ is recurrent, both boundary points are attractive. \modifbis{Inspecting the conditions in \citep[Prop.~2.27]{CHS19} reveals that under~\eqref{eq:dust}, the Bernstein-dual can sometimes be recurrent and all cases where this happens fall in our parameter region $\Theta_2$ (the inclusion being strict), so that part of their result is also covered by Theorem~\ref{thm:accessibility_condition}. What happens if the dual is transient had remained an open problem, already in the absence of environmental selection}. 
\smallskip

\emph{One sided environmental shocks.} Environments favoring exclusively one type (i.e. when $\mu$ is either supported in $(0,1)$ or in $(-1,0)$) have been considered in \citep{casanova2019effective,CV19}. \citep{CV19} considers $\Lambda=\delta_0$ and shows that both boundary points are always attractive, which is related to the existence of a (positive recurrent) moment dual. \citep{casanova2019effective} assumes  $\smash{\int_{[0,1]}\log(1/(1-r))r^{-2}\Lambda(\dd r)<\infty}$ and $\sel(x)=-\sum_{i\geq0} s_ix^i$, where $(s_i)_{i\in \N_0}$ is a decreasing sequence in~$\R_+$ whose series is absolutely convergent. In that setting, coexistence has also not been observed. Note that their environmental selection mechanism is somewhat different. In their notation: choosing for $y\in (0,1)$, $K_y-1$ as a Bernoulli with parameter~$1-y$ leads to our setup with $\mu=\delta_1$ \modifbis{(however, we assume $\mu \in \cM((-1,1))$).}

\smallskip

\emph{Wright--Fisher diffusion with opposing environments.} A first step in studying the effect of opposing environments has been made in \citep{Vec2023}. There, $\Lambda=\delta_0$, $\sigma\leq 0$ and $\mu$ is assumed to be a finite measure in $(-1,1)$. \modifbis{In that setting, both boundary points are attractive and \citep{Vec2023} uses combinatorial properties of the ASG to derive a series representation and Taylors expansions for the fixation probability (i.e. the probability that $X$ is absorbed at the boundary point $1$).}

\subsection{Essential fixation and extinction}\label{sec:main:survival}
If a boundary point is attractive, one would like to know how long it takes for the process to essentially reach that point. In more biological terms, we would like to determine the time until fixation of one of the two types has essentially taken place. More precisely, for $\rho,t>0$, we say that type $a$ is essentially extinct (fixated) of polynomial order~$\rho$ at time~$t$ if $X_t\leq t^{-\rho}$ ($X_t\geq 1-t^{-\rho}$). Similarly, we speak of
extinction (fixation) of exponential order~$\rho$ at time~$t$ if $X_t\leq e^{-\rho t}$ ($X_t\geq 1-e^{-\rho t}$).

\smallskip

We distinguish two major regimes, depending on integrability conditions on the measures $\Lambda$ and $\mu$. To this end, for $\gamma\geq 0$ and $\nu\in\cM(A)$, with $A=[-1,1]$ or $A=[0,1]$, we define 
\begin{equation}\label{eq:strongint}
	s_\gamma(\nu)\coloneqq \int_{(1/2,1)} \left(\frac{1}{1-r}\right)^{\gamma} \nu(\dd r).
\end{equation}
Clearly, for $\gamma>0$, $s_\gamma(\nu)<\infty$ implies $w_\gamma(\nu)<\infty$ ($w_\gamma$ is defined in \eqref{eq:weakint}). In this sense, it is a stronger integrability condition on the measure around $1$.
To simplify the presentation of the next result, we use the notation $\mu_0\coloneqq\mu$ and $\mu_1\coloneqq\bar{\mu}$.
\begin{theo}[Essential fixation/extinction] \label{thm:survival_probability} 
\,
	\begin{enumerate}
		\item[(1)] Let $b\in\{0,1\}$ and $(\Lambda,\mu,\sigma)\in \Theta_b$.
		\begin{enumerate} 
			\item[(i)] (polynomial decay) Suppose that for some~$\gamma>0$, $$w_\gamma(r^{-2}\Lambda(\dd r))<\infty\quad\text{and}\quad w_\gamma(\mu_b)<\infty.$$
			Then, for any $\alpha \in (0,\gamma)$ and $\varepsilon \in (0,1)$, there is a constant $K=K(\alpha,\varepsilon)> 0$ such that for any $x \in (\varepsilon b,1-\varepsilon(1-b))$, $\rho \in (0,4 \alpha)$ and $t \geq 0$, \begin{equation}
				\mathbb{P}_x  ( |X_t-b| \leq t^{-\rho}  ) \geq 1-K t^{-(\alpha-\rho/4)}. 
			\end{equation}
			\item[(ii)] (exponential decay) Suppose that for some~$\gamma>0$, $$s_\gamma(r^{-2} \Lambda(\dd r))<\infty\quad \text{and}\quad s_\gamma(\mu_b)<\infty.$$ Then for any $\varepsilon \in (0,1)$ there are constants $K_1=K_1(\varepsilon), K_2=K_2(\varepsilon) > 0$ such that for any $x \in (\varepsilon b,1-\varepsilon(1-b))$, $\rho \in (0,4K_2)$, and $t \geq 0$,
			\begin{equation}
				\mathbb{P}_x  ( |X_t-b| \leq e^{-\rho t}  ) \geq 1-K_1 e^{-(K_2-\rho/4) t}.
			\end{equation}
		\end{enumerate} 
		\item[(2)] Let $(\Lambda,\mu,\sigma)\in \Theta_2$. Then, there are constants $K_1, K_2 > 0$ such that for any $\rho \in (0,4K_2)$, $x \in (0,1)$ and $t \geq 0$,
		\begin{eqnarray}
			\mathbb{P}_x \left ( X_t \in [0,e^{-\rho t}]\cup [1-e^{-\rho t},1] \right ) \geq 1-K_1 e^{-(K_2-\rho/4) t}. \label{eq:expocvtoboundary00}
		\end{eqnarray}
	\end{enumerate}
\end{theo}
\mofe{Theorem~\ref{thm:survival_probability} } quantifies the deviation of~$X$ from the behavior established in Theorem~\ref{thm:accessibility_condition}. In fact, \mofe{we will use Theorem~\ref{thm:survival_probability} to prove Theorem~\ref{thm:accessibility_condition} for the parameter regime $\Theta_0 \cup \Theta_1 \cup \Theta_2$ }(see Section~\ref{sec:accessibility}). The proof of Theorem~\ref{thm:survival_probability} is in Section~\ref{sec:essfix}. 

\subsection{Siegmund dual process}\label{sec:main:siegmund}

The proofs of Theorems \ref{thm:accessibility_condition} and \ref{thm:survival_probability} use properties of an auxiliary process that exhibits a duality relation to $X$. We now introduce that process and state the corresponding relation. Subsequently, we discuss the role of duality in proving Theorem~\ref{thm:survival_probability}.
Let $(\Lambda,\mu,\sel)\in\Theta$ and $Y=(Y_t)_{t\geq0}$ be the solution of the SDE
\begin{align}\label{eq:SDE_Y}
\dd Y_t  \,= &\int_{(0,1)^2} \big(\mr_{r,u}(Y_{t-}) - Y_{t-}\big)N(\dd t, \dd r, \dd u)\nonumber\\
		&-Y_t(1-Y_t)\, \sel(Y_t)\, \dd t \ + \int_{(-1,1)}(s_r(Y_{t-})-Y_{t-})\, S(\dd t, \dd r).
\end{align}
\mofe{where $N$ } and $S$ are as in SDE \eqref{eq:SDEWFP},
\begin{equation}
	\mr_{r,u}(y) =  \median{\frac{y-r}{1-r}, \frac{y}{1-r}, u} \ \ \text{and}\ \ s_r(y)= \frac{1+r-\sqrt{(1+r)^2-4r y} }{2r}.\label{eq:defmruandsr}
\end{equation}
We set $s_0(y)=y$ so that $r\in(-1,1)\mapsto s_r(y)$ is continuous. We establish existence of a strong pathwise unique solution to~\eqref{eq:SDE_Y} with initial condition $Y_0=y\in[0,1]$ in Appendix~\ref{sec:ExistenceUniqueness}. We also show that the solution is such that $Y_t\in[0,1]$ for all $t\geq 0$.

The duality relation connecting $X$ and~$Y$ is stated in the next theorem.
\begin{theo}[Siegmund duality]\label{thm:siegmund_duality}
	Let $(\Lambda,\mu,\sigma)\in \Theta$. Then, for $t\geq 0$, \begin{equation}
		\forall x,y\in [0,1],\qquad \P_x(X_t\geq y)=\P_y(x\geq Y_t).\label{eq:siemgund_duality_relation}
	\end{equation} 
\end{theo}
This duality holds in many cases where classical genealogical methods (i.e. based on the ASG) fail to provide information on $X$; for example, if $X$ has no moment dual and the Bernstein dual is transient. 
Theorem \ref{thm:siegmund_duality} will be proved in Section \ref{sec:proofsiegdual} in two steps; we first prove the duality in the absence of small and large jumps of the involved Poisson measures using an elementary approach, and then, we extend the result using a limiting procedure. 
\smallskip

Theorem \ref{thm:siegmund_duality} can be strengthen into the following corollary providing a duality-type result involving two trajectories of $X$ and one of $Y$ (and vice versa). 
\begin{cor}\label{2dual}
Let $(\Lambda,\mu,\sigma)\in \Theta$. For any $0\leq \hat{x}< \check{x}\leq 1$ (the notation aims at $\hat{x}=\hat{x}\wedge \check{x}\coloneqq\min\{\hat{x},\check{x}\}$ and $\check{x}=\hat{x}\vee \check{x}\coloneqq\max\{\hat{x},\check{x}\}$), $y\in [0,1]$ and $t\geq 0$,
$$\P_{\hat{x},\check{x}}(\widehat{X}_t<y\leq \check{X}_t)=\P_y(\hat{x}<Y_t \leq \check{x}),$$
where under $\P_{\hat{x},\check{x}}$, $\check{X}$ and $\widehat{X}$ are strong solutions of SDE \eqref{eq:SDEWFP} {in the same background }with $\widehat{X}_0=\hat{x}$ and $\check{X}_0=\check{x}$. Similarly, for any $0\leq \hat{y}< \check{y}\leq 1$, $x\in [0,1]$ and $t\geq 0$,
$$\P_{x}(\hat{y}\leq X_t<\check{y})=\P_{\hat{y},\check{y}}(\widehat{Y}_t\leq x< \widecheck{Y}_t),$$
where under $\P_{\hat{y},\check{y}}$, $\widecheck{Y}$ and $\widehat{Y}$ are strong solutions of SDE \eqref{eq:SDEWFP} {in the same background }with $\widehat{Y}_0=\hat{y}$ and $\widecheck{Y}_0=\check{y}$.
\end{cor}

In light of Theorem~\ref{thm:siegmund_duality}, a remark addressing the \mofe{appearance of a median } in~\eqref{eq:SDE_Y} is in order. Consider $x,y,u,r\in (0,1)$ and note that
\begin{equation}
	x+r(\1_{\{u\leq x\}}(1-x)-\textbf{1}_{\{u>x\}} x) \, \geq \, y \quad \Longleftrightarrow \quad x \, \geq \, \median{\frac{y-r}{1-r}, \frac{y}{1-r}, u},\label{eq:medianexplanation}
\end{equation}
which follows by distinguishing the cases $u\leq x$ and $u>x$. Thus, $m_{r,u}$ is the generalized inverse of the non-decreasing, right-continuous function $x \mapsto x+r(\textbf{1}_{\{u\leq x\}}(1-x)-\textbf{1}_{\{u>x\}} x)$ appearing in the {SDE~\eqref{eq:SDEWFP} at each jump of $N(\dd t, \dd r,\dd u)$}. Similarly, the function $s_r$ is the inverse of the increasing, continuous function $x\mapsto x+rx(1-x)$ \modifbis{(from~$[0,1]$ to~$[0,1]$)}. 

\modifbis{\begin{remark} \label{boundaryfinitetime}
A consequence of Theorem \ref{thm:siegmund_duality} is that for $(\Lambda,\mu,\sigma)\in \Theta$ and $x\in (0,1)$, the process $X$ under $\P_x$ cannot reach the boundaries $0$ and $1$ in finite time. Indeed, if it was the case for the boundary $1$, there would exist a $t>0$ such that $\P_x(X_t=1)>0$ but, by Theorem~\ref{thm:siegmund_duality}, this probability equals $\P_1(x\geq Y_t)$ and, by uniqueness for \eqref{eq:SDE_Y}, the later probability is null. The same argument applied to $\tilde X:=1-X$ shows the claim for the boundary $0$.  
\end{remark}}

\begin{remark}
In the neutral case, i.e. $\sigma=0$ and $\mu=0$, \citep{Bertoin2003,Bertoin2005,Bertoin2006} already established connections between~$X$ and~$Y$. There, $X$ corresponds to the $1$-point
motion of the {dual flow }of Bridges associated to the $\Lambda$-coalescent; see \citep[Sect. 5]{Bertoin2003}) and \citep[Thm. 2]{Bertoin2005}. Whereas $Y$ corresponds to the $1$-point motion of the flow of inverses; see \citep[Sect. 5]{Bertoin2005} (their function $\psi_{z,v}(\cdot)$ is our~$m_{z,v}(\cdot)$). More generally, what they call $p$-motion corresponds to $p$ copies of~$Y$ starting at~$p$ different (space) points and they show that it solves a martingale problem \citep[Thm. 5]{Bertoin2005}. 
\end{remark}
\begin{remark}[Role of~$C_0(\Lambda,\mu,\sel)$ and $C_1(\Lambda,\mu,\sel)$]The coefficients in the SDE \eqref{eq:SDE_Y} allow for convenient approximations of $Y$ by functions of L\'evy processes near the boundaries $0$ and $1$. The coefficients $C_0(\Lambda,\mu,\sel)$ and $C_1(\Lambda,\mu,\sel)$ arise as the limit rate of growth of those L\'evy processes; see Remark \ref{limrateofgrowth} for details. It should then not come as a surprise that the signs of $C_0(\Lambda,\mu,\sel)$ and $C_1(\Lambda,\mu,\sel)$ play a crucial role in the long-term behavior of $X$. Moreover, the structure of the SDE \eqref{eq:SDE_Y} allows to establish and exploit interesting renewal properties for $Y$ (see below, {and Section \ref{sec:accbdd}}) and coalescing properties for several trajectories of $Y$ in the same background ({see the discussion after Theorem~\ref{thm:survival_probability}, and }Section \ref{sec:accbdd}). This turns out to be very helpful for establishing our results in the case $(\Lambda,\mu,\sigma)\in \Theta_2$.
\end{remark}
The next result describes the long-term behavior of~$Y$. Let $J_N$ be the set of jump times of the Poisson measure $N$, and for $T\in J_N$, let $(T,R_T, U_T)$ be the corresponding jump.
\begin{theo}[Asymptotic behavior of $Y$]\label{prop:stationary_distribution_Y}
$\,$
\begin{enumerate}
\item[(1)] Let $(\Lambda,\mu,\sigma)\in \Theta_1$ (resp. $\Theta_0$)\mofe{ and $y\in(0,1)$. Assume } that $\rW_\gamma(r^{-2}\Lambda(\dd r))<\infty$ and $\rW_\gamma(\bar{\mu})<\infty$ (resp. $\rW_\gamma(\mu)<\infty$) for some $\gamma>0$. Then we have $\P_y$-almost surely $\lim_{t\to\infty} Y_t=0$ (resp. $1$).
\item[(2)] \mofe{Let $(\Lambda,\mu,\sigma)\in \Theta_2$ and $y \in (0,1)$. Then, under~$\P_y$, $Y$ is conservative, open-set recurrent, and $\P_y(Y_t\in~\cdot~)$ converges in total variation distance to $\pi_Y(\cdot)$ as $t$ goes to infinity, where $\pi_Y(\cdot)$ is the unique stationary distribution supported on $(0,1)$}.
	\item[(3)] Let $(\Lambda,\mu,\sel)\in\Theta_3$. For any $y \in (0,1)$ we have $\P_y$-almost surely that $Y_{\infty}\coloneqq \lim_{t\to\infty} Y_t$ exists and $Y_\infty\in \{0,1\}$.
\end{enumerate}
\end{theo}
With regard to Theorem~\ref{prop:stationary_distribution_Y}(2), note that $\delta_0$ and $\delta_1$ are also stationary distributions for $Y$, but they are not supported on $(0,1)$.
\smallskip

Let us now assume that $(\Lambda,\mu,\sigma)\in \Theta_2$. Fix $\kappa\in (0,\maxsupp{\Lambda})$ and $\eta > 0$ such that $\theta\coloneqq \frac{\kappa+\eta}{1+\eta} < \maxsupp{\Lambda}$. Define 
	\begin{equation}\textstyle \renewal \coloneqq \inf \left\{ T\in J_N: \  Y_{T-} \in \left [\frac{1-\kappa}{2},\frac{1+\kappa}{2} \right ],\ R_T\in (\theta,1) \text{ and }U_T\in\left [\frac{1-\eta}{2},\frac{1+\eta}{2}\right ]\right\}. \label{eq:defrenewaltime}\end{equation}
 The random time~$\renewal$ in~\eqref{eq:defrenewaltime} is a renewal time for~$Y$. Its definition might seem a little unwieldy; but the idea is simple: it is defined such that $Y_{\renewal}$ is uniformly distributed on $[(1-\eta)/2,(1+\eta)/2 ]$. Indeed, if $(\renewal,R_{\renewal},U_{\renewal})$ is the corresponding jump in~$N$, then $$Y_{\renewal}=\median{\frac{Y_{{\renewal}-}-R_{\renewal}}{1-R_{\renewal}},\frac{Y_{{\renewal}-} }{1-R_{\renewal}},U_{\renewal}}.$$ But $\theta$ in the definition of~$\renewal$ is chosen such that $(Y_{{\renewal}-}-R_{\renewal})/(1-R_{\renewal})<(1-\eta)/2$ and $Y_{{\renewal}-}/(1-R_{\renewal})>(1+\eta)/2$. Since $U_{\renewal}\in [(1-\eta)/2,(1+\eta)/2]$, it follows that $Y_{\renewal}=U_{\renewal}$, and it is not difficult to show that $U_{\renewal}$ is a uniform random variable in $[0,1]$ conditioned to be in $[(1-\eta)/2,(1+\eta)/2]$.
\begin{theo}[Representation of the stationary distribution]\label{prop:stationary_distribution_Y-II}
Let $(\Lambda,\mu,\sigma)\in \Theta_2$ and denote by $\cV$ the uniform distribution on $[(1-\eta)/2,(1+\eta)/2]$. Write $\mathbb{P}_{\cV}$ for the distribution of $Y$ under the starting law $\cV$; and $\E_{\cV}$ for the associated expectation. Then there exists $\zeta>0$ such that $\E_{\cV}[e^{\zeta \renewal}]<\infty$ and for any Borel set $A$, 
	\begin{eqnarray}
		\pi_Y(A) = \frac{1}{\E_{\cV}[\renewal]} \E_{\cV} \left [ \int_0^{\renewal} \1_{\{Y_s \in A\}} \dd s \right ]. \label{eq:represpiy}
	\end{eqnarray}
\end{theo}
In the proof of Theorem~\ref{thm:survival_probability}, we use that the long-time behavior of~$X$ is linked to that of its dual~$Y$. More precisely, possible fixation (resp.~possible extinction) is related to~$1$ (resp. $0$) being repelling for~$Y$. Thus, in~$\Theta_1$ (resp.~$\Theta_0$), we study the probability for $Y$ to not be close to $0$ (resp.~$1$). \modifbis{To this end, we sandwich $\log(1/Y)$ (resp. $\log(1/(1-Y))$) between two L\'evy processes near the boundary $0$ (resp. $1$)}. The polynomial and exponential decay in $\Theta_1$ (and $\Theta_0$) is related to \mofe{fine } properties of L\'evy processes. We conjecture that the polynomial upper bound for $\P_x(\lvert X_t-b\rvert >t^{-\rho})$ in Theorem~\ref{thm:survival_probability}(1)(i) is optimal in the sense that, under some additional conditions, we can expect polynomial lower bounds to hold as well; see Section~\ref{rem:conjpolynomial} for more details. This in turn suggests some optimality of the strong integrability conditions in Theorem~\ref{thm:survival_probability}(1)(ii) in order to get exponential convergence. Remarkably, in $\Theta_2$, the decay of the probability of not having essential fixation or extinction is always exponential, \modifbis{irrespective of the properties of $\Lambda$ near~$1$, and of $\mu$ near~$1$ and $-1$}. This is because in $\Theta_2$, $Y$ is recurrent and {the probability of not having essential fixation or extinction at time~$t$ } is related to the probability of two trajectories of $Y$ not being merged before~$t$; we show that the latter always decays exponentially. Obtaining estimates for the rate of convergence of~$X$ to its stationary distribution in $\Theta_3$, i.e. when the whole boundary is repelling, seems to require different techniques. We plan to study this \modifbis{in a future work}. 

\smallskip

\modifbis{The Siegmund duality leads to a representation of the fixation probability of~$X$}. 
This representation implies that simulating $(Y_s)_{{0 \leq s < \renewal}}$ under $\P_{\cV}$ suffices to estimate the fixation probability of~$X$ for all starting points $x$ simultaneously. 

\begin{theo}[Representation of absorption probability] \label{thm:representation_absorption_probability}
	Let $(\Lambda,\mu,\sel)\in\Theta_2$. For $x\in (0,1)$, $\P_x(\lim_{t\to\infty}X_t=1)=\pi_Y([0,x])\in (0,1).$
	In particular, 
	\begin{eqnarray}
		\mathbb{P}_x\big(\lim_{t\to\infty}X_t=1\big) = \frac{1}{\mathbb{E}_{\cV}[\renewal]} \mathbb{E}_{\cV} \left [ \int_0^{\renewal} \1_{\{Y_s \leq x\}} \dd s \right ]. \label{eq:represhx}
\end{eqnarray}\end{theo}
Theorems~\ref{prop:stationary_distribution_Y} and \ref{thm:representation_absorption_probability} are proved in Section~\ref{sec:behavy} and \ref{sec:accessibility}, respectively.

\subsection{Organization of the paper} \label{sec:orgpaper}

The rest of the paper is organized as follows. In Section~\ref{sec:stratproof} we state some intermediary results about $Y$ (Theorems~\ref{thm:stuckatboundary}, \ref{thm:probmergeYt}, and \ref{thm:coexistenceY}) and use them to prove all results of Section~\ref{sec:main} except those related to Siegmund duality (Theorem \ref{thm:siegmund_duality} and Corollary~\ref{2dual}). \modifbis{Sections \ref{sec:comparisonlevy} to \ref{sect:coex} are then dedicated to proving the intermediary results about~$Y$ from Section \ref{sec:stratproof}. More precisely, in }Section \ref{sec:comparisonlevy} we prove some key properties of $Y$, including the L{\'e}vy sandwiching property. In Section \ref{sec:accbdd} we focus on the case $(\Lambda,\mu,\sel)\in \Theta_2$ and prove Theorem \ref{thm:probmergeYt}. In Section \ref{sec:asextinction} we focus on the case $(\Lambda,\mu,\sel)\in \Theta_0\cup\Theta_1$ and prove Theorem \ref{thm:stuckatboundary}. In Section \ref{sect:coex} we focus on the case $(\Lambda,\mu,\sel)\in \Theta_3$ and prove Theorem \ref{thm:coexistenceY}. In Appendix~\ref{sec:ExistenceUniqueness} we establish the existence and uniqueness of~$X$ and $Y$. The Siegmund duality is proved in Appendix \ref{sec:proofsiegdual}. Appendix~\ref{sec:app:levy} contains some fluctuation-theory results for L{\'e}vy processes. Finally, Appendix~\ref{sec:TechEst} collects technical estimates we require throughout the manuscript.

\section{Proofs of main results} \label{sec:stratproof}
This section states results describing the long-term behavior of the dual process $Y$ and uses them, together with the Siegmund duality (Theorem \ref{thm:siegmund_duality} and Corollary~\ref{2dual}), to prove our main results. We start out in Section~\ref{sec:mainingredients} to collect results about long-term properties and decay rates of~$Y$ in the four different parameter regions. In Section~\ref{sec:behavy}, we prove Theorem~\ref{prop:stationary_distribution_Y}. Section~\ref{sec:essfix} contains the proof of Theorem~\ref{thm:survival_probability}. Finally Section~\ref{sec:accessibility} contains the proof of Theorems~\ref{thm:accessibility_condition} and \ref{thm:representation_absorption_probability}. In particular, by the end of Section~\ref{sec:stratproof} all main results, except the ones related to Siegmund duality, will have been proved using the results stated in Section~\ref{sec:mainingredients} and the Siegmund duality.

\subsection{Long-term properties and decay rates of the dual process} \label{sec:mainingredients}

The following theorems \mofe{summarize } the long-term \mofe{behavior } of~$Y$ in the four parameter settings $\Theta_i$, $i\in\{1,2,3,4\}$.
They are instrumental in proving our main results.

The following result provides the long-term behavior of $Y$ in \mofe{$\Theta_0$ and $\Theta_1$ } under weak and strong integrability assumptions on the measures $\Lambda$ and $\mu$.
\begin{theo} \label{thm:stuckatboundary} Let $b\in\{0,1\}$ and $(\Lambda,\mu,\sigma)\in \Theta_b$. Recall that $\mu_0\coloneqq \mu$ and $\mu_1\coloneqq \bar{\mu}$.
	\begin{enumerate} 
		\item[(i)] (polynomial decay) Suppose $\rW_\gamma(r^{-2}\Lambda(\dd r))<\infty$ and $\rW_\gamma(\mu_b)<\infty$ for some $\gamma>0$. Then we have $\P_y$-almost surely $\lim_{t\to\infty} Y_t=1-b$, and, for any $\alpha\in (0,\gamma)$ and $\varepsilon\in (0,1)$, there is $K=K(\alpha,\varepsilon)>0$ such that for any $y\in (0,1)$ and $t\geq 0$, \begin{eqnarray}
			\P_y(\lvert Y_t-(1-b)\rvert>\varepsilon)\leq (1-y)^{-1/4}Kt^{-\alpha}. \label{survYpoldecay}
		\end{eqnarray}
		\item[(ii)] (exponential decay) Suppose $s_\gamma(r^{-2}\Lambda(\dd r))<\infty$ and $s_\gamma(\mu_b)<\infty$ for some $\gamma>0$. Then, for any $\varepsilon\in (0,1)$, there are $K_1=K_1(\varepsilon)>0$, $K_2=K_2(\varepsilon)>0$ such that for any $y\in (0,1)$ and $t\geq 0$, \begin{eqnarray}
			\P_y(\lvert Y_t-(1-b)\rvert>\varepsilon)\leq (1-y)^{-1/4}K_1e^{-K_2t}. \label{survYexpodecay}
		\end{eqnarray}
	\end{enumerate}
\end{theo}

In the next theorem, we consider \mofe{parameter regime $\Theta_2$ } and study coalescing and renewal properties of the flow of $Y$, as well as its invariant measure. In particular, we provide an estimate of the decay rate of the probability that two coupled trajectories $\widehat{Y}$ and $\widecheck{Y}$ of~$Y$ (as defined in Corollary \ref{2dual}) have not merged before time $t$. Note that by the pathwise uniqueness and the strong Markov property (see Proposition \ref{fullgenerator}), two such solutions agree after the first time they intersect.

\begin{theo} \label{thm:probmergeYt} Assume $(\Lambda,\mu,\sel)\in \Theta_2$. 
	\begin{enumerate}
		\item[(i)] (exponential decay) Then there are constants $K_1,K_2>0$ such that for any $\hat{y},\check{y}\in (0,1)$ with $\hat{y}\leq \check{y}$ and $t\geq 0$, \begin{align*}
			\P_{\hat{y},\check{y}}(\widehat{Y}_t\neq \widecheck{Y}_t)\leq \big(\hat{y}^{-1/4}+(1-\check{y})^{-1/4}\big) K_1 e^{-K_2t}.
		\end{align*}
		\item[(ii)] \label{prop:boundrenewalexponential} The renewal time $\renewal$ (defined in~\eqref{eq:defrenewaltime}) admits exponential moments. More specifically, there exist positive constants $\zeta$ and $M$ such that for any $y\in (0,1)$ we have $$\E_y[e^{\zeta \renewal}]\leq (y^{-1/4}+(1-y)^{-1/4})M.$$
		\item[(iii)]\label{thm:propYstat} \mofe{The function
		$x\in[0,1]\mapsto\E_{\cV}\big[\int_0^{\renewal}\ind{Y_t\leq x} \dd t\big]/\E_{\cV}[\renewal]$ is bounded, continuous, strictly increasing, vanishes at $0$ and equals $1$ at $1$.}
	\end{enumerate}
\end{theo}

Finally, let us consider parameter regime $\Theta_3$.
\begin{theo} \label{thm:coexistenceY} Assume $(\Lambda,\mu,\sigma)\in \Theta_3$. For any $\rho\in (0,C_0(\Lambda,\mu,\sigma) \wedge C_1(\Lambda,\mu,\sigma))$ and $y\in (0,1)$ we have $\mathbb{P}_y$-almost surely that either a) $Y_t \in [0,e^{-\rho t}]$ for all large $t$, or b) $Y_t \in [1-e^{-\rho t},1]$ for all large $t$. Moreover, 
	
	\begin{enumerate}
		\item[(i)] (polynomial decay)  If for some $\gamma>0$, $\rW_\gamma(r^{-2}\Lambda(\dd r))<\infty,\ \rW_\gamma(\mu)<\infty, \ \text{and}\ \rW_\gamma(\bar{\mu})<\infty$, then for any {$\rho\in (0,C_0(\Lambda,\mu,\sigma) \wedge C_1(\Lambda,\mu,\sigma))$ }and $\alpha\in (0,\gamma)$ there is $K=K(\alpha,\rho)$ such that for any $y\in (0,1)$ and $t\geq 0$ we have \begin{equation}
			{\P_y(Y_t\in [e^{-\rho t},1-e^{-\rho t}])\leq Kt^{-\alpha}}.
		\end{equation}
		\item[(ii)] (exponential decay) If for some $\gamma>0$, $s_\gamma(r^{-2}\Lambda(\dd r))<\infty,\ s_\gamma(\mu)<\infty$, and $s_\gamma(\bar{\mu})<\infty$, then for any $\rho\in (0,C_0(\Lambda,\mu,\sigma) \wedge C_1(\Lambda,\mu,\sigma))$ there are positive constants $K_1=K_1(\rho), K_2=K_2(\rho)$ such that for any $y\in (0,1)$ and $t\geq 0$ we have \begin{equation}
			{\P_y(Y_t\in [e^{-\rho t},1-e^{-\rho t}])\leq K_1e^{-K_2t}}.
		\end{equation}
		
	\end{enumerate}
\end{theo}
\modifbis{Theorem \ref{thm:coexistenceY} is stronger than what is required for proving the results from Section~\ref{sec:main} in parameter region $\Theta_3$. But it opens the way to studying the rate of convergence of $X$ to its stationary distribution in the case $\Theta_3$, which will be the object of a future work}. 
Theorems~\ref{thm:stuckatboundary}, \ref{thm:probmergeYt}, and \ref{thm:coexistenceY} are proved in Sections~\ref{sec:asextinction}, \ref{sec:accbdd}, and~\ref{sect:coex}, respectively.

\subsection{Proofs of Theorems \ref{prop:stationary_distribution_Y} and \ref{prop:stationary_distribution_Y-II}: asymptotic behavior of $Y$} \label{sec:behavy}
In this section we assume the results from Section \ref{sec:mainingredients} hold true. 
While we have presented Theorems \ref{prop:stationary_distribution_Y} and \ref{prop:stationary_distribution_Y-II} separately for the sake of clarity, their proofs are in fact intertwined. Therefore, we will prove both results simultaneously.

\begin{proof}[Proof of Theorems \ref{prop:stationary_distribution_Y} and \ref{prop:stationary_distribution_Y-II}]
Theorem \ref{prop:stationary_distribution_Y}(1) is already contained in Theorem~\ref{thm:stuckatboundary}(i), and Theorem \ref{prop:stationary_distribution_Y}(3) is a direct consequence of Theorem~\ref{thm:coexistenceY}. 

We now prove Theorem \ref{prop:stationary_distribution_Y}(2) and Theorem \ref{prop:stationary_distribution_Y-II}. Note first that Theorem~\ref{thm:probmergeYt}(ii) implies that $\E_y[e^{\zeta \renewal}]<\infty$ and $\E_{\cV}[e^{\zeta \renewal}]<\infty$ for some $\zeta>0$. In particular, for all $y\in (0,1)$, $\E_{y}[\renewal]<\infty$, and $\E_{\cV}[\renewal]<\infty$. {It follows from this and the strong Markov property (see Proposition \ref{fullgenerator}) }that $Y$ under $\P_\cV$ (resp. under~$\P_y$) is a (resp. delayed) regenerative process, i.e. $(Y_{t+\renewal})_{t\geq 0}$ is independent of $((Y_t)_{t<\renewal} ,\renewal)$, $(Y_{t+\renewal})_{t\geq 0}$ is stochastically equivalent to $(Y_{t})_{t\geq 0}$, and the regeneration epoch is $\renewal$. By the Poissonian properties of $N$ and $S$, $\renewal$ is nonlattice. Therefore, since $Y$ is right-continuous, we can use~\citep[Chap. VI., Thm. 1.2]{asmussen2008applied} to deduce that the limiting distribution $\pi_Y$ of~$Y$ exists and satisfies for all measurable sets $A\subseteq[0,1]$, $$\pi_Y(A)=\frac{\E_{\cV}\big[\int_0^{\renewal}\ind{Y_t\in A} \dd t\big]}{\E_{\cV}[\renewal]}.$$
\mofe{According to Theorem~\ref{thm:probmergeYt}(iii), $\pi_Y$ is supported in $(0,1)$}. We now show the convergence of $\P_y(Y_t\in~\cdot~)$ to $\pi_Y(\cdot)$ for the total variation distance. Let $y,z \in (0,1)$ and~$Y,\tilde{Y}$ be solutions to~\eqref{eq:SDE_Y} in the same random background with~${Y}_0=y$ and~$\tilde{Y}_0=z$. If $y \in (0,1)$ and~$z$ is chosen according to~$\pi_Y(\cdot)$, then ${Y}_t \sim \P_y(Y_t\in~\cdot~)$ while $\tilde{Y}_t \sim \pi_Y(\cdot)$. We thus get that for any $y \in (0,1)$, 
\begin{eqnarray}
d_{TV}(\P_y(Y_t\in~\cdot~),\pi_Y(\cdot)) \leq \int_{(0,1)} \P_{y,z} ({Y}_t \neq \tilde{Y}_t) \pi_Y(\dd z). \label{dtvytpiy}
\end{eqnarray}
By Theorem~\ref{thm:probmergeYt}(i), for any $z \in (0,1)$, {$\P_{y,z} ({Y}_t \neq \tilde{Y}_t)$} converges to $0$ as $t$ goes to infinity. By dominated converge the right-hand side of \eqref{dtvytpiy} converges to $0$, yielding the convergence of $\P_y(Y_t\in~\cdot~)$ for the total variation distance. This implies that there are no other stationary distributions of $Y$ supported on $(0,1)$. \mofe{This already ends the proof of Theorem \ref{prop:stationary_distribution_Y-II}. For the proof of Theorem \ref{prop:stationary_distribution_Y} it only remains to prove that $Y$ is open set recurrent.}
	
For any sufficiently small interval centered around $1/2$, $\renewal$ can be set up in a way (by choosing~$\eta$ appropriately) so that $Y_{\renewal}$ is distributed according to $\cV$, the uniform law on that interval. By Theorem~\ref{thm:probmergeYt}(ii), $\E_{y}[\renewal]<\infty$ for any $y \in (0,1)$ and $\E_{\cV}[\renewal]<\infty$. Using this and the renewal property, we get that the interval is almost surely visited at arbitrary large times. However, the choice of $1/2$ as the center of the interval was arbitrary and can be replaced by any other point in $(0,1)$ as we explain in Remark~\ref{rem:modificationsrenewal}. In particular, $Y$ is open set recurrent, i.e. for any open set $A\in (0,1)$, almost surely there exists a sequence $(t_n)_{n\in \N}$ such that $t_n\nearrow\infty$ as $n\to\infty$ and $Y_{t_n}\in A$ for all $n\in \N$. 
\end{proof}
\subsection{Proof of Theorem \ref{thm:survival_probability}: essential fixation/extinction} \label{sec:essfix}
In this section we assume that the results related to Siegmund duality (Theorem \ref{thm:siegmund_duality} and Corollary~\ref{2dual}) and the results of Section \ref{sec:mainingredients} hold true. We now prove Theorem \ref{thm:survival_probability}, which contains the probability for essential fixation/extinction. 
\begin{proof}[Proof of Theorem~\ref{thm:survival_probability}]
Since  $(\Lambda,\mu,\sigma)\in \Theta_0$ implies that $(\Lambda,\bar{\mu},\bar{\sigma})\in\Theta_1$ with $\bar{\sigma}(y)=-\sigma(1-y)$, it suffices to prove part (1) for $(\Lambda,\mu,\sigma)\in \Theta_1$; the proof for $(\Lambda,\mu,\sigma)\in \Theta_0$ follows by applying the result in $\Theta_1$ to $1-X$, which is a $(\Lambda,\bar{\mu},\bar{\sigma})$-Wright--Fisher process.

Consider $(\Lambda,\mu,\sigma)\in \Theta_1$.	Let $\varepsilon\in (0,1)$ and $x\in (\varepsilon,1)$. We first prove (1)(i). Assume that $\rW_\gamma(r^{-2}\Lambda(\dd r))<\infty$ and $\rW_\gamma(\bar{\mu})<\infty$ for some $\gamma>0$. Let $\alpha\in (0,\gamma)$ and $K>0$ such that $ \P_y(Y_t>\varepsilon)\leq (1-y)^{-1/4}Kt^{-\alpha}$ for all $y\in (0,1)$ and $t\geq 0$; the existence of such $K$ is ensured by Theorem~\ref{thm:stuckatboundary}(i). Let $\rho\in (0,4\alpha)$. \mofe{Using the previous bound and Theorem~\ref{thm:siegmund_duality}}, we obtain $$\P_x(X_t<1-t^{-\rho} )=\P_{1-t^{-\rho}}(x<Y_t)\leq \P_{1-t^{-\rho}}(\varepsilon<Y_t)\leq Kt^{-(\alpha-\rho/4)}.$$
	This proves (1)(i). \mofe{The proof of (1)(ii) is analogous, but using Theorem \ref{thm:stuckatboundary}(ii) instead of Theorem \ref{thm:stuckatboundary}(i).}

We now consider $(\Lambda,\mu,\sigma)\in \Theta_2$. Using Corollary~\ref{2dual} we get for $\hat{y}<\check{y}$, $$\P_x(\hat{y}< X_t< \check{y})\leq \P_x(\hat{y}\leq X_t< \check{y})={\P_{\hat{y},\check{y}}(\widehat{Y}_t\leq x< \widecheck{Y}_t)\leq \P_{\hat{y},\check{y}}(\widehat{Y}_t\neq \widecheck{Y}_t)}.$$
	Choose $\hat{y}=e^{-\rho t}$ and $\check{y}=1-e^{-\rho t}$. Thus, by \modifbis{Theorem \ref{thm:probmergeYt}(i)}, $$\P_x(X_t\in(e^{-\rho t}, 1-e^{-\rho t}))\leq 2K_1e^{-(K_2-\rho/4)t},$$ which completes the proof (note that the $K_1$ from Theorem \ref{thm:survival_probability}(2) is taken as the $K_1$ from Theorem \ref{thm:probmergeYt}(i) multiplied by $2$).\end{proof}

\subsection{Proofs of Theorems \ref{thm:accessibility_condition} and  \ref{thm:representation_absorption_probability}: boundary classification and representation of fixation probability} \label{sec:accessibility}
Assuming the Siegmund duality (Theorem \ref{thm:siegmund_duality}) and Theorem~\ref{thm:propYstat}  hold true, 
and having already proved Theorems~\ref{thm:survival_probability}, \ref{prop:stationary_distribution_Y} and \ref{prop:stationary_distribution_Y-II}, 
we proceed to prove Theorems~\ref{thm:accessibility_condition} and~\ref{thm:representation_absorption_probability}. 
As the proof of Theorem \ref{thm:accessibility_condition}(0/1) relies on Theorem~\ref{thm:accessibility_condition}(2), we begin proving the latter. For convenience, we simultaneously prove Theorem~\ref{thm:representation_absorption_probability}.

\begin{proof}[Proofs of Theorems~\ref{thm:accessibility_condition}(2) and \mofe{\ref{thm:representation_absorption_probability}}]
Assume $(\Lambda,\mu,\sigma)\in \Theta_2$. Consider $f(x)\coloneqq \pi_Y([0,x])$. Recall from \mofe{Theorem~\ref{prop:stationary_distribution_Y-II} and Theorem~\ref{thm:propYstat}(iii) } that~$f$ is well-defined, bounded, continuous, strictly increasing and satisfies $f(0)=0$ and $f(1)=1$. In particular, $f^{-1}$ exists and is continuous. \mofe{According to~\cite[Thm.~4.1]{Foucart2022} (with $\theta=0$ and $\mu_{\theta}=\pi_Y$) } $(f(X_t))_{t\geq 0}$ is a bounded martingale. Hence, by Doob's martingale convergence theorem, $\lim_{t\to\infty} f(X_t)$ exists almost surely. Since $f^{-1}$ is continuous, $X_{\infty}\coloneqq \lim_{t\to\infty} X_t$ exists almost surely. From Theorem~\ref{thm:survival_probability}, it follows that $X_{\infty}\in \{0,1\}$. Thus, $$\E_x[f(X_{\infty})]=f(0)\P_x(X_{\infty}=0)+f(1)\P_x(X_{\infty}=1)=\P_x(X_{\infty}=1).$$ By the martingale property of $(f(X_t))_{t\geq0}$, {$\E_x[f(X_{\infty})]=f(x)$}. Thus, since $f(0)=0$, $f(1)=1$ and $f$ is strictly increasing, \mofe{$\P_x(X_{\infty}=1)=f(x)\in(0,1)$. This ends the proof of Theorem~\ref{thm:accessibility_condition}(2). Theorem \ref{thm:representation_absorption_probability} follows using the previous identity and Theorem \ref{prop:stationary_distribution_Y-II}}.
\end{proof}
\mofe{Having proved Theorem~\ref{thm:accessibility_condition} part (2), we now proceed to prove part (0), (1), and (3).}
\begin{proof}[Proofs of Theorem~\ref{thm:accessibility_condition}(0/1/3)]	
\modifbis{Let $(\Lambda,\mu,\sigma)\in \Theta_0$. By assumption there is~$\gamma>0$ such that $w_\gamma(r^{-2}\Lambda(\dd r))<\infty$ and $w_\gamma(\mu)<\infty$. 
By Theorem~\ref{thm:survival_probability}--(1)(i), for any starting point $x \in (0,1)$, $X_t$ converges in probability to $0$}. To deduce the almost sure convergence, choose $\tilde{\sigma}\in \Cs^1([0,1])$ such that for all $z\in [0,1]$, $\tilde{\sigma}(z) \geq \sigma(z)$ and $(\Lambda,\mu,\tilde{\sigma})\in \Theta_2$. Write $\tilde{X}$ for the $(\Lambda,\mu,\tilde{\sigma})$-Wright--Fisher process with $\tilde{X}_0=z$. 
\modifbis{By Theorems \ref{thm:representation_absorption_probability} and \ref{thm:probmergeYt}(iii), we infer that $\P_z(\lim_{t\to\infty} \tilde{X}_t = 0)\to 1$ as $z\to 0$}. By~\citep[Thm.2.2]{DL12}, almost surely for all~$t\geq 0$, $X_t \leq \tilde{X}_t$. \mofe{Thus $\P_z(\lim_{t\to\infty} X_t = 0)\geq\P_z(\lim_{t\to\infty} \tilde{X}_t= 0)$. We conclude that $\P_z(\lim_{t\to\infty}X_t =0)\to 1$ as $z\to 0$. Hence, for $\varepsilon > 0$, } there is $z_0 > 0$ such that for any $z \in (0,z_0]$, $\P_z(\{\lim_{t\to\infty}X_t =0\}^c) \leq \varepsilon/2$. Moreover, since $X_t$ converges in probability to $0$ under $\P_x$, there is $t_0 >0$ such that $\P_x(X_{t_0} > z_0) \leq \varepsilon/2$. The Markov property of $X$ at $t_0$ and the above estimates yield
	\[\P_x(\{\lim_{t\to\infty}X_t =0\}^c) \leq \P_x(X_{t_0} > z_0) + {\E_x[\ind{X_{t_0} \leq z_0} \P_{x}(\{\lim_{t\to\infty} X_{t+t_0} =0\}^c \mid \mathcal{F}_{t_0} ) ]} \leq   \varepsilon. \]
	Therefore, $\P_x(\lim_{t\to\infty} X_t=0) \geq 1-\varepsilon$. \mofe{Letting } $\varepsilon\to0$ proves that {$X_t$ converges }almost surely to~$0$. \mofe{The proof for $(\Lambda,\mu,\sigma)\in \Theta_1$ is analogous}.


For $(\Lambda,\mu,\sigma)\in \Theta_3$, let $Y$ be the solution of~\eqref{eq:SDE_Y} with $Y_0=y\in (0,1)$. By Theorem \ref{prop:stationary_distribution_Y}(3), $Y_{\infty}$ exists and is either~$0$ or~$1$ almost surely. Thus, for all $x\in (0,1)$, $$\P_x(X_t \geq y)=\P_y(x \geq Y_t)\xrightarrow{t\to\infty} \P_y(Y_{ \infty}=0).$$ In particular, the distribution function of $X_t$ under $\P_x$ converges to a limit. Since $(X_t)_{t>0}$ is a tight family of random variables (because they are supported in $[0,1]$), we deduce that $X_t$ converges to a limit distribution $\pi_X$ that satisfies $\pi_X([y,1])=\P_y(Y_{ \infty}=0)$.
	It follows from the SDE~\eqref{eq:SDE_Y} that if $y=1$, then $Y_t=1$ for all $t\geq0$. Hence, $\pi_X(\{1\})=\P_1(Y_{ \infty}=0)=0$ so $1$ is not attractive for $X$ (if it was, we would have $\pi_X(\{1\})>0$). An analogous argument yields that $\pi_X(\{0\})=0$ and that $0$ is not attractive for $X$. 
\end{proof}

All results from Section~\ref{sec:main} are thus proved, assuming the Siegmund duality (Theorem \ref{thm:siegmund_duality} and Corollary~\ref{2dual}) and the results from Section \ref{sec:mainingredients} hold true. In the remainder of the manuscript, we prove the results related to Siegmund duality and those stated in Section~\ref{sec:mainingredients}.

\section{Key properties of \texorpdfstring{$Y$}{Y} and a useful estimate on the tail of \modifbis{some random times}}\label{sec:comparisonlevy}

It remains to prove the results of Section~\ref{sec:mainingredients}. An important tool is a comparison principle near the boundaries of the dual~$Y$ with functions of L{\'e}vy processes. This will allow us to control~$Y$ near the boundaries and derive useful estimates.

In Section~\ref{sec:comparisonlevy:sub:levysandwich}, we \mofe{introduce } the corresponding L{\'e}vy processes and provide the comparison property we will use throughout. \modifbis{In Sections \ref{sec:comparisonlevy:sub:escape} and \ref{sec:asextinction:sub:prepext}, we study those L\'evy processes and derive an estimate for the time it takes $Y$ to exit a boundary strip. In Section~\ref{sec:keyproperties:deviation}, we derive an estimate for the deviation of $Y$ from its initial value in compact time intervals. In Section~\ref{sec:keyproperties:rigthdistributiontail}, we provide conditions for the exponential/polynomial decay of the tail of \modifbis{some random times}}.

\modifbis{Recall that we assume throughout that $(\Lambda,\mu,\sigma)\in \Theta$, in particular, $\int_{(0,1)}r^{-1}\Lambda(\dd r)<\infty$ and $\int_{(-1,1)} |r| \mu(\dd r)<\infty$}. 
\subsection{A L{\'e}vy sandwich}\label{sec:comparisonlevy:sub:levysandwich}
We bound $\log(1/Y)$ from above and below by L\'evy processes when $Y$ is small (analogous estimates follow for $\log(1/(1-Y))$ when $Y$ is close to $1$). This can be thought of as a local and approximated version of Lamperti transform. Indeed, applying $\log(1/\cdot)$ to $Y$ yields an SDE in which all terms are almost state-independent near infinity, leading to L\'evy process approximations for $\log(1/Y)$ when $Y$ is small. Then, using monotonicity with respect to the coefficients of the SDE we obtain a L\'evy process that serves as a lower bound and another that serves as an upper bound, after augmenting a correction term. 

Note from $(\Lambda,\mu,\sigma)\in \Theta$ that the $r$-components of the jumps of $S$ and $N$ are summable on finite time intervals. Using this and It\^{o}'s formula~\citep[Thm.~4.4.10]{Applebaum2009}, we find that
\begin{align}
	\log(1/Y_t)=&\log(1/Y_0)+\int_{[0,t]\times(0,1)^2}\!\!\!\!\!\!\!\!\!\! \log(\mr_{r,u}(Y_{s-})^{-1} Y_{s-})N(\dd s, \dd r, \dd u) \nonumber \\
	&\quad +\int_{[0,t]\times(-1,1)}\!\!\!\!\! \log(s_r(Y_{s-})^{-1} Y_{s-} )S(\dd s, \dd r)+\int_{[0,t]}\sigma(Y_s)(1-Y_s)\dd s\label{eq:loglevy}.
\end{align}
For $b\geq \log(2)$, $\delta\in (0,1)$, define the L\'evy processes $\hat{L}^b=(\hat{L}_t^b)_{t\geq 0}$ and $\check{L}^{b,\delta}\coloneqq (\check{L}^{b,\delta})_{t\geq 0}$ via
\begin{align}
\hat{L}_t^b & \coloneqq \!\int\limits_{[0,t]\times(0,1)^2}\!\!\!\!\!\!\!\!\!\!\log(1-r)N(\dd s,\dd r,\dd u) + \frac{1}{2}\int\limits_{[0,t]\times(-1,1)}\!\!\!\!\!\!\!\!\!\!\log( (1+r)^2-4re^{-b}\ind{r>0} ) S(\dd s, \dd r)\nonumber \\
	& \qquad +t\big(\sigma(0)-e^{-b}\lVert \sigma\rVert_{\Cs^1([0,1])}\big)\label{eq:levylower},  	\\
\check{L}_t^{b,\delta}&\coloneqq \modifbis{\!\!\!\!\int\limits_{[0,t]\times[\delta,1)\times (0,1)}\!\!\!\!\!\!\!\!\!\!\!\!\!\!\!\! \log\Big(\!(1-r)\vee\frac{1}{ue^{b}}\!\Big) N(\dd s,\dd r,\dd u)} \nonumber \\
\textstyle	& \quad +\frac{1}{2} \int\limits_{[0,t]\times(-1,1)}\!\!\!\!\!\!\!\!\!\!\log( (1+r)^2-4re^{-b}\ind{r<0} ) S(\dd s, \dd r) + t\big(\sigma(0)+e^{-b}\lVert \sigma\rVert_{\Cs^1([0,1])} \big). \label{eq:levyupper}
\end{align}
See also Fig.~\ref{fig:levysandwich} for a sketch of the involved processes. 
Note that $\hat{L}^b$ and $\check{L}^{b,\delta}$ are defined from $N$ and $S$; in particular, they live on the same probability space as $Y$. When the L\'evy processes are considered along with a trajectory of $Y$ starting at $y$, we write $\P_y$ for the corresponding probability measure, and $\P$ if they are considered alone. Let us denote by $\hat{\nu}_b$ and $\check{\nu}_{b,\delta}$ the intensity measures associated with the jumps of $\hat{L}^{b}$ and $\check{L}^{b,\delta}$ respectively. The next lemma \mofe{shows } that the L\'evy processes are indeed well-defined.

\begin{figure}[t]
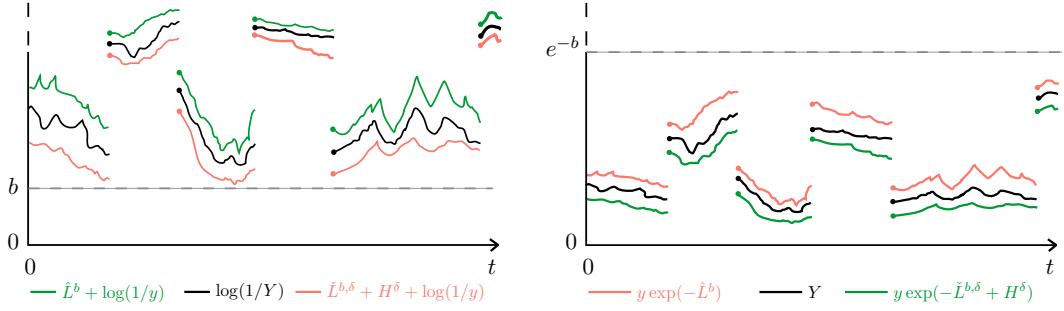

	\begin{minipage}{0.45\textwidth}
		\scalebox{3.5}{
}
	\end{minipage}
	\caption{Sketch of the L{\'e}vy sandwich in terms of $\log(1/Y)$ (left) and $Y$ (right).}\label{fig:levysandwich}
\end{figure}

\begin{lemme} \label{levywelldef}
	Let $(\Lambda,\mu,\sigma)\in \Theta$. For $b\geq \log(2)$ and $\delta\in (0,1)$, we have 
\begin{align}
\modifbis{\int_{\mathbb{R}} (1\wedge|x|)\hat{\nu}_b(\dd x)<\infty\quad\textrm{and}\quad\int_{\mathbb{R}} (1\wedge|x|)\check{\nu}_{b,\delta}(\dd x)<\infty}. \label{levymeascond}
\end{align}
Moreover, $\hat{L}^b$ and $\check{L}^{b,\delta}$ are well-defined and are L\'evy processes.
\end{lemme}
\begin{proof}
\mofe{Note that, for $r\in(-1/4,1/4)$, we have 
\begin{align*}
|\log((1+r)^2-4re^{-b}\ind{r>0})| &\leq \modifbis{|\log(1-3|r|)|},\\
|\log((1+r)^2-4re^{-b}\ind{r<0})| &\leq |\log(1-3|r|)|.
\end{align*}
Therefore,} 
\begin{align*}
\int_{\mathbb{R}} (1\wedge|x|)\hat{\nu}_b(\dd x) & \leq \int_{(0,1/2]}|\log(1-r)|\frac{\Lambda(\dd r)}{r^2} + \int_{(1/2,1)}\frac{\Lambda(\dd r)}{r^2} \\
&+\frac{1}{2} \left (\int_{(-\frac14,\frac14)} |\log(1-3|r|)| \mu(\dd r) + \mu\Big(\Big(-\frac14,\frac14\Big)^c\Big)\right ), \\
\int_{\mathbb{R}} (1\wedge|x|)\check{\nu}_{b,\delta}(\dd x) & \leq \int_{[\delta,1)}\frac{\Lambda(\dd r)}{r^2} +\frac{1}{2} \left (\int_{(-\frac14,\frac14)}\!\!\!\!|\log(1-3|r|)| \mu(\dd r) + \mu\Big(\Big(-\frac14,\frac14\Big)^c\Big)\right ). 
\end{align*}
By $(\Lambda,\mu,\sigma)\in \Theta$, the integrals in the above upper bounds are finite, \modifbis{which proves \eqref{levymeascond}}. This entails that $\hat{L}^b_t$ and $\check{L}^{b,\delta}_t$ are well-defined and \cadlag in $t$ (see for example \citep[proof of Thm.~19.3]{KenIti1999}). The independence and stationarity of increments of $\hat{L}^b$ and $\check{L}^{b,\delta}$ are easy consequences of properties of the Poisson random measures $N(\dd s, \dd r, \dd u)$ and $S(\dd s, \dd r)$. The processes $\hat{L}^b$ and $\check{L}^{b,\delta}$ are thus indeed L\'evy processes.
\end{proof}
Given $Y$, the process of small jumps of~$\log(1/Y)$ induced by~$N$ is defined as
\begin{align*}
	H_t^{\delta}&\coloneqq \int_{[0,t]\times(0,\delta)\times (0,1)}\log\Big(\frac{Y_{s-}}{m_{r,u}(Y_{s-})}\Big)N(\dd s,\dd r,\dd u). 
\end{align*}
Recall that for $A\subset[0,1]$, $T_{Y}A= \inf\{t\geq 0: Y_t\in A\}$. 

\modifbis{The following lemma establishes the announced comparison principle between $\log(1/Y)$ and the two L\'evy processes. Applying the result as it stands is not straightforward due to the term $H_t^{\delta}$. However, in Section \ref{sec:comparisonlevy:sub:escape}, we will establish an estimate to control its effect.}
\begin{lemme}[Sandwich lemma]\label{lem:levysandwich}
	Let $(\Lambda,\mu,\sel)\in\Theta$. Fix $b\geq \log(2)$, $\delta\in (0,1)$ and $Y_0\in (0,e^{-b}]$. 
	Almost surely, for all $t\in [0,T_{Y}(e^{-b},1]]$, \begin{equation}
		\hat{L}_t^b \leq \log(1/Y_t)-\log(1/Y_0)\leq \check{L}_t^{b,\delta}+H_t^\delta.\label{eq:comparison}
	\end{equation}
\end{lemme}
\begin{proof}
	Note first that for $t\in [0,T_{Y}(e^{-b},1]]$, \modifbis{$Y_s \vee Y_{s-}\leq e^{-b}$ for all $s \in [0,t]$}. We first prove the \modifbis{lower bound for $\log(Y_0/Y_t)$}. To this end, \mofe{we lower bound the three integrands appearing in}~\eqref{eq:loglevy}, which will turn out to be the corresponding integrands in the definition of $\hat{L}^b$. \mofe{Clearly}, for all $(r,u)\in (0,1)^2$ and $y\in [0,1]$, $\mr_{r,u}(y)\leq y/(1-r)$ so that \modifbis{$\log(\mr_{r,u}(Y_{s-})^{-1}Y_{s-})\geq \log(1-r)$}, \mofe{which yields the desired bound for the first integrand}. For the second integrand, note that for $r\in(-1,1)$, \begin{align*}
		\modifbis{Y_{s-}^{-1}s_r(Y_{s-}) =\frac{Y_{s-}^{-1}}{2\lvert r\rvert }\int_{0\wedge 4rY_{s-}}^{0\vee 4rY_{s-}}\frac{1}{2\sqrt{(1+r)^2-v} } \dd v \ \leq \big((1+r)^2-\ind{r>0}4r e^{-b}\big)^{-1/2}},
	\end{align*}
	where we used that \modifbis{$Y_{s-}\leq e^{-b}$}. Taking the inverse and applying $\log$ yields \[\modifbis{\log(s_r(Y_{s-})^{-1}Y_{s-})}\geq \frac{1}{2}\log\left(\smash[b]{(1+r)^2-\ind{r>0}4r e^{-b}}\right).\]
	The third integrand in~\eqref{eq:loglevy} \modifbis{is bounded using $Y_{s}\leq e^{-b}$ and a Taylor expansion around~$0$. This yields $\sigma(Y_s)(1-Y_s)\geq \sigma(0)-e^{-b}\lVert \sigma\rVert_{\Cs^1([0,1])}$}. 
	
	\smallskip
	
	Next, we turn to \modifbis{the upper bound for $\log(Y_0/Y_t)$}. \mofe{For this, we upper bound the } three integrands in~\eqref{eq:loglevy}. For the first integrand in~\eqref{eq:loglevy},  consider \modifbis{$(s,r,u)$ with $r\in [\delta,1)$ and $s \leq t \leq T_{Y}(e^{-b},1]$, and note that $m_{r,u}(y)\geq u\wedge (y/(1-r))$ and $Y_s \vee Y_{s-}\leq e^{-b}$}. In particular, \[\modifbis{\frac{Y_{s-}}{m_{r,u}(Y_{s-})}\leq \frac{ 1-r }{1\wedge (Y_{s-}^{-1}u(1-r))}}\leq \frac{ 1-r }{1\wedge (e^bu(1-r))}=(1-r)\vee \frac{1}{ue^b}.\]
\modifbis{Taking the $\log$ in this inequality yields the corresponding term of $\check{L}_t^{b,\delta}$}. 
For the second integrand in~\eqref{eq:loglevy}, proceed as for \mofe{the lower bound}, but \mofe{using this time the bound } $\modifbis{Y_{s-}^{-1} s_r(Y_{s-}) \geq
	(\smash[b]{(1+r)^2-\ind{r<0}4r e^{-b}})^{-1/2}}$. Taking the inverse and applying $\log$ yields the second integrand of $\check{L}_t^{b,\delta}$. The third integrand in~\eqref{eq:loglevy} is upper bounded using a Taylor approximation as in the lower bound.
\end{proof}
\subsection{Properties of the approximating L\'evy processes, and time for $Y$ to escape a strip around boundary}\label{sec:comparisonlevy:sub:escape}
In this section, we provide explicit estimates for the exit time of~$Y$ from a strip around the boundary. The corresponding result is presented below. 
\begin{prop}\label{lem:escapetimeY}
	Let $(\Lambda,\mu,\sigma)\in\Theta$. 
	\begin{enumerate}
		\item If $C_0(\Lambda,\mu,\sigma)<0$, there exists $a\in (0,1/2)$ and $\gamma>0$ such that for any $y\in (0,1)$, 
		\[\E_y\left[e^{\gamma T_{Y}(a,1]} \right]\leq 2\left(\frac{a}{y}\vee 1\right)^{1/4}<\infty.\] In particular, $T_{Y}(a,1]<\infty$ $\P_y$-a.s..
		\item If $C_1(\Lambda,\mu,\sigma)<0$, there exists $a\in (1/2,1)$ and $\gamma>0$ such that for any $y\in (0,1)$, \[\E_y[e^{\gamma T_{Y}[0,a)} ]\leq 2\left(\frac{1-a}{1-y}\vee 1\right)^{1/4}<\infty.\] 
		In particular, $T_{Y}[0,a)<\infty$ $\P_y$-a.s..
	\end{enumerate}
\end{prop}
 \modifbis{The proof of Proposition~\ref{lem:escapetimeY} requires some preparation and takes the rest of this subsection; it relies on our L\'evy sandwich. To use it, we first estimate the Laplace transforms of $\hat{L}_1^{b}$ and $\check{L}_1^{b,\delta}$ and then the accumulation of small jumps of $\log(1/Y)$ induced by~$N$}.

\smallskip 

Note that the support of $\hat{\nu}_b$ is bounded from above by $\log(2)$ and that 
\begin{align*}
\int_{(1,\infty)} e^{\lambda x}\check{\nu}_{b,\delta}(\dd x)\leq \int_{[\delta,1)}r^{-2}\Lambda(\dd r)  \int_{(0,1)}e^{\lambda\log(1/u)}\dd u, \end{align*}
where the right-hand side is finite if $\lambda<1$. One can easily see that $\int_{(1,\infty)} e^{\lambda x}\check{\nu}_{b,\delta}(\dd x)<\infty$ is actually equivalent to $\lambda<1$. By \citep[Thm.~25.3]{KenIti1999} we get that the Laplace transforms $\E[e^{\lambda\hat{L}_1^{b}}]$ and $\E[e^{\lambda\check{L}_1^{b,\delta}}]$ are finite for $\lambda \in [0,\infty)$ and $\lambda \in [0,1)$, \mofe{respectively}. Denote the Laplace exponent of $\hat{L}^{b}$ and $\check{L}^{b,\delta}$ by $\hat{\psi}_{b}(\lambda)\coloneqq\log(\E[e^{\lambda\hat{L}_1^{b}}])$ and $\check{\psi}_{b,\delta}(\lambda)\coloneqq\log(\E[e^{\lambda\check{L}_1^{b,\delta}}])$, \mofe{respectively}. By Lemma \ref{levywelldef} and the L{\'e}vy-Kintchine formula we obtain 
\modifbis{\begin{align*}
\hat{\psi}_{b}(\lambda)&=\int_{\mathbb{R}} (e^{\lambda x}-1)\hat{\nu}_b(\dd x)+\lambda (\sigma(0)-e^{-b} \lVert \sigma\rVert_{\Cs^1([0,1])}),\quad \lambda \in [0,\infty),\\ 
\check{\psi}_{b,\delta}(\lambda)&=\int_{\mathbb{R}} (e^{\lambda x}-1)\check{\nu}_{b,\delta}(\dd x)+\lambda (\sigma(0)+e^{-b}\lVert \sigma\rVert_{\Cs^1([0,1])}),\quad \lambda\in [0,1).
\end{align*}}
This translates into
\begin{eqnarray}
	&\hat{\psi}_{b}(\lambda)=\int_{(-1,1)} \Big( \big( (1+r)^2-4re^{-b}\ind{r>0}\big)^{\lambda/2}-1\Big)\mu(\dd r)+ \lambda \left(\sigma(0)-e^{-b}\lVert \sigma\rVert_{\Cs^1([0,1])} \right)\nonumber \\
	&+\int_{(0,1)}( (1-r)^{\lambda}-1 )r^{-2}\Lambda(\dd r), \label{eq:laplacelower}\\
	&\check{\psi}_{b,\delta}(\lambda)= \int_{(-1,1)} \Big( \big( (1+r)^2-4re^{-b}\ind{r<0}\big)^{\lambda/2}-1\Big)\mu(\dd r) +\lambda \left(\sigma(0)+e^{-b}\lVert \sigma\rVert_{\Cs^1([0,1])} \right)\nonumber \\
	& \quad\quad+\int_{[\delta,1)}\left\{\ind{r< 1-e^{-b}}\left((1-r)^{\lambda}-1 +\frac{\lambda e^{-b}}{(1-\lambda)(1-r)^{1-\lambda}}\right)+ \ind{r\geq 1-e^{-b}} \left ( \frac{e^{-\lambda b}}{1-\lambda}-1 \right ) \right\}\frac{\Lambda(\dd r)}{r^2} \label{eq:laplaceupper}.
\end{eqnarray}
For later reference we note that, if $s_\gamma(r^{-2}\Lambda(\dd r))<\infty$ and $s_\gamma(\bar{\mu})<\infty$ for some $\gamma>0$ then $\int_{(-\infty,-1)} e^{\lambda x}\hat{\nu}_b(\dd x)<\infty$ for all $\lambda\in [-\gamma,\infty)$. Combining with \citep[Thm.~25.3]{KenIti1999}, Lemma \ref{levywelldef} and the L{\'e}vy-Kintchine formula we get 
\begin{lemme} \label{eq:remark_stronglaplace}
If $s_\gamma(r^{-2}\Lambda(\dd r))<\infty$ and $s_\gamma(\bar{\mu})<\infty$ for some $\gamma>0$ then $\E[e^{\lambda\hat{L}_1^{b}}]<\infty$ for all $\lambda\in [-\gamma,\infty)$ and, on this interval, $\hat{\psi}_{b}(\lambda)\coloneqq\log(\E[e^{\lambda\hat{L}_1^{b}}])$ has the expression \eqref{eq:laplacelower}.  
\end{lemme}
The following estimate involving~$\check{\psi}_{b,\delta}$ will be useful in the proof of Proposition~\ref{lem:escapetimeY}.
\begin{lemme}\label{lem:controlLaplaceofL}
	Let $(\Lambda,\mu,\sigma)\in\Theta$. If $C_0(\Lambda,\mu,\sigma)<0$, then there exists $\delta\in (0,1/2)$, $b\geq \log(2)$ and $\lambda\in (0,1)$ such that $\check{\psi}_{b,\delta}(\lambda)+2\lambda\int_{(0,\delta)}r^{-1}\Lambda( \dd r)<0.$
\end{lemme}
\begin{proof}
Set $p_{b,\delta}(\lambda) \coloneqq \check{\psi}_{b,\delta}(\lambda)+2\lambda\int_{(0,\delta)}r^{-1}\Lambda( \dd r)$. Note that $p_{b,\delta}(0)=0$. Hence, it suffices to show that $p'_{b,\delta}(0)<0$ for some $b\geq \log(2)$ and $\delta\in (0,1/2)$; the result would then follow. 
Differentiating \eqref{eq:laplaceupper} we get that
\begin{align}
&	p_{b,\delta}'(0)=	\int_{[\delta,1)} \left\{ \ind{r< 1-e^{-b}}\left( \log(1-r)+\frac{e^{-b}}{1-r} \right)+\ind{r\geq 1-e^{-b}} \modifbis{(1-b)}\right\} \frac{\Lambda(\dd r)}{r^2} \label{exprderp} \\
		&+\frac{1}{2}\int_{(-1,1)} \log\left((1+r)^2-4re^{-b}\ind{r<0}\right)\mu(\dd r) +  \sigma(0)+e^{-b}\lVert \sigma\rVert_{\Cs^1([0,1])}+ 2\int_{(0,\delta)}\!\!\!\frac{\Lambda( \dd r)}{r}. \nonumber
\end{align}
We first assume that $C_0(\Lambda,\mu,\sigma)\in(-\infty,0)$, which, combined with $(\Lambda,\mu,\sigma)\in \Theta$, implies
\begin{align} 
\int_{(0,1)}|\log(1-r)|r^{-2}\Lambda(\dd r)<\infty\quad\textrm{and}\quad\int_{(-1,0)}|\log(1+r)|\mu(\dd r)<\infty. \label{csqc0finite}
\end{align}
Then $p_{b,\delta}'(0)$ equals
\begin{align*}
		&C_0(\Lambda,\mu,\sigma) +e^{-b}\lVert \sigma\rVert_{\Cs^1([0,1])}+\frac{1}{2}\int_{(-1,0)}\log\left(1-\frac{4re^{-b}}{(1+r)^2}\right)\mu(\dd r)+ 2\int_{(0,\delta)}\frac{\Lambda( \dd r)}{r} \\
		& -\!\!\int_{(0,1)}\!\!\left[\ind{r<\delta} \log(1-r)-\ind{r\in[\delta,1-e^{-b})}\frac{e^{-b}}{1-r}+\ind{r\geq 1-e^{-b}}\left(\log(1-r)-\modifbis{(1-b)}\right) \right]\frac{\Lambda(\dd r)}{r^2}.
	\end{align*}
Thanks to \eqref{csqc0finite} and $(\Lambda,\mu,\sigma)\in \Theta$, all the involved integrals are finite. Also, since $(\Lambda,\mu,\sigma)\in \Theta$, we can choose $\delta\in (0,1/2)$ such that $\int_{(0,\delta)}(2r^{-1}-\log(1-r)r^{-2})\Lambda(\dd r)< \lvert C_0(\Lambda,\mu,\sigma)\rvert/3.$ For this choice of $\delta$, \mofe{by dominated convergence theorem}, we also have 
\begin{align}
\int_{[\delta,1-e^{-b})}\frac{e^{-b}}{(1-r)}\frac{\Lambda(\dd r)}{r^2}\xrightarrow[b\to\infty]{} 0. \label{termcvd}
\end{align}
Hence, for $b\geq\log(2)$ sufficiently large, $\int_{[\delta,1-e^{-b})}e^{-b}(1-r)^{-1}r^{-2}\Lambda(\dd r)< \lvert C_0(\Lambda,\mu,\sigma)\rvert/3.$
	Moreover, \modifbis{choosing a larger $b$ if necessary}, we get \begin{align*}
		e^{-b}\lVert \sigma\rVert_{\Cs^1([0,1])}+\int_{(-1,0)}\frac{1}{2}\log\left(1-\frac{4re^{-b}}{(1+r)^2}\right)\mu(\dd r)-&\int_{[1-e^{-b},1)} \frac{\log(1-r)-\modifbis{(1-b)}}{r^2}\Lambda(\dd r)\\
		&{<} \lvert C_0(\Lambda,\mu,\sigma)\rvert/3.
	\end{align*}
	Altogether, \mofe{for this } choice of $b$ and $\delta$, \mofe{we have } $p_{b,\delta}'(0)<0$. 

Assume now that $C_0(\Lambda,\mu,\sigma)=-\infty$ and fix $\delta\in (0,1/2)$. \modifbis{In this case, because $(\Lambda,\mu,\sigma)\in \Theta$, we have }$\int_{[\delta,1)}\log(1-r)r^{-2}\Lambda(\dd r)=-\infty$ or $\int_{(-1,-1/2)}\log(1+r)\mu(\dd r)=-\infty$. In either case, at least one term in \eqref{exprderp} converges to $-\infty$ as $b$ goes to infinity and, since \eqref{termcvd} still holds, all other terms are either negative or bounded. We thus get that $p_{b,\delta}'(0)$ converges to $-\infty$ as $b$ goes to infinity so $p_{b,\delta}'(0)<0$ for $b$ sufficiently large, which ends the proof.
\end{proof}
\begin{remark} \label{limrateofgrowth} Note that $\check{\psi}_{b,\delta}'(0)=\E[\check{L}_1^{b,\delta}]$. \mofe{By inspecting the previous proof we see that } $\lim_{\delta \rightarrow 0} \lim_{b \rightarrow \infty} \check{\psi}_{b,\delta}'(0)=C_0(\Lambda,\mu,\sigma)$.  
Moreover, one can similarly infer from the proof of Lemma~\ref{lem:existence_suit_approx} that $\E[\hat{L}_1^{b}]$ approximates $C_0(\Lambda,\mu,\sigma)$ as $b$ is large. \mofe{Therefore $C_0(\Lambda,\mu,\sigma)$ is the limit growth rate } of the L\'evy processes $\hat{L}^b$ and $\check{L}^{b,\delta}$ sandwiching $\log(1/Y)$ when $Y$ is close to the boundary $0$. 
\end{remark}
Next, we control the accumulation of small jumps of $\log(1/Y)$ induced by~$N$.

\begin{lemme}\label{lem:controlsmalljumps}
	Let $(\Lambda,\mu,\sigma)\in\Theta$. For any $\delta\in (0,1/2)$, $b \geq \log(2)$, and $y\leq e^{-b}$, \begin{eqnarray}
		\E_{y }\left[\exp\left(\frac{1}{4}H^\delta_{t\wedge T_{Y}(e^{-b},1]}\right)\right]\leq \exp \left(\frac{t}{2}\int_{(0,\delta)}r^{-1}\Lambda(\dd r)\right).
	\end{eqnarray}
\end{lemme}
\begin{proof}
	Define $D(b,\delta,t)\coloneqq[0,t \wedge T_{Y}(e^{-b},1]]\times (0,\delta)\times(0,1)$ and
\begin{eqnarray*}M_t&\coloneqq \exp\Big(\frac{1}{2}H^\delta_{t \wedge T_{Y}(e^{-b},1]}-\int_{D(b,\delta,t)} \Big(\sqrt{Y_{s-}/m_{r,u}(Y_{s-})}-1\Big)\dd s\,\frac{\Lambda(\dd r)}{r^2} \dd u\Big).
\end{eqnarray*} 
\mofe{Recall that the $r$-components of } the jumps of $N$ are summable on finite time intervals. Hence, applying It\^{o}'s formula~\citep[Thm.~4.4.10]{Applebaum2009} to the process 
\begin{eqnarray*}\Big(H^\delta_{t \wedge T_{Y}(e^{-b},1]},\int_{D(b,\delta,t)} \Big(\sqrt{Y_{s-}/{m_{r,u}(Y_{s-})}}-1\Big)\dd s \frac{\Lambda(\dd r)}{r^2} \dd u\Big)_{t\geq 0}
\end{eqnarray*} and the function $(x_1,x_2) \mapsto \exp(x_1/2 - x_2)$ yields
\begin{eqnarray*} M_t - 1 = \int_{D(b,\delta,t)} M_{s-} \times \Big(\sqrt{{Y_{s-}}/{m_{r,u}(Y_{s-})}}-1\Big)\tilde N(\dd s,\dd r,\dd u).
\end{eqnarray*}
Using \eqref{majova} from Lemma \ref{claim1} we get that $(M_t)_{t \geq 0}$ is a local martingale. Moreover, $M_0=1$ almost surely, so $(M_t)_{t \geq 0}$ is a non-negative \cadlag local martingale such that $M_0$ is integrable. Therefore, it is a supermartingale, and so {$\mathbb{E}_y[M_t]\leq 1$} for any $t\geq 0$. Using this, the definition of $M_t$, the Cauchy-Schwartz inequality, and \eqref{majoplusmoins} from Lemma \ref{claim1} yields
	\begin{eqnarray*}
		&\E_{y}\left[\exp\left(\frac{1}{4}H^\delta_{t\wedge T_{Y}(e^{-b},1]}\right)\right]
		=\E_{y}\Big[\sqrt{M_{t}}\exp\Big(\frac{1}{2}\int\limits_{D(b,\delta,t)} \big(\sqrt{{Y_{s-}}/{m_{r,u}(Y_{s-})}}-1\big)\dd s\frac{\Lambda(\dd r)}{r^2}\dd u \Big)\Big] \\
		&\leq{\E_{y}\Big[\exp\Big(\int\limits_{D(b,\delta,t)} \Big(\sqrt{{Y_{s-}}/{m_{r,u}(Y_{s-})}}-1\Big)\dd s\frac{\Lambda(\dd r)}{r^2}\dd u\Big)\Big]}^{\frac12}\leq\exp \Big(\frac{t}{2}\int_{(0,\delta)}\frac{\Lambda(\dd r)}{r}\Big).
	\end{eqnarray*}\end{proof}
We are now ready to prove Proposition~\ref{lem:escapetimeY}.
\begin{proof}[Proof of Proposition~\ref{lem:escapetimeY}]
Let us first prove (1).	Assume $C_0(\Lambda, \mu,\sigma)<0$.
We claim that for $b$ sufficiently large and any $y\leq e^{-b}$, we have $\P_{y }(T_{Y}(e^{-b},1]>t)\leq (ye^b)^{-1/4}e^{-2\gamma t}$ for some fixed $\gamma>0$ {independent of $y$}. Assume the claim is true. Then,  \begin{align*}
		\E_{y }[e^{\gamma T_{Y}(e^{-b},1]} ]&=1+\gamma \int_0^{\infty}e^{\gamma t}\P_{y }(T_{Y}(e^{-b},1]>t)\dd t\\
		&\leq 1+\gamma\int_0^\infty (ye^b)^{-1/4}e^{-\gamma t}\dd t= 1+ (ye^b)^{-1/4}\leq 2(ye^b)^{-1/4}<\infty.
	\end{align*}
Hence, setting $a:=e^{-b}$, we have for any $y\in(0,1)$
$$\E_{y }[e^{\gamma T_{Y}(a,1]} ]\leq 2(a/y)^{1/4}\ind{y\leq a}+\ind{y> a}\leq 2((a/y)\vee 1)^{1/4},$$
proving (1). It remains to prove the claim. By Lemma~\ref{lem:controlLaplaceofL}, we can fix $b\geq \log(2)$, $\delta\in (0,1/2)$, and $\lambda\in (0,1)$ such that $\check{\psi}_{b,\delta}(\lambda)+2\lambda \int_0^{\delta}r^{-1}\Lambda(\dd r)<0$. Set $\lambda_0\coloneqq \lambda/(4\lambda+1)$. Then, using the upper bound in Lemma~\ref{lem:levysandwich} and Chernoff's inequality,
	\begin{align*}
		&\P_{y }(T_{Y}(e^{-b},1]>t)=\P_{y }(T_{Y}(e^{-b},1]>t, \log(1/Y_t)\geq b)\\
		&\leq \P_{y }(T_{Y}(e^{-b},1]>t, \check{L}_t^{b,\delta}+H_t^{\delta}\geq b+\log(y))\leq\P_{y }(\check{L}_t^{b,\delta}+H_{t\wedge T_{Y}(e^{-b},1] }^{\delta}\geq b+\log(y))\\&\mofe{\leq (ye^{b})^{-\lambda_0} \E_{y }\left[\exp\left(\lambda_0 \check{L}_t^{b,\delta} \right) \exp\Big(\lambda_0 H_{t\wedge T_{Y}(e^{-b},1]}^{\delta}\Big) \right]}.
	\end{align*}
Since $\lambda_0\leq 1/4$, $(ye^{b})^{-\lambda_0}\leq (ye^{b})^{-1/4}$ so that it remains to deal with the expectation. Note that using H{\"o}lder's inequality, the definition of the Laplace exponent and Lemma~\ref{lem:controlsmalljumps},
	\begin{eqnarray*}
		\E_{y } &\left[\exp\Big(\lambda_0 \check{L}_t^{b,\delta} +\lambda_0 H_{t\wedge T_{Y}(e^{-b},1]}^{\delta}\Big) \right]
		\leq\E_{y }\left[\exp\left(\lambda \check{L}_t^{b,\delta} \right)\right]^{\frac{\lambda_0}{\lambda}}  \E_{y }\left[\exp\left({\frac{1}{4} H_{t\wedge T_{Y}(e^{-b},1]}^{\delta}}\right) \right]^{4\lambda_0}\\
		&\qquad\qquad\leq \exp\bigg( t\frac{\lambda_0}{\lambda}\bigg( \check{\psi}_{b,\delta}(\lambda)+2\lambda \int_{(0,\delta)}r^{-1}\Lambda(\dd r)\bigg)\bigg)\modifbis{= e^{-2\gamma t}},
	\end{eqnarray*}
	where $\gamma\coloneqq -\frac{\lambda_0}{2\lambda}(\check{\psi}_{b,\delta}(\lambda)+2\lambda\int_{(0,\delta)}r^{-1}\Lambda(\dd r))$, which is positive by Lemma~\ref{lem:controlLaplaceofL}. This completes the proof of (1).
	
	Let us now prove (2). Assume $C_1(\Lambda,\mu,\sigma)<0$ and note that $1-Y$ is \mofe{a $(\Lambda,\bar{\mu},\bar{\sigma})$-Wright--Fisher process with $\bar{\sigma}(y)\coloneqq-\sigma(1-y)$ satisfying $C_0(\Lambda,\bar{\mu},\bar{\sigma})<0$. Hence, we get from } part~(1) that there is $\bar{a}\in (0,1/2)$ and $\gamma>0$ such that $\E_y[e^{\gamma T_{1-Y}(\bar{a},1]}]\leq 2(\bar{a}/(1-y))^{1/4}$.
	Since $T_{Y}[0,1-\bar{a})=T_{1-Y}(\bar{a},1]$, we can choose this $\gamma$ and set $a\coloneqq1-\bar{a}$ to  complete the proof.
\end{proof}
\subsection{On the L\'evy majorant around $0$}\label{sec:asextinction:sub:prepext}
This section focuses on properties of $\hat{L}^b$ and will enable us, in Section \ref{sec:asextinction} (more precisely in the proof of Propositions~\ref{prop:controltimetobdd} and \ref{prop:controltimetobddstrong}), to estimate the time that $Y$ spends at the boundary strip before leaving it (if it departs).

Recall that for $b\geq \log(2)$, the L\'evy process $\hat{L}^b$ is a lower bound for $\log(1/Y)$ as long as $Y$ is in $(0,e^{-b}]$; see~\eqref{eq:levylower} to recall the definition of~$\hat{L}^b$ and Lemma \ref{lem:levysandwich}. Denote by $\iota$ be the identity function in $\R_+$, i.e. $\iota(s)=s$ for $s\geq 0$.  
We first establish the long-term behavior of the process $\hat{L}^{b}-m\iota\coloneqq(\hat{L}^b_t-mt)_{t\geq 0}$ for~$m$ sufficiently small and $b$ sufficiently large. The following lemma states that if $C_0(\Lambda,\mu,\sigma)>0$, then there is a neighborhood around~$0$ for~$m$ such that for $b$ large enough $\hat{L}^b-m\iota$ is essentially bounded from below.
\begin{lemme}\label{lem:existence_suit_approx}
	Let $(\Lambda,\mu,\sel)\in\Theta$ \mofe{and assume that } $C_0(\Lambda,\mu,\sigma)>0$. Then, for any $m \in [0,C_0(\Lambda,\mu,\sigma))$, there exists $\beta_m\geq \log(2)$ such that for all $b\geq \beta_m$, $\E[\hat{L}_1^b]$ is well-defined and $\E[\hat{L}_1^b]>m$. In particular, \[\inf_{t\in [0,\infty)}(\hat{L}_t^b-mt)>-\infty\quad\textrm{almost surely and}\quad \P\left(\inf_{t\in[0,\infty)} (\hat{L}_t^b-mt)>-\log(2)\right)>0.\]
\end{lemme}
\begin{proof}
By the discussion \mofe{preceding } \eqref{eq:laplacelower}, the support of $\hat{\nu}_b$ is bounded from above by $\log(2)$. Therefore, by \citep[Thm.~25.3]{KenIti1999}, $\E[\hat{L}_1^b \vee 0]<\infty$ so \modifbis{$\E[\hat{L}_1^b]=\int_{\mathbb{R}} x\hat{\nu}_b(\dd x)\in [-\infty,\infty)$}. This translates into
	\begin{align*}
		\E[\hat{L}_1^b]&=\int_{(0,1)}\log(1-r)\frac{\Lambda(\dd r)}{r^2}+\int_{(-1,1)}\frac{1}{2}\log\Big((1+r)^2-4re^{-b}\ind{r>0}\Big)\mu(\dd r)\\
		&+\sigma(0)-e^{-b}\lVert \sigma\rVert_{\Cs^1([0,1])}\\
		&= C_0(\Lambda,\mu,\sigma)+\frac{1}{2}\int_{(0,1)}\log\Big( 1-\frac{4re^{-b}}{(1+r)^2}\Big)\mu(\dd r)-e^{-b}\lVert \sigma\rVert_{\Cs^1([0,1])}.
	\end{align*}
	Condition $(\Lambda,\mu,\sigma)\in \Theta$ implies that $\int_{(0,1)}|\log( 1-4re^{-b}(1+r)^{-2})|\mu(\dd r)<\infty$ so, by monotone convergence, the above expression converges to $C_0(\Lambda,\mu,\sigma)$ as $b\to\infty$. Hence, for each $m \in [0,C_0(\Lambda,\mu,\sigma))$, we can choose $\beta_m\geq \log(2)$ sufficiently large such that $\E[\hat{L}_1^b]-m>0$ for all $b \geq \beta_m$. Thus, for such $b$, by the law of large numbers for L\'evy processes (e.g. \citep[Thm.~36.5]{KenIti1999}), the L\'evy process $\hat{L}^b - m\iota$ drifts to~$\infty$. This allows us to deduce from \citep[Thm.~48.1]{KenIti1999} that $\inf_{t\in[0,\infty)}(\hat{L}_t^b - mt)>-\infty$ almost surely, and that $\inf_{t\in[0,\infty)}(\hat{L}_t^b - mt)$ is infinitely divisible \modifbis{with null drift component. The nullity of the drift component implies }that $0$ is contained in the support of the law of $\inf_{t\in[0,\infty)}(\hat{L}_t^b - mt)$ (\citep[Cor.~24.8]{KenIti1999}). Therefore, $\P(\inf_{t\in[0,\infty)}(\hat{L}_t^b - mt)>-\log(2))>0$, ending the proof.
\end{proof}


The next result provides conditions for the \mofe{existence of polynomial/exponential moments for the } last time at which $\hat{L}^b-m\iota$ reaches its global minimum. 
As a consequence, we also obtain an upper bound for the tail probability of the first time at which $\hat{L}^b-m\iota$ goes below a given level, when that time is finite.

\begin{lemme}\label{lem:lowerlevyhittingtime}
	Let $(\Lambda,\mu,\sel)\in\Theta$. Assume that $C_0(\Lambda,\mu,\sigma)>0$ and let $m \in [0,C_0(\Lambda,\mu,\sigma))$ and $b>\log(2)$ such that $\E[\hat{L}^b_1]>m$ (such $b$ exists because of Lemma~\ref{lem:existence_suit_approx}). 
	Let \[H\coloneqq \sup\{t\geq 0:\ \hat{L}_t^b\wedge \hat{L}_{t-}^b-mt=\inf_{s\in [0,\infty)}(\hat{L}_s^b-ms)\}.\] 
	\begin{enumerate}
		\item If $\rW_\gamma(r^{-2} \Lambda(\dd r))<\infty$ and $\rW_\gamma(\bar{\mu})<\infty$ for some $\gamma>0$, 
		then we have, for all $\alpha\in (0,\gamma)$, $\E[H^\alpha]<\infty$ and for all $t,x>0$, \[\P({t<T_{\hat{L}^b -m\iota}(-\infty,-x)<\infty} )\leq \E[H^\alpha] t^{-\alpha}.\] 
		\item If $s_\gamma(r^{-2} \Lambda(\dd r))<\infty$ and $s_\gamma(\bar{\mu})<\infty$ for some $\gamma>0$, 
		then we have, for some $\alpha>0$, $\E[e^{2\alpha H}]<\infty$ and for all $t,x>0$,  \[\P({t<T_{\hat{L}^b -m\iota}(-\infty,-x)<\infty} )\leq \E[e^{2\alpha H}] e^{-2 \alpha t}.\]
	\end{enumerate}
\end{lemme}
\begin{proof}
	We first prove (1). 
Since $\E[\hat{L}_1^b]>m$, \modifbis{$\hat{L}^b-m\iota$ drifts to $\infty$ (see the proof of Lemma~\ref{lem:existence_suit_approx}), so $H$ is well-defined and finite}.
By \eqref{eq:levylower}, if $(t,r,u)$ is a jump of $N$, then the corresponding jump of $\hat{L}^b-m\iota$ is $\log(1-r)<0$. Similarly, if $(t,r)$ is a jump of~$S$ with $r<0$, then the corresponding jump of $\hat{L}^b-m\iota$ is $\log(1+r)<0$. If $(t,r)$ is a jump of~$S$ with $r>0$, then the corresponding jump of $\hat{L}^b-m\iota$ is non-negative. 
	Thus, $\rW_\gamma(r^{-2} \Lambda(\dd r))<\infty$ and $\rW_\gamma(\bar{\mu})<\infty$ ensures that $-(\hat{L}^b-m\iota)$ satisfies the requirement of Lemma~\ref{lem:aux_driftlevy} (with $\theta=\gamma)$. Hence, for all $\alpha \in (0,\gamma)$, $\E[H^\alpha]<\infty$. Then, using the definition of~$H$ and Markov's inequality, we obtain $$\P({t<T_{\hat{L}^b-m\iota}(-\infty,-x)<\infty})\leq \P(H>t)\leq \E[H^\alpha]t^{-\alpha}.$$ 
	This ends the proof of (1).	For (2), note from Lemma \ref{eq:remark_stronglaplace} that $\hat{L}^b$ (and therefore $\hat{L}^b-m\iota$) has a Laplace transform on $[-\gamma,\infty)$ and that its Laplace exponent on this interval is given by $\hat{\psi}_b$. Let us write $\tilde{\psi}_{b,m}(\lambda)\coloneqq \hat{\psi}_b(\lambda)-m\lambda$ for the Laplace exponent of $\hat{L}^b-m\iota$. Then, using that $(\tilde{\psi}_{b,m})'(0)=\E[\hat{L}^b_1-m]>0$, and that $\tilde{\psi}_{b,m}(0)=0$, we deduce that there is $\lambda_0\in (0,\gamma)$ small enough such that $\tilde{\psi}_{b,m}(-\lambda_0)<0$. Thus, $-(\hat{L}^b-m\iota)$ satisfies the requirements of Lemma~\ref{lem:auxlevyexp}. In particular, there exists $\alpha>0$ such that $\E[e^{2\alpha H}]<\infty$. Then, using the definition of~$H$ and Chernoff's inequality, $$\P({t<T_{\hat{L}^b-m\iota}(-\infty,-x)<\infty})\leq \P(H>t)\leq \E[e^{2\alpha H}]e^{-2\alpha t},$$ which completes the proof of (2). 
\end{proof}
\begin{cor}\label{coro:lowerlevyhittingtime}
	Let $(\Lambda,\mu,\sel)\in\Theta$. Assume that $C_0(\Lambda,\mu,\sigma)>0$, $s_\gamma(r^{-2} \Lambda(\dd r))<\infty$ and $s_\gamma(\bar{\mu})<\infty$ for some $\gamma>0$. Let $m \in [0,C_0(\Lambda,\mu,\sigma))$ and $b>\log(2)$ such that $\E[\hat{L}^b_1]>m$  (the existence of $b$ follows from Lemma~\ref{lem:existence_suit_approx}). Then, there is $\alpha,K>0$ such that for any $x>0$, $$\E[e^{\alpha T_{\hat{L}^b - m\iota}(-\infty,-x)}\ind{T_{\hat{L}^b - m\iota}(-\infty,-x)<\infty}]\leq K.$$
\end{cor}
\begin{proof}
Let $H$ be as in Lemma~\ref{lem:lowerlevyhittingtime}. By Lemma~\ref{lem:lowerlevyhittingtime}--(2), there is $\alpha>0$ such that $\E[e^{2\alpha H}]<\infty$. Fix $x\in \R$. We have, \begin{align*}
		&\E[e^{\alpha T_{\hat{L}^b - m\iota}(-\infty,-x)}\ind{T_{\hat{L}^b - m\iota}(-\infty,-x)<\infty}]
		= \P(T_{\hat{L}^b - m\iota}(-\infty,-x)<\infty)\\
		&+\alpha\int_0^\infty e^{\alpha t}\P(t<T_{\hat{L}^b - m\iota}(-\infty,-x)<\infty)\dd t
		 \leq 1+ \E[e^{2\alpha H}]<\infty,
	\end{align*}
where in the second last inequality we have used Lemma~\ref{lem:lowerlevyhittingtime}--(2).
\end{proof}
\subsection{Deviation from initial condition after removing large jumps}\label{sec:keyproperties:deviation}
In some proofs to come (of Lemmas~\ref{lem:aux_mergingeventestimate}, \ref{lem:renewalauxeventbound}, and \ref{lem:specialeventbound}), it will be useful to control the maximal deviation of $Y$ from its initial value in compact time intervals if we cap the large jumps of~$N$ (i.e. with large $r$-component). This is the content of the following lemma.

\begin{lemme}\label{lem:capdeviation}
	{Let $(\Lambda,\mu,\sigma)\in \Theta$}. For $c\in (0,1)$, denote by $Y^c=(Y_s^c)_{s\geq 0}$ the unique strong solution to 
\begin{align}
\dd Y_t^c&=\int_{(0,c]\times(0,1)}\!\!\! \!\!\!\!\!(\mr_{r,u}(Y_{t-}^c)-Y_{t-}^c)\,N(\dd t, \dd r, \dd u)\nonumber\\ 
&-Y^c_t(1-Y^c_t)\, \sel(Y^c_t)\, \dd t + \int_{(-1,1)}\!\!\!\!\!(s_r(Y^c_{t-})-Y^c_{t-})\, S(\dd t, \dd r),\label{Yc}
\end{align}
with $Y_0^c=y\in (0,1)$. {There is a constant $\capco{\star}>0$ such that for any $c\in (0,1)$}, $y\in (0,1)$, $t\in [0,1]$, and $\lambda>0$, we have  
	\[	\P_y\Big(\sup_{s\in [0,t]}\lvert Y_s^c-y\rvert\geq \lambda\Big)\leq {\capco{\star}} \frac{\sqrt{t}}{\lambda}. \]
\end{lemme}
\begin{proof}
\modifbis{The process $Y^c$ from \eqref{Yc} is a $(\Lambda^c,\mu,\sigma)$-Wright--Fisher process (where $\Lambda^c(\dd r):=\textbf{1}_{r \leq c}\Lambda(\dd r)$) so Proposition \ref{lem:existuniqueY} provides existence and pathwise uniqueness of strong solutions to SDE \eqref{Yc}}. By Markov's inequality, $$\P_y\Big(\sup_{s\in[0,t]}\lvert Y_s^c-y\rvert\geq \lambda\Big)\leq\frac{\E_y[\sup_{s\in[0,t]}\lvert Y_s^c-y\rvert]}{\lambda}.$$ We now show that {\mofe{there is $\capco{\star}>0$ (independent of $c, y$ and $t$) such that the expectation above is upper bounded by } $\capco{\star}\sqrt{t}$, from which the result follows}. Let $\tilde{N}$ be the compensated Poisson measure with intensity $\dd t\times r^{-2}\Lambda(\dd r)\times \dd u$. \mofe{Setting $I_{s,c}\coloneqq[0,s]\times(0,c]\times(0,1)$, we have 
\begin{eqnarray}
		&\qquad\E_y\left[\sup_{s\in[0,t]}\lvert Y_s^c-y\rvert\right] \leq \E_y\bigg[\sup_{s\in[0,t]}\Big\lvert \int_{I_{s,c}}(\mr_{r,u}(Y^c_{w-})-Y^c_{w-})\tilde{N}(\dd w, \dd r, \dd u)\Big\rvert\bigg]\label{eq:capbound1} \\ 
		 &+\E_y\bigg[  \int_{[0,t]}\Big\lvert\int_{(0,c]\times(0,1)} (\mr_{r,u}(Y^c_{s-})-Y^c_{s-}) r^{-2} \Lambda(\dd r) \dd u\Big\rvert\dd s\bigg] \label{eq:capbound2}\\
		& \qquad\quad+ \E_y\bigg[\int_{[0,t]} \lvert Y_s^{c}(1-Y_s^{c})\, \sel(Y^c_{s})\rvert \dd s\bigg]+ \E_y\bigg[\int_{[0,t]\times(-1,1)}\left\lvert s_r(Y_{s-}^c)-Y_{s-}^c\right\rvert S(\dd s, \dd r)\bigg].\label{eq:capbound3}
	\end{eqnarray}
	Using Cauchy-Schwartz inequality, we get that the term on the right-hand side of ~\eqref{eq:capbound1} is less than or equal to 
	\begin{eqnarray*}
&\E_y\left[\Big(\sup_{s\in[0,t]} \int_{I_{s,c}}(\mr_{r,u}(Y^c_{w-})-Y^c_{w-})\tilde{N}(\dd w, \dd r, \dd u)\Big)^2\right]^{\frac12}\\
		&\quad\qquad\qquad\leq \left(4\E_y\left[\Big(\int_{I_{t,c}}(\mr_{r,u}(Y^c_{s-})-Y^c_{s-})\tilde{N}(\dd s, \dd r, \dd u)\Big)^2\right]\right)^{\frac12}\leq 4\sqrt{\Lambda((0,1))t},
	\end{eqnarray*}}
	\modifbis{where in the last two steps we used Doobs martingale inequality, and Lemma~\ref{lem:medianboundspoisson}}. The term in \eqref{eq:capbound2} is less than or equal to $4\Lambda([0,1])t$ by Lemma~\ref{lem:medianboundspoisson}. The first term in \eqref{eq:capbound3} is upper bounded by $t\lVert \sigma\rVert_{\infty}/4$. Using Lemma~\ref{bsr} and that $s_r(y)\in[0,1]$ for all $r\in(-1,1)$ and $y\in[0,1]$, we deduce that the second term in \eqref{eq:capbound3} is upper bounded by
$$\E_y\bigg[\int_{[0,t]\times(-1,1)} (C_1|r|\wedge 1) \,S(\dd s, \dd r)\bigg]=t\int_{(-1,1)} (C_1|r|\wedge 1) \, \mu(\dd r),$$	
where $C_1>0$ is the constant in Lemma~\ref{bsr}. Putting all pieces together yields the result. \end{proof}

\subsection{Estimates on the decay of the tail probabilities of some random times}\label{sec:keyproperties:rigthdistributiontail}
In the next sections, we will require an estimate on the tail probabilities of some specific random times (e.g. merging time of two trajectories, renewal time, trapping time, etc.). To this end, we define two \emph{interlaced sequences of stopping times}; one sequence providing the entrance times into a favorable situation in which we can estimate the probability for the random time of interest to occur, and the other sequence providing the exit times from the favorable situation. The following lemma provides upper bounds for the \mofe{tail probabilities } of such random times under some assumptions on the interlacing stopping time sequences. The lemma will be used in particular to conclude the proofs of Theorems~\ref{thm:stuckatboundary}, \ref{thm:probmergeYt}, and \ref{thm:coexistenceY} in Sections~\ref{sec:asextinction}, \ref{sec:accbdd}, and~\ref{sect:coex}, respectively. The assumptions in the lemma ensure that the number of iterations necessary for the random time to occur is bounded by a geometric random variable, and that the cumulative duration of $n$ iterations can be controlled.
\begin{lemme} \label{unificationlemma}
Let $T$ be a random time (not necessarily a stopping time). Assume that there are two sequences of (possibly infinite) stopping times $(T^1_n)_{n\geq 0}$ and $(T^2_n)_{n\geq 0}$ such that the following holds: 
\begin{enumerate}
\item[(i)] Almost surely, $T^1_0=0$ and $T^1_n \leq T^2_n \leq T^1_{n+1}$ for all $n\in \N_0$ (with the convention $\infty\leq \infty$). 
\item[(ii)] There is a constant $c>0$ such that for any $n\in \N_0$, \mofe{$$\P(T \in (T^2_n, T^1_{n+1}]\mid T>T^2_n)>c$$}(with the convention $\infty \nless \infty$, so that $T>T^2_n$ implies $T^2_n<\infty$). 
\item[(iii)] For any $n\geq 1$, $\P(\{T^2_{n-1}<\infty, T^2_{n}=\infty\} \cap \{T\neq T^1_n\})=0$. 
\item[(iv)] There are constants $\lambda>0$, $K>1$ and $a\geq 1$ such that for any $n\geq 1$, \mofe{$$\E[e^{\lambda T^1_n}\textbf{1}_{T^2_{n-1}<\infty}]\leq aK^n.$$} In particular, $\P(T^1_n<\infty \mid T^2_{n-1}<\infty)=1$. 
\end{enumerate}
The condition (iv) will sometimes be replaced by the following 
\begin{enumerate}
\item[(iv)'] There are constants $\ell>0$ and $a\geq 1$ such that for any $n\geq 1$ and $t>0$, \mofe{$$\P(T^1_n>t,T^2_{n-1}<\infty)\leq a n^{1+\ell}t^{-\ell}.$$} In particular, $\P(T^1_n<\infty\mid T^2_{n-1}<\infty)=1$. 
\end{enumerate}
Under Conditions (i), (ii), (iii), (iv) (resp. (i), (ii), (iii), (iv)'), $T$ is almost surely finite and there is a constant $H>0$ (resp. for any $h \in (0,\ell)$ there is a constant $M>0$) such that for any $t>0$ we have 
\begin{align}
\P(T>t) \leq a \Big ( \frac{K}{K-1}+\frac{1}{1-c} \Big ) e^{-H t} \qquad \big ( \text{resp.} \ \P(T>t) \leq a M t^{-h} \big{)}, \label{estimexpopol}
\end{align}
where $H:=\lambda \log(1/(1-c))/\log(K/(1-c))$ (resp. $M\coloneqq 1\vee ((2+\ell)^{-1} + (1-c)^{-2}m_h)$ with  $m_h\coloneqq \textstyle\sup_{s \geq 0} s^h \exp(s^{(\ell-h)/(2+\ell)}\log(1-c))$). 
\end{lemme}

\begin{proof}
We assume that Conditions (i), (ii), (iii) hold. Let $t>0$ and $z=z(t)$ that will be determined later. We have 
\begin{align}\label{eq:probunif}
\P(T>t)\leq \P \left (T>T^1_{\lfloor z \rfloor} \right)+\P \left (T^1_{\lfloor z \rfloor} \geq T > t\right).
\end{align}
We first deal with the first term on the right-hand side in \eqref{eq:probunif}. Note that for any~$n\geq1$,
\begin{align*}
\P \left (T>T^1_n \right)&\leq \E \left [\ind{T>T^2_{n-1}} \P \left(T>T^1_n \mid T>T^2_{n-1}\right )\right ]\\
	&\leq (1-c)\P \left( T>T^2_{n-1} \right)
	\leq (1-c) \P \left (T>T^1_{n-1} \right),
\end{align*}
where we used Conditions (i) and (ii). By induction, we get $\P(T>T^1_n)\leq (1-c)^n$. Thus, 
\begin{align}
\P \left (T>T^1_{\lfloor z \rfloor} \right) \leq (1-c)^{\lfloor z \rfloor}\leq e^{\log(1-c)z}/(1-c). \label{eq:probunif1}
\end{align}
Next, we deal with the second term on the right-hand side in \eqref{eq:probunif}. Using Conditions (i) and (iii) we get 
\begin{align}
\P \left (T^1_{\lfloor z \rfloor} \geq T > t\right)&\leq \sum_{k=1}^{\lfloor z\rfloor } \P \left ( T^1_k>t, T \in (T^1_{k-1}, T^1_k] \right) \leq \sum_{k=1}^{\lfloor z \rfloor} \P \left ( T^1_k>t, T^2_{k-1}<\infty \right). \label{eq:probunif2}
\end{align}
Now we have to distinguish cases according to the assumptions. In the first case we assume that Condition (iv) holds true. Using Chernoff's inequality in \eqref{eq:probunif2} and Condition (iv) we get 
\begin{align*}
\P \left (T^1_{\lfloor z \rfloor} \geq T > t\right) \leq e^{-\lambda t}\sum_{k=1}^{\lfloor z \rfloor}\E \left [ e^{\lambda T^1_k}\textbf{1}_{T^2_{k-1}<\infty} \right ]\leq a e^{-\lambda t} \frac{K^{z+1}}{K-1}.
\end{align*}
Choosing $z=\lambda t/\log(K/(1-c))$ and combining with \eqref{eq:probunif} and \eqref{eq:probunif1} yields the first part of \eqref{estimexpopol}. 

In the second case we assume that Condition (iv)' holds true. The sought estimate clearly holds for $t\leq 1$ so we assume $t>1$. Using Condition (iv)' in \eqref{eq:probunif2} we get 
\begin{align*}
\P \left (T^1_{\lfloor z \rfloor} \geq T > t\right) \leq a t^{-\ell}\sum_{k=1}^{\lfloor z \rfloor} k^{1+\ell} \leq a t^{-\ell}\int_1^{z+1} v^{1+\ell}\dd v\leq a t^{-\ell}\frac{(z+1)^{2+\ell}}{2+\ell}.
\end{align*}
Choosing $z=t^{(\ell-h)/(2+\ell)}-1$ and combining with \eqref{eq:probunif} and \eqref{eq:probunif1} yields the second part of \eqref{estimexpopol}. 
\end{proof}


\section{The case \texorpdfstring{$\Theta_2$}{Theta2}: Proof of Theorem \ref{thm:probmergeYt}}\label{sec:accbdd}

In this section, we prove Theorem \ref{thm:probmergeYt}. To begin, we outline the strategy for establishing part (i) and (ii). Subsequently, we state preparatory results and then prove the theorem. The \mofe{rest of the section will be devoted to the proofs of the preparatory results}.
\smallskip

Recall that in Theorem~\ref{thm:probmergeYt}, we consider two strong solutions $\widehat{Y}$ and $\widecheck{Y}$ of SDE~\eqref{eq:SDEWFP} that evolve in the same background and with $\widehat{Y}_0=\hat{y}\leq\check{y} =\widecheck{Y}_0$. In part (i) and (ii), we want to find an upper bound in terms of $\hat{y}$ and $\check{y}$ for the probability that the two processes \mofe{have not merged by time }~$t$ and for an exponential moment of the renewal time~$\renewal$\mofe{, respectively}. To this end, it will be convenient to consider two sub-intervals of~$[0,1]$. For $\rin\in (0,1/2)$ and $\rbs\in (0,\rin)$, we refer to $(\rin,1-\rin)$ as the \emph{$\rin$-interior} and to $(0,1)\setminus[\rbs,1-\rbs]$ as the \emph{$\rbs$-boundary strip}. We will assign distinct values to $\rin$ and $\rbs$ based on the situation.

The three main ingredients to obtain the result are: 1) Give a lower bound for the probability that the relevant event happens before the $\rbs$-boundary strip is reached and before some fixed finite time (where the event for Theorem~\ref{thm:probmergeYt} part (i) and (ii) is the merger of two trajectories and the occurrence of the renewal time~$\renewal$, respectively).  2) If one of the two solutions has entered the $\rbs$-boundary strip before the occurrence of the event of interest, estimate the time until the two trajectories are again simultaneously in the $\rin$-interior. 3) \modifbis{Control the total time taken by $n$ iterations of the phases from steps 1) and 2) and conclude using Lemma \ref{unificationlemma}}. We think our approach is of interest in its own right, because the argument seems to generalize to processes that admit a L{\'e}vy sandwich \modifbis{near the boundary of their domain}.

The underlying idea appears in classic works and in many settings it is formalized in a Lyapunov-function framework. A minorisation condition on the kernel out of a predefined set usually ensures the possibility of merging two solutions whenever they are in that set. A drift condition, which usually involves the generator and a so-called Lyapunov-function, ensures that the predefined set is sufficiently recurrent. We refer to~\citep{Jones2001} for an accessible introduction, and to~\cite{Meyn1993} for relevant details. In contrast, our approach employs a (generalized) Doeblin condition as it appears for example in~\citep{Kulik2011} (see \citep[p.192 ff.]{Doob1990} for a more classic formulation). 

In order to carry out the steps outlined above, we define the following sequence of stopping times (see Fig.~\ref{fig:bothboundaries} for an illustration). Let $\tin>0$. Set $\waitbds{0}\coloneqq 0$ and for $n\geq 0$,
\begin{equation}\label{eq:hittingtimes}
	\begin{split}
		\waitint{n}&\coloneqq \inf\{t\geq \waitbds{n}:\rin< \widehat{Y}_t\leq \widecheck{Y}_t<1-\rin\}, \\
	\waitbds{n+1}&\coloneqq (\waitint{n}+\tin)\wedge \inf\{t\geq \waitint{n}: \widehat{Y}_t<\rbs \text{ or }\widecheck{Y}_t>1-\rbs\}.
	\end{split}
\end{equation}
\modifbis{The dependence of $(\waitbds{n}, \waitint{n})_{n\geq 0}$ on the choice of $\rin, \rbs, \tin$ is omitted in the notation. Let also}
\begin{align*}
\Tint&\coloneqq \inf\{t\geq 0: \rin< \widehat{Y}_t\leq \widecheck{Y}_t<1-\rin\},\quad
\modifbis{\Tstr\coloneqq \inf\{t\geq 0: \widehat{Y}_t <\rbs \text{ or }\widecheck{Y}_t >1-\rbs\}}. 
\end{align*}
Then $\waitint{}$ and $\waitbdsu$ are the first time $\widehat{Y}$ and $\widecheck{Y}$ are simultaneously in the $\rin$-interior and the first time $\widehat{Y}$ or $\widecheck{Y}$ enters the $\rbs$-boundary strip, respectively. We have \mofe{$\waitint{}=\waitint{0}$}. 
Note that if $\widehat{Y}_{\waitbds{n}},\widecheck{Y}_{\waitbds{n}}\in (\rin,1-\rin)$, then $\waitint{n}=\waitbds{n}$. 
Moreover, if $\hat{y},\check{y}\in(\rin,1-\rin)$, then $\waitint{0}=0$ so in this case $\waitbds{1}=\waitbdsu \wedge \tau$. \modifbis{By a slight abuse of notation, we shall use $\waitint{n}, \waitbds{n}, \Tint, \Tstr$ also if we consider only a single trajectory under $\P_y$, i.e. if $\hat{y}=\check{y}=y$}, for example in Proposition \ref{prop:renewaloccbound}.


In the strategy outlined above, \mofe{step 1) corresponds to } the following Doeblin's conditions.
\begin{prop}[Doeblin's condition for part (i)]\label{prop:mergeprob}
	Let $(\Lambda,\mu,\sel)\in\Theta$ and $\rin\in(0,3^{-1}\wedge \max\supp\Lambda)$. 
	Then there exists $\rbs\in (0,\rin)$, $\tstop>0$ and $c_{\rin}>0$ 
	such that for any $\hat{y},\check{y}$ 
	with $\rin\leq  \hat{y}\leq  \check{y}\leq 1-\rin$,
	$$\P_{\hat{y},\check{y}}(\widehat{Y}_{\tstop\wedge  \waitbdsu}= \widecheck{Y}_{\tstop\wedge\waitbdsu})\geq c_{\rin}.$$
\end{prop}

\begin{prop}[Doeblin's condition for part (ii)]\label{prop:renewaloccbound}
	Let $(\Lambda,\mu,\sel)\in\Theta$. Recall the definition of $\renewal$ from \eqref{eq:defrenewaltime}. For all $\rin\in (0,1/2)$ there is $\rbs\in (0,\rin)$, $\tstop>0$, and $c_{\rin}\in (0,1)$ such that for all $y\in [\rin,1-\rin]$, 
	\begin{equation}\label{Doeb2}
		\P_y(\renewal<\tstop\wedge \waitbdsu)\geq c_{\rin}.
	\end{equation}	
\end{prop}

\modifbis{For step 2), the following proposition estimates the time we need to wait for two trajectories to be both found in the interior. The proof relies on the results from Section \ref{sec:comparisonlevy}}. 
\begin{prop}[Time to interior]\label{prop:bddprobreturn}
	Let $(\Lambda,\mu,\sel)\in \Theta_2$. There exists $\rin\in(0,3^{-1}\wedge \max\supp\Lambda)$ and $\lambda,\, C>0$ such that for any $\hat{y},\, \check{y}\in (0,1)$ with $\hat{y}\leq \check{y}$, we have 
	\begin{equation}\label{expmom}
		\E_{\hat{y},\check{y}}[e^{\lambda \wht}]\leq C(\hat{y}^{-1/4}+(1-\check{y})^{-1/4})<\infty.
	\end{equation}
	In particular, $\wht<\infty$, {$\P_{\hat{y},\check{y}}$-a.s.}.
\end{prop}
The condition $\rin<\maxsupp{\Lambda}$ ensures the existence of jumps $(t,r,u)$ of~$N$ with $r>\rin$. 

Finally, step 3) will consist in an application of Lemma \ref{unificationlemma} to the interlaced sequences of stopping times defined above. The following proposition will allow to check the assumptions of that lemma. It estimates the amount of time accumulated after $n$ repetitions of successively reaching the interior and trying to see the event of interest occur. The proof relies on Proposition \ref{prop:bddprobreturn} and on \mofe{an analysis of the law of $Y$ once it enters the } boundary strip.
\begin{prop}[Accumulated time after $n$ steps]\label{prop:nthboundaryvisit}
	Let $(\Lambda,\mu,\sigma)\in \Theta_2$. 
	Let $\rin$, $\lambda,\, C$ as provided by Proposition~\ref{prop:bddprobreturn}. 
	Fix $\tin>0$ and $\rbs\in (0,\rin)$. 
	Then there is $K\geq 1$ 
	such that for all $n\geq0$ and $0<\hat{y}\leq\check{y}<1$, 
	$$\E_{\hat{y},\check{y}}[e^{\lambda \waitbds{n}}]\leq (\hat{y}^{-1/4}+(1-\check{y})^{-1/4})K^n.$$
\end{prop}
\modifbis{On the basis of these preparatory results, we are ready to prove Theorem \ref{thm:probmergeYt}.}
\begin{proof}[Proof of Theorem \ref{thm:probmergeYt}]
	First we prove part (i).
	\modifbis{Let $\rin$ and $\lambda$ be given by Proposition~\ref{prop:bddprobreturn}. 
	Choose $\rbs\in (0,\rin)$, $\tstop>0$ and $c_{\rin}\in (0,1)$ as in Proposition~\ref{prop:mergeprob}.
	Set $\tin\coloneqq\tstop$. Let also $0<\hat{y}\leq\check{y}<1$ be initial starting points for $\widehat{Y}$ and $\widecheck{Y}$. 
	Let $\waitint{n},\waitbds{n}$ be the stopping times defined in~\eqref{eq:hittingtimes} (with $\rin, \rbs, \tin$ as we just chose)}.
	Let $K$ be as provided by Proposition~\ref{prop:nthboundaryvisit} (\mofe{recall that } $K$ does not depend on the choice of $\hat{y}$ and $\check{y}$). 
	We check that the assumptions of Lemma~\ref{unificationlemma} are satisfied for 
	$T\coloneqq\inf\{t\geq 0: \widehat{Y}_t=\widecheck{Y}_t\}$, $(T^1_n)_{n \geq 0}\coloneqq(\waitbds{n})_{n \geq 0}$ and $(T^2_n)_{n \geq 0}\coloneqq(\waitint{n})_{n \geq 0}$. 
	Condition (i) is clearly satisfied. 
	Proposition~\ref{prop:mergeprob} ensures that Condition (ii) holds with $c:=c_{\rin}$. Proposition~\ref{prop:nthboundaryvisit} shows that, almost surely, all stopping times $T^1_n$ are finite, and therefore all stopping times $T^2_n$ as well so Condition (iii) is satisfied. 
	That lemma also shows that Condition (iv) is satisfied with $a:=\hat{y}^{-1/4}+(1-\check{y})^{-1/4}$ and $\lambda$ and $K$ as chosen above. Then, the conclusion of part (i) follows from Lemma \ref{unificationlemma}.

	Next, we prove (ii). 
	Let $\rin\in(0,3^{-1}\wedge \max\supp\Lambda)$, and $C,\lambda>0$ be such that
	for all $y \in (0,1)$, 
	$\E_y[e^{\lambda\waitint{}}]\leq C(y^{-1/4}+(1-y)^{-1/4})$, 
	the existence of these parameters is ensured by Proposition~\ref{prop:bddprobreturn}
	with $\hat{y}=\check{y}=y$. 
	Let $\rbs\in (0,\rin)$, $\tstop>0$ and $c_{\rin}\in (0,1)$ be from Proposition~\ref{prop:renewaloccbound}. 
	Set $\tin\coloneqq\tstop$. \modifbis{Let $y \in (0,1)$ and $\waitint{n},\waitbds{n}$ be the stopping times defined in~\eqref{eq:hittingtimes} with $\widehat{Y}_0= \widecheck{Y}_0=y$ (and $\rin, \rbs, \tin$ as we just chose)}.
	Proposition~\ref{prop:nthboundaryvisit} delivers $K>0$, not depending on the choice of $y$, such that, 
	\begin{align}
		\E_y[e^{\lambda\waitbds{n}}]\leq (y^{-1/4}+(1-y)^{-1/4})K^n. \label{nthboundaryvisitspec}
	\end{align}
	Now, we verify that the assumptions of Lemma~\ref{unificationlemma} are satisfied for $T:=\renewal$, $(T^1_n)_{n \geq 0}:=(\waitbds{n})_{n \geq 0}$ and $(T^2_n)_{n \geq 0}:=(\waitint{n})_{n \geq 0}$. Condition (i) is clearly satisfied. Proposition~\ref{prop:renewaloccbound} ensures that Condition (ii) holds with $c:=c_{\rin}$. We get from \eqref{nthboundaryvisitspec} that, almost surely, all stopping times $T^1_n$ are finite, and therefore all stopping times $T^2_n$ as well so Condition (iii) is satisfied. We also get from \eqref{nthboundaryvisitspec} that Condition (iv) is satisfied with $a:=y^{-1/4}+(1-y)^{-1/4}$ and $\lambda$ and $K$ as chosen above. Then, the conclusion easily follows from Lemma \ref{unificationlemma}.
		
	Finally, we prove part (iii). 
	Let $f(x)\coloneqq\E_\cV[\int_0^{\renewal} \ind{Y_s\leq x}\dd s]/\E_\cV[\renewal]$. 
	By part (ii), 
	$\E_\cV[\renewal]<\infty$,
	so $f$ is well-defined and bounded. 
	By Fubini $f(x)=\int_0^\infty h_s(x)\dd s/\E_\cV[\renewal],$ where
	for $s\geq 0$, $h_s(x)\coloneqq \P_\cV(s\leq \renewal, Y_s\leq x)$.
	Since $h_s$ is the distribution function of the measure $\P_\cV(s\leq \renewal, Y_s\in \dd x)$, 
	$h_s$ is continuous if and only if $\P_\cV(s\leq \renewal, Y_s=x)=0$ for all $x$. 
	Let $F_\cV$ be the distribution function of $\cV$, which is continuous. 
	Using the Feller property for $X$ and the Siegmund duality, we deduce that, for any $x \in (0,1)$ and $\varepsilon \in (0,x\wedge(1-x))$,  
	\begin{align*}
		&\P_\cV(Y_s=x)  = \int_{(0,1)} \P_y(Y_s=x) \cV(\dd y) \leq \int_{(0,1)} (\P_y(Y_s \leq x+\varepsilon) - \P_y(Y_s \leq x-\varepsilon)) \cV(\dd y) \\
		& =\int_{(0,1)} (\P_{x+\varepsilon}(y \leq X_s) - \P_{x-\varepsilon}(y \leq X_s)) \cV(\dd y) = \E_{x+\varepsilon}[F_\cV(X_s)] - \E_{x-\varepsilon}[F_\cV(X_s)] \xrightarrow{\varepsilon\to 0 }0. 
	\end{align*}
	A similar argument holds for $x=0$ and $x=1$, so $\P_\cV(Y_s=x)=0$ for all $x \in [0,1]$. Hence $h_s$ is continuous and, by dominated convergence, $f$ is continuous. Also, $h_s(0)=0$ and $h_s(1)=\P_{\cV}(s\leq \renewal)$ for all $s\geq 0$, so $f(0)=0$ and $f(1)=1$. It is plain that $f$ is non-decreasing. If $f(\hat{x})=f(\check{x})$ for some $\hat{x}, \check{x}$ with $0\leq \hat{x} \leq \check{x} \leq 1$, then $f(\check{x})-f(\hat{x})=\E_\cV[\int_0^{\renewal} \ind{Y_s\in (\hat{x},\check{x}]}\dd s]=0$. This implies $\P_\cV(\int_0^{\infty}\ind{Y_s\in (\hat{x},\check{x})}\dd s>0)=0$, since $\renewal$ is a renewal time for $Y$. Since $Y$ is open-set recurrent, we deduce that $\hat{x}=\check{x}$. Thus, $f$ is strictly increasing. 
\end{proof}
\begin{remark}
One may wonder if the almost sure merging of two trajectories, implied by Theorem~\ref{thm:probmergeYt}(i), also occurs in the case $(\Lambda,\mu,\sel)\in \Theta_i$ for $i\neq 2$. In these cases, the trajectories of $Y$ are absorbed at at least one boundary point, and the rate of jumps that cause a merging may decay very rapidly toward $0$, so the probability of two trajectories never merging can be positive. 
\end{remark}
\begin{figure}[t]
	
	\scalebox{3.5}{
}
	\caption{Sketch of two trajectories of~$Y$ and the sequences $(\waitbds{k})_{k\in \N}$ and $(\waitint{k})_{k\in \N}$. Since $\widehat{Y}_{\waitbds{3}}, \widecheck{Y}_{\waitbds{3}}\in (\rin,1-\rin)$, we have $\waitbds{3}=\waitint{3}$.}
	\label{fig:bothboundaries}
\end{figure}
\subsection{Proof of Proposition~\ref{prop:mergeprob}: Merging trajectories in the interior}\label{sec:accbdd:sub:controlinside}
The proof of Proposition~\ref{prop:mergeprob} involves estimates whose derivation is technical and rather delicate. This will require some notation.
Throughout this subsection, we let $\rin\in (0,1/3\wedge \maxsupp{\Lambda})$ and $\rbs=\rin/(1+\rin)$. 
Recall the definitions of $\Tint$ and $\Tstr$ just after \eqref{eq:hittingtimes}. Recall that $J_N$ is the set of jump times of~$N$. For $t\in J_N$, $(t,R_t, U_t)$ denotes the corresponding jump of $N$. 
For $c\in (0,1)$ and $k\in \N$, 
define 
$$\tu{c}_0\coloneqq 0\quad\textrm{and}\quad\tu{c}_k\coloneqq \inf\{t> \tu{c}_{k-1}:  t\in J_N\text{ with }R_t> c\},$$ 
with the convention $\inf\emptyset=\infty$. 
If $\tu{c}_k<\infty$, 
let $\ru{c}_k$ and $\uu{c}_k$ be such that $(\tu{c}_k,\ru{c}_k,\uu{c}_k)\in N$. 
To lighten the notation, we write $\tu{c}$ instead of $\tu{c}_1$ 
(and $\ru{c}$, $\uu{c}$ instead of $\ru{c}_1$, $\uu{c}_1$).

For $y\in (0,1)$ and $(t,r)\in \R_+\times (0,1)$, 
define $$\underline{{\rm m}}(y,r)\coloneqq \frac{y-r}{1-r}\ \text{ and }\ \overline{{\rm m}}(y,r)\coloneqq \frac{y}{1-r},$$ 
so that $\mr_{r,u}(y)=\median{\underline{{\rm m}}(y,r),u, \overline{{\rm m}}(y,r)}$. 
For later reference, note that for $0<r_1<r_2<1$, 
\begin{equation}
	\underline{{\rm m}}(y,r_1)>\underline{{\rm m}}(y,r_2)\quad\text{and}\quad \overline{{\rm m}}(y,r_1)<\overline{{\rm m}}(y,r_2). \label{eq:yuylmono}
\end{equation}
Using \eqref{eq:yuylmono}, we infer that 
\[\tu{\rbs}\leq \wct\Rightarrow\mro{\widehat{Y}}{\tu{\rbs}}{\ru{\rbs}}> \rin\quad\text{and}\quad\mru{\widecheck{Y}}{\tu{\rbs}}{\ru{\rbs}}< 1-\rin.\] 
Recall also that we are assuming that $\widehat{Y}_0=\hat{y}\leq\check{y}= \widecheck{Y}_0$. In particular, by the monotonicity property \eqref{eq:monotinictyY} of $Y$, we have 
\begin{equation}
	\tu{\rbs}\leq \wct\ \text{and}\ \uu{\rbs}\in(\rin,1-\rin) \ \Rightarrow\ \rin < \widehat{Y}_{\tu{\rbs}} \leq \widecheck{Y}_{\tu{\rbs}} <1-\rin \ \text{and}\ \tu{\rbs}<\wct. \label{eq:yuylbound}
\end{equation} 
So a jump larger than $\rbs$ with a $u$-component in the $\rin$-interior does not lead to the $\rbs$-boundary strip.
Finally, we introduce a set of parameters. For $n\in \N$, define \begin{equation}
	b_n\coloneqq1-(1-\rbs)^{n/2}, \ \delta_n\coloneqq(\rin-\rbs)\wedge \frac{(b_{n+2}-b_{n+1})}{2}\,\, \text{ and }\,\, t_n\coloneqq1\wedge {\left(\frac{\delta_n}{4\capco{\star}}\right)^2},\label{eq:mergeparametersboth}
\end{equation}
where $\capco{\star}$ is the constant in Lemma~\ref{lem:capdeviation}. Set for $v,w\in[0,1]$ with $v\leq w$, 
\[N_0(v,w)\coloneqq\left\lfloor 1+\frac{2\log(1-w+v)}{\log(1-\rbs)}\right\rfloor\quad\text{and}\quad N_0\coloneqq N_0(\hat{y},\check{y}).\]
The intuition behind these parameters is as follows. We will consider events such that $N_0+1$ large jumps of~$N$ occur in the $\rin$-interior, where `large' is parametrized by~$\rbs$. At the $k$th such jump the distance of the two trajectories will be bounded by~$b_{N_0-k+1}$, the incremental time between the $k$th and the $(k+1)$th jump will be bounded by $t_{N_0-k+1}$, and $\delta_{N_0-k+1}$ bounds the deviation of a trajectory from its starting point between these jumps. 

\begin{proof}[Proof of Proposition \ref{prop:mergeprob}]
	For $k\in \{1,\ldots,N_0\}$, 
	consider the event $$A_k\coloneqq \{\tu{\rbs}_k<\waitbdsu\wedge (\tu{\rbs}_{k-1} +t_{N_0-k+1}),\, \rin\leq \widehat{Y}_{\tu{\rbs}_k}\leq \widecheck{Y}_{\tu{\rbs}_k}\leq 1-\rin,\, \widecheck{Y}_{\tu{\rbs}_k}-\widehat{Y}_{\tu{\rbs}_k}<b_{N_0-k+1}\}.$$
	Let $\repsb\in (\rbs,\maxsupp{\Lambda})$ and 
	set $$\varsigma_\star\coloneqq \inf\{t> \tu{\rbs}_{N_0} : \ t\in J_N\,\&\,R_t > \repsb\}\quad\textrm{and}\quad \modifbis{A_{N_0+1}\coloneqq\{\varsigma_{\star}<\waitbdsu\wedge (\tu{\rbs}_{N_0} +t_{0}),\, \widecheck{Y}_{\varsigma_{\star}}=\widehat{Y}_{\varsigma_\star}\}},$$
	\mofe{Set $\tstop \coloneqq \sum_{i=0}^{N_0(\rin,1-\rin)} t_i$ and note that 
	$$\bigcap_{k=1}^{N_0+1}A_k\subset \{\widehat{Y}_{\tstop\, \wedge\, \wct}=\widecheck{Y}_{\tstop\, \wedge\, \wct}\}.$$} 
	
	Since $Y$ satisfies the strong Markov property (see Proposition~\ref{fullgenerator}), 
	we have 
\mofe{	\begin{align*}
		\textstyle &\P_{\hat{y},\check{y}}\big(	\widehat{Y}_{\tstop\, \wedge\, \wct}=\widecheck{Y}_{\tstop\, \wedge\, \wct}\big)\textstyle\geq  \P_{\hat{y},\check{y}}\big(\bigcap_{k=1}^{N_0+1} A_k\big)\geq \P_{\hat{y},\check{y}}\big(A_1\big)\, \prod_{k=2}^{N_0+1}\P_{\hat{y},\check{y}}\big(A_k\mid \bigcap_{i=1}^{k-1} A_i\big).
	\end{align*}}
	We show in the subsequent lemma that there are positive constants $p_0,p_1,p_2,\ldots$ such that \begin{equation}
		\P_{\hat{y},\check{y}}(A_1)\geq p_{N_0},\text{ and for }k\in \{2,\ldots,N_0+1\},\ \P_{\hat{y},\check{y}}(A_k\mid \textstyle\bigcap_{i=1}^{k-1} A_i)\geq p_{N_0-k+1}.\label{eq:contentprop}
	\end{equation}
	Since $N_0\leq N_0(\rin,1-\rin)<\infty$, the result follows by setting $c_{\rin}\coloneqq \prod_{i=0}^{N_0(\rin,1-\rin)}p_i.$
\end{proof}
It remains to prove \eqref{eq:contentprop}, which we do in the following lemma. 

\begin{lemme}\label{lem:aux_mergingeventestimate}
	Let $(\Lambda,\mu,\sel)\in\Theta$. For any $n\geq 1$, there exists $p_n>0$ such that for any $\hat{y},\check{y}\in[\rin,1-\rin]$ with $\hat{y}\leq \check{y}$ and $\check{y}-\hat{y}< b_{n+1}$,
	\begin{equation}\label{eq:auxeventestimate}
		\P_{\hat{y},\check{y}}(\tu{\rbs}<t_n\wedge \waitbdsu,\ \rin\leq \widehat{Y}_{\tu{\rbs}}\leq \widecheck{Y}_{\tu{\rbs}}\leq 1-\rin,\ \widecheck{Y}_{\tu{\rbs}}-\widehat{Y}_{\tu{\rbs}}<b_n)\geq p_n. 
	\end{equation}
	Moreover, for $\repsb\in (\rbs,\maxsupp{\Lambda})$, there is $p_0>0$ such that for any $\hat{y},\check{y}\in[\rin,1-\rin]$ with $\hat{y}\leq \check{y}$ and $\check{y}-\hat{y}< b_{1}$,
	\begin{equation}\label{eq:auxeventestimate2}
		\P_{\hat{y},\check{y}}(\tu{\repsb}<t_0\wedge \waitbdsu,\ \widehat{Y}_{\tu{\repsb}}=\widecheck{Y}_{\tu{\repsb}})\geq p_0.
	\end{equation}
\end{lemme}
\begin{proof}
	First, we prove~\eqref{eq:auxeventestimate}. 
	Fix~$n\in \N$ and let 
	\[A_n\coloneqq\{\tu{\rbs}<t_n\wedge \waitbdsu,\ \rin\leq \widehat{Y}_{\tu{\rbs}}\leq \widecheck{Y}_{\tu{\rbs}}\leq 1-\rin,\ \widecheck{Y}_{\tu{\rbs}}-\widehat{Y}_{\tu{\rbs}}<b_n\}.\] 
	Our aim is to find a lower bound for $\P_{\hat{y},\check{y}}(A_n)$, 
	which is uniform on $\hat{y},\check{y}\in[\rin,1-\rin]$ with $\check{y}-\hat{y}< b_{n+1}$. 
	To this end, consider the following events
	\begin{align*}
		\mathcal{S}_1&\coloneqq \{\rin\leq \mru{\widecheck{Y}}{\tur}{\rbs},\mro{\widehat{Y}}{\tur}{\rbs}\leq 1-\rin\},\quad \mathcal{S}_2\coloneqq \{\mru{\widecheck{Y}}{\tur}{\rbs}<\rin\}, \\
		\mathcal{S}_3&\coloneqq \{1-\rin<\mro{\widehat{Y}}{\tur}{\rbs}\}.
	\end{align*}	
	These events are disjoint and exhaustive, 
	because $\mru{\widecheck{Y}}{\tur}{\rbs}<\rin$ and $\mro{\widehat{Y}}{\tur}{\rbs}> 1-\rin$ cannot simultaneously hold. Indeed, if this were the case, \mofe{since $\rin=\rbs/(1-\rbs)$, we would have by the monotonicity property \eqref{eq:monotinictyY} of $Y$ that } $$1-\rin<\mro{\widehat{Y}}{\tur}{\rbs}\leq \mro{\widecheck{Y}}{\tur}{\rbs}=\mru{\widecheck{Y}}{\tur}{\rbs}+\rbs/(1-\rbs)< 2\rin,$$ which contradicts $\rin<1/3$.

	Set $M(\tur)\coloneqq (\mru{\widecheck{Y}}{\tur}{\rbs}+\mro{\widehat{Y}}{\tur}{\rbs})/2$. 
	For~$n\in \N$, define the random interval 
	$$I_n\coloneqq (\rin,1-\rin) \cap \begin{cases}
		\left(M(\tur)-b_n/2,M(\tur)+b_n/2\right), &\text{on } \mathcal{S}_1,\\
		\left(\mro{\widehat{Y}}{\tur}{\rbs},\mro{\widehat{Y}}{\tur}{\rbs}+b_n\right), &\text{on }\mathcal{S}_2,\\
		\left(\mru{\widecheck{Y}}{\tur}{\rbs}-b_n,\mru{\widecheck{Y}}{\tur}{\rbs} \right),&\text{on } \mathcal{S}_3.
	\end{cases}$$
	We now claim that 
	\begin{equation}\label{cl1}
		\mathcal{E}_n\coloneqq \{\tur\leq \waitbdsu,\, \tur< t_n,\, \widecheck{Y}_{\tur-}-\widehat{Y}_{\tur-}\leq b_{n+2},\, \uur\in I_n\}\subset A_n.
	\end{equation}
	To prove this intermediate statement, note first that ~\eqref{eq:yuylbound} implies that $$\mathcal{E}_n\subset\{\tur< t_n\wedge \waitbdsu,\, \rin\leq \widehat{Y}_{\tur}\leq \widecheck{Y}_{\tur}\leq 1-\rin\}.$$ 
	So to prove \eqref{cl1},
	it only remains to show that on $\mathcal{E}_n$, $\widecheck{Y}_{\tur}-\widehat{Y}_{\tur}<b_n$. 
	
	On $\mathcal{E}_n\cap\mathcal{S}_1$, $I_n=(M(\tur)-b_n/2,M(\tur)+b_n/2)\cap (\rin,1-\rin)$. 
	Note that by~\eqref{eq:yuylmono} and the definitions of $\mathcal{E}_n$ and $b_{n+2}$, 
	we have on $\mathcal{E}_n\cap\mathcal{S}_1$,
	\begin{align*}
		\mro{\widehat{Y}}{\tur}{\rur}>\mro{\widehat{Y}}{\tur}{\rbs}&=M(\tur)-\frac{\widecheck{Y}_{\tur-}-\widehat{Y}_{\tur-}-\rbs}{2(1-\rbs)}\\ &\geq M(\tur) - \frac{b_{n+2}-\rbs}{2(1-\rbs)}=M(\tur)-\frac{b_n}{2}.
	\end{align*}
	Similarly, $\mru{\widecheck{Y}}{\tur}{\rur}<M(\tur)+b_n/2$. 
	Moreover, on $\mathcal{E}_n\cap\mathcal{S}_1$ we have $\uur\in I_n$. Therefore, using the monotonicity property \eqref{eq:monotinictyY} of $Y$, we get that $M(\tur)-b_n/2<\widehat{Y}_{\tur}\leq \widecheck{Y}_{\tur}<M(\tur)+b_n/2$. 
	In particular, $\widecheck{Y}_{\tur}-\widehat{Y}_{\tur}<b_n$.
	
	On $\mathcal{E}_n\cap \mathcal{S}_2$, $I_n=(\mro{\widehat{Y}}{\tur}{\rbs},\mro{\widehat{Y}}{\tur}{\rbs}+b_n)\cap(\rin, 1-\rin)$. 
	Since here $\uur>\mro{\widehat{Y}}{\tur}{\rbs}$ and, by~\eqref{eq:yuylmono}, $\mro{\widehat{Y}}{\tur}{\rur}>\mro{\widehat{Y}}{\tur}{\rbs}$, 
	we have $\widehat{Y}_{\tur}> \mro{\widehat{Y}}{\tur}{\rbs}$. 
	Moreover, by~\eqref{eq:yuylmono}, the definition of~$\mathcal{S}_2$ and $\rbs$, 
	and because $\tur\leq \waitbdsu$
	\[\mru{\widecheck{Y}}{\tur}{\rur}<\mru{\widecheck{Y}}{\tur}{\rbs}<\rin=\frac{\rbs}{1-\rbs}\leq \mro{\widehat{Y}}{\tur}{\rbs}<\mro{\widehat{Y}}{\tur}{\rbs}+b_n.\]
	Since on $\mathcal{E}_n\cap \mathcal{S}_2$ also $\uur<\mro{\widehat{Y}}{\tur}{\rbs}+b_n$, 
	we deduce $\widecheck{Y}_{\tur}<\mro{\widehat{Y}}{\tur}{\rbs}+b_n.$ 
	By the monotonicity property \eqref{eq:monotinictyY} of $Y$,
	we have $\mro{\widehat{Y}}{\tur}{\rbs}< \widehat{Y}_{\tur}\leq \widecheck{Y}_{\tur}< \mro{\widehat{Y}}{\tur}{\rbs}+b_n$. 
	In particular, $\widecheck{Y}_{\tur}-\widehat{Y}_{\tur}<b_n$. 
	In a similar way, one derives that also on $\mathcal{E}_n\cap \mathcal{S}_3$, $\widecheck{Y}_{\tur}-\widehat{Y}_{\tur}<b_n$. 
	Altogether we have shown~\eqref{cl1}.

	Recall that $Y^{\rbs}$ is~$Y$ capped by jumps $(t,r,u)$ of~$N$ for~$r>\rbs$. Now, we claim that
	\begin{equation}\label{cl2}
		\mathcal{E}_n^*\coloneqq\Big\{\ \tur< t_n,\, \sup_{s\in [0,t_n]}\lvert \widehat{Y}_s^{\rbs}-\hat{y}\lvert< \delta_n,\ \sup_{s\in [0,t_n]}\lvert \widecheck{Y}_s^{\rbs}-\check{y}\lvert< \delta_n,\, \uur\in I_n\Big\}\subset \mathcal{E}_n.
	\end{equation} 
	To prove this, use the definition of $\delta_n$ and that for $t< \tur$, $\widehat{Y}_t^{\rbs}=\widehat{Y}_t$ and $\widecheck{Y}_t^{\rbs}=\widecheck{Y}_t$ to obtain that  $$\mathcal{E}_n^*\subset\{\forall s\in [0,\tur), \rbs\leq \widehat{Y}_s\leq \widecheck{Y}_s\leq 1-\rbs,\ \widecheck{Y}_s-\widehat{Y}_s\leq b_{n+2},\ \tur< t_n,\uur\in I_n \}.$$ It is not difficult to see that this set in turn is a subset of $\mathcal{E}_n.$ This proves \eqref{cl2}.

	It remains to find a lower bound for {$\P_{\hat{y},\check{y}}(\mathcal{E}_n^*)$}. 
	First we show that on $\mathcal{E}_n$, 
	\[\lvert I_n\rvert \geq \ell_n\coloneqq (1-3\rin)\wedge b_n/2\] 
	On $\mathcal{E}_n\cap\mathcal{S}_1$, 
	either ${\lvert I_n\rvert}=b_n$, ${\lvert I_n\rvert}=M(\tur)+b_n/2-\rin$, 
	${\lvert I_n\rvert}=1-\rin-M(\tur)+b_n/2$, 
	or ${\lvert I_n\rvert}=1-2\rin$. 
	Since $M(\tur)\in[\rin,1-\rin]$, 
	we deduce $\lvert I_n\rvert \geq \ell_n$ on $\mathcal{E}_n\cap\mathcal{S}_1$. 
	Since $\tur\leq \waitbdsu$, 
	we have on $\mathcal{E}_n\cap\mathcal{S}_2$ 
	\[\rin=\rbs/(1-\rbs)\leq \mro{\widehat{Y}}{\tur}{\rur}\leq \mru{\widecheck{Y}}{\tur}{\rur}+\rbs/(1-\rbs)<2\rin<1-\rin,\]
	where we have used that $\rin<1/3$. 
	Distinguishing whether $\mro{\widehat{Y}}{\tur}{\rur}+b_n\leq 1-\rin$ 
	or $\mro{\widehat{Y}}{\tur}{\rur}+b_n> 1-\rin$, 
	we respectively have $\lvert I_n\rvert\geq b_n$ 
	or $\lvert I_n\rvert\geq 1-3\rin$; 
	in both cases $\lvert I_n\rvert \geq \ell_n$. 
	The proof on $\mathcal{E}_n\cap\mathcal{S}_3$ is analogous.
	We thus have $\lvert I_n\rvert \geq \ell_n$ on $\mathcal{E}_n$. 
	
	Finally, using \eqref{cl1}, \eqref{cl2}, the independence of $\tur$, $\uur$, $(\widehat{Y}^{\rbs},\widecheck{Y}^{\rbs})$, and \mofe{that on $\mathcal{E}_n$ we have $\lvert I_n\rvert \geq \ell_n$, we get}
	\begin{align*}
		&{\P_{\hat{y},\check{y}}}(A_n)\geq \textstyle \P_{\hat{y},\check{y}}\left(\mathcal{E}_n^*\right)\textstyle\\&\geq \Big(1-\P_{\hat{y}}\Big(\sup_{s\in [0,t_n]}\lvert {\widehat{Y}^{\rbs}_s}-\hat{y}\lvert\geq \delta_n\Big)-\P_{\check{y}}\Big(\sup_{s\in [0,t_n]}\lvert {\widecheck{Y}^{\rbs}_s}-\check{y}\lvert\geq \delta_n\Big)\Big) \left(1-e^{-t_n\int_{(\rbs,1)}\frac{\Lambda(\dd r)}{r^2}}\right)\ell_n\textstyle\\
		&\geq \Big(1-e^{-t_n\int_{(\rbs,1)}\frac{\Lambda(\dd r)}{r^2}}\Big)\frac{\ell_n}{2}\eqqcolon p_n>0, \end{align*}
	where in the last step we used Lemma~\ref{lem:capdeviation} and the definition of $t_n$. 
	This completes the proof of the first part of the proposition.
		
	The second part of the proposition is proven along the same lines. Define 
	{\begin{align*}
			A_0 & \coloneqq\{\tu{\repsb}<t_0\wedge \waitbdsu,\ \widehat{Y}_{\tu{\repsb}}=\widecheck{Y}_{\tu{\repsb}}\},\quad I_0\coloneqq(\mru{\widecheck{Y}}{\turb}{\repsb},\mro{\widehat{Y}}{\turb}{\repsb})\cap(\rin,1-\rin), \\
			\mathcal{E}_0 & \coloneqq\{\turb\leq \waitbdsu,\, \turb< t_0,\, \widecheck{Y}_{\turb-}-\widehat{Y}_{\turb-}\leq b_2,\, \uurb\in I_0\}. 
	\end{align*}}
	$I_0$ is well-defined and non-empty if $\widecheck{Y}_{\turb-}-\widehat{Y}_{\turb-}\leq b_2$. Note that {on $\mathcal{E}_0$}, 
	\begin{align*}
		\mro{\widecheck{Y}}{\turb}{\rurb}&\geq \mro{\widehat{Y}}{\turb}{\rurb}> {\mro{\widehat{Y}}{\turb}{\repsb}}\\
		&{> \uurb > \mru{\widecheck{Y}}{\turb}{\repsb}}>\mru{\widecheck{Y}}{\turb}{\rurb}\geq \mru{\widehat{Y}}{\turb}{\rurb}.
	\end{align*}
	Hence, {on $\mathcal{E}_0$}, $\widehat{Y}_{\turb}=\widecheck{Y}_{\turb}=\uurb$. {This together with \eqref{eq:yuylbound} yields $\mathcal{E}_0 \subset A_0$}. 
	Proceeding as in the the proof of \eqref{cl2}, one can show that	
	\[\mathcal{E}_0^*\coloneqq\Big\{\turb< t_0,\, \sup_{s\in [0,t_0]}\lvert \widehat{Y}_s^{\repsb}-\hat{y}\lvert< \delta_0,\ \sup_{s\in [0,t_0]}\lvert \widecheck{Y}_s^{\repsb}-\check{y}\lvert< \delta_0,\, {\uurb}\in I_0\Big\}\subset \mathcal{E}_0.\] 
	Note that on $\mathcal{E}_0$, $\mro{\widehat{Y}}{\turb}{\repsb}>\rbs/(1-\rbs)=\rin$ and $\mru{\widecheck{Y}}{\turb}{\repsb}<(1-2\rbs)/(1-\rbs)=1-\rin$. Hence, on $\mathcal{E}_0$, we have the following four possibilities: 
	\begin{align*}
		\text{(i)}\quad \lvert I_0\rvert&= \mro{\widehat{Y}}{\turb}{\repsb}-\mru{\widecheck{Y}}{\turb}{\repsb}\geq (\repsb-\rbs)/(1-\repsb),\\
		\text{(ii)}\quad\lvert I_0\rvert& =\mro{\widehat{Y}}{\turb}{\repsb}-\rin = \mro{\widehat{Y}}{\turb}{\repsb}-\mro{\widehat{Y}}{\turb}{\rbs}+\mro{\widehat{Y}}{\turb}{\rbs}-\rin,\\ 
		&\geq \mro{\widehat{Y}}{\turb}{\repsb}-\mro{\widehat{Y}}{\turb}{\rbs}\geq {\rin (\repsb-\rbs)/(1-\repsb)},\\
		\text{(iii)}\quad\lvert I_0\rvert &= 1-\rin-\mru{\widecheck{Y}}{\turb}{\repsb} \geq {\rin (\repsb-\rbs)/(1-\repsb)}\quad \text{(which is deduced as~(ii))},\\
		\text{(iv)}\quad\lvert I_0\rvert &= 1-2\rin.
	\end{align*}
	This shows that, on $\mathcal{E}_0$, $\lvert I_0\rvert\geq l_0\coloneqq (1-2\rin)\wedge \rin(\repsb-\rbs)/(1-\repsb)$. The final part of the proof is analogous to the first case.
\end{proof}
\subsection{Proof of Proposition~\ref{prop:renewaloccbound}: Timely occurrence of renewal time}\label{sec:accbdd:sub:representation}
First, we state a lemma that will be relevant for the proof of Proposition~\ref{prop:renewaloccbound}. We then use it to prove Proposition~\ref{prop:renewaloccbound}.

Recall from the definition of~$\renewal$ in \eqref{eq:defrenewaltime} that $\kappa\in (0,\maxsupp{\Lambda})$ and $\eta > 0$ are chosen such that  $\theta=(\kappa+\eta)/(1+\eta) < \maxsupp{\Lambda}$. 
As in the previous section, we require parameters that play a similar role as the parameters in~\eqref{eq:mergeparametersboth}. 
Since their use will be only local, 
we use the same notation as in~\eqref{eq:mergeparametersboth}, 
even if the values differ. For $n\in \N$, 
\begin{equation}\label{para3}
	b_n\coloneqq \frac{(1-\kappa)^{n/2}}{2}, \qquad \delta_n\coloneqq (b_{n+1}-b_{n+2})\wedge (b_{n-1}-b_n),\quad t_n\coloneqq 1\wedge {\frac{\delta_n^2}{4\capco{\star}^2}}.
\end{equation}
complemented by $b_{-1}=(1+\kappa)/2$, 
where $\capco{\star}$ is as in Lemma~\ref{lem:capdeviation}. 
For $y\in (0,1)$, 
let \[N_0(y)\coloneqq \lfloor 2\log(2y)/\log(1-\kappa)\rfloor.\]
For later reference, we note that with this choice \begin{equation}
	b_{N_0(y)+1}<y\leq b_{N_0(y)}.\label{eq:propN0}
\end{equation}
\begin{lemme}\label{lem:renewalauxeventbound}
	Let $(\Lambda,\mu,\sel)\in\Theta$. For any $n\geq 1$, there is $p_n>0$ such that for all $y\in [b_{n+1},b_n]$ and for $\rbs=b_{n+2}$ (so the boundary strips depends on $n$), \begin{equation}\label{eq:renewalbound}
		\P_y(\turk<t_n\wedge \waitbdsu \text{ and }Y_{\turk}\in (b_n,b_{n-1}))\geq p_n.
	\end{equation}
	Moreover, there is $p_0>0$ such that for all $y\in [b_1,b_0]$ and for $\rbs=\frac{1-\kappa}{2}\wedge \frac{1-\eta}{2}$,
	\begin{equation}\label{eq:renewalbound2}
		\P_y(\renewal< t_0\wedge \waitbdsu)\geq p_0.
	\end{equation}
\end{lemme}
\begin{proof}
	First, we prove~\eqref{eq:renewalbound}.
	Let $n\geq 1$ and $\rbs=b_{n+2}$.
	Recall that \[Y_{\turk}=\median{\Yups{\turk}{\rurk},\Ylows{\turk}{\rurk},\uurk}.\]
	If $Y_s\in (b_{n+2},b_{n-1})$ for $s\in [0,\turk)$, 
	then $Y_{\turk-}\geq b_{n+2}$ so that $\Yups{\turk}{\rurk}> b_{n+2}/(1-\kappa )= b_n$, 
	and $Y_{\turk-}\leq b_{n-1}$ so that $\Ylows{\turk}{\rurk}< Y_{\turk-}\leq b_{n-1}$. 
	If in addition $\uurk\in (b_n,b_{n-1})$, 
	we get $Y_{\turk}\in (b_n,b_{n-1})$ so $\turk<\waitbdsu$. 
	We thus have
	\begin{align*}
		A_n\coloneqq &\{Y_{\turk}\in (b_n,b_{n-1}),\ \turk< t_n\wedge \waitbdsu\}\\
		\supseteq &\big\{\forall s\in [0,\turk),\, Y_s\in (b_{n+2},b_{n-1}),\,  \turk<t_n,\, \uurk \in (b_n,b_{n-1})\big\}\eqqcolon B_n.
	\end{align*}
	Moreover, by the definition of $\delta_n$ and~$Y^{\kappa }$ 
	(see Lemma~\ref{lem:capdeviation}), 
	we have
	\begin{align*}
	B_n\supseteq\Big\{\sup_{s\in [0,t_n]} \lvert Y_s^{\kappa }-y\rvert<\delta_n,\ \turk<t_n,\ \uurk\in (b_n,b_{n-1})\Big\}\eqqcolon C_n.
	\end{align*}
	Using this, Lemma~\ref{lem:capdeviation} with $c=\kappa$, the independence of $Y^\kappa$, $\uurk$ and $\turk$, and the definition of $t_n$, we obtain
	\begin{eqnarray}
		\P_{y}(A_n)& \geq {\P_{y}(C_n)}\geq \textstyle \Big(1-{\capco{\star}}\frac{\sqrt{t_n}}{\delta_n} \Big)\Big(1-\exp\big(-t_n\int_{(\kappa,1)} r^{-2}\Lambda(\dd r)\big)\Big) (b_{n-1}-b_n) \nonumber\\
		&\geq \textstyle \frac{1}{2}\Big(1-\exp\big(-t_n\int_{(\kappa,1)} r^{-2}\Lambda(\dd r)\big)\Big) (b_{n-1}-b_n)\eqqcolon p_n>0.\label{eq:renewalbound_aux}
	\end{eqnarray}
	This completes the proof of~\eqref{eq:renewalbound}. 

	Leu us now prove~\eqref{eq:renewalbound2}. To this end, set $\rbs=\frac{1-\kappa}{2}\wedge \frac{1-\eta}{2}$ in the definition of $\waitbdsu$. 
	We have 
	\begin{align*}
		&\Big\{\sup_{s\in [0,t_0]}\lvert Y_s^{\theta}-y\lvert<\delta_0,\ \turt<t_0,\ \uurt\in \Big(\frac{1-\eta}{2},\frac{1+\eta}{2}\Big)  \Big\}\\
		&\subseteq \{ b_2\leq Y_{\turt-}\leq b_{-1},\ \turt\leq \waitbdsu,\ \turt<t_0,\ \uurt\in \Big(\frac{1-\eta}{2},\frac{1+\eta}{2}\Big) \}\subseteq \{\renewal< t_0\wedge \waitbdsu\}.
	\end{align*}                                    	
	In the last inclusion we used that $\turt=\renewal$, which follows from the definitions of~$\renewal$, $b_2$ and $b_{-1}$. We also used that, since $Y_{\renewal} \in ((1-\eta)/2,(1+\eta)/2)$, the inequality $\renewal\leq \waitbdsu$ is actually strict. The proof is achieved proceeding as in~\eqref{eq:renewalbound_aux}.
\end{proof}

Now, we have all the ingredients to prove Proposition~\ref{prop:renewaloccbound}.

\begin{proof}[Proof of Proposition~\ref{prop:renewaloccbound}]
Let $y\in [\rin,1/2]$ and set $\rbs=b_{N_0(\rin)+2}\wedge (1-\eta)/2$ in the definition of $\waitbdsu$. 
For $j\in\{1,\ldots,N_0(y)\}$, 
define $$B_j\coloneqq \{\turk_j<\waitbdsu\wedge (\turk_{j-1}+t_{N_0(y)-j+1} ),\ Y_{\turk_{j}}\in (b_{N_0(y)-j+1}, b_{N_0(y)-j}) \},$$ 
complemented by $B_{N_0(y)+1} \coloneqq \{\renewal<\waitbdsu\wedge (\turk_{N_0(y)}+ t_0)\}$. 
Set $\tstop\coloneqq \sum_{i=0}^{N_0(\rin)}t_i $. 
Note that $N_0(\rin)\geq N_0(y)$ for all $y\in [\rin,1/2]$. 
Recall also from~\eqref{eq:propN0} that $b_{N_0(y)+1}< y\leq  b_{N_0(y)}$. 
Hence, 
$$\bigcap_{j=1}^{N_0(y)+1} B_j\subset \{\renewal<\tstop\wedge \waitbdsu\}.$$
Finally, 
$$\P_y(\renewal<\tstop\wedge \waitbdsu )\geq \P_y(B_1)\prod_{j=2}^{N_0(y)+1}\P_y(B_j\mid \cap_{i=1}^{j-1} B_i).$$
Note that, since $N_0(y)\leq N_0(\rin)$, 
we have for all $j\in \{1,\ldots,\N_0(y)+1\}$ that $b_{N_0(y)-j+3}\geq b_{N_0(\rin)+2}\geq \repsek$ 
so $\waitbdsu \geq T_Y([0,b_{N_0(y)-j+3})\cup(1-b_{N_0(y)-j+3},1])$.
This allows to use Lemma~\ref{lem:renewalauxeventbound} to bound the above probabilities. 
For $j\in \{2,\ldots,\N_0(y)+1\}$, 
using the strong Markov property at time $\turk_{j-1}$ and Lemma~\ref{lem:renewalauxeventbound}, 
we have $\P_y(B_j\mid \cap_{i=1}^{j-1}B_i)\geq p_{N_0(y)-j+1}$. 
Lemma~\ref{lem:renewalauxeventbound} yields $\P_y(B_1)\geq p_{N_0(y)}$. 
Setting $c\coloneqq \prod_{i=0}^{N_0(\rin)} p_i$ yields the result for $y\in [\rin,1/2]$. 
The result for $y\in [1/2,1-\rin]$ is obtained using the previous arguments to the process~$1-Y$.
\end{proof}
\begin{remark}\label{rem:modificationsrenewal}
Let $a\in(0,1)$. Let $\kappa < (2a\, \maxsupp{\Lambda})\wedge 4a^2((1/a)-1)$, $\theta = (\eta + \kappa)/(2a+\eta)$, and set $\renewal^{(a)}\coloneqq \inf \{ T\in J_N: \  Y_{T-} \in \left [a-\kappa/2,a+\kappa/2 \right ],\ R_T\in (\theta,1) \text{ and }U_T\in\left [a-\eta/2,a+\eta/2\right ]\}.$
{Arguing as in Section \ref{sec:main:siegmund}, one can see that $Y_{\renewal^{(a)}}$ has a uniform distribution on $[a-\eta/2,a+\eta/2 ]$}. The results in this section involving $\renewal$ can be easily generalized to hold for $\renewal^{(a)}$ (note that $\renewal=\renewal^{(1/2)}$). The main changes concern the parameters $(b_n)$. More precisely, we need to set $b_{-1} \coloneqq a+\kappa/2$, $b_0\coloneqq a$, and for $n\geq 1$, $b_n \coloneqq(2a - \kappa)(1-\kappa)^{(n-2)/2}/2$. (The extra condition $\kappa < 4a^2((1/a)-1)$ ensures $b_1 < b_0$.) Moreover, $N_0(y)$ has to be adapted so that $b_{N_0(y)+1} < y \leq b_{N_0(y)}$. The definition of $(t_n)_{n\in \N}$, $(\delta_n)_{n\in \N}$, and $\repsek$ remain unchanged.
\end{remark}

\subsection{Proof of Proposition \ref{prop:bddprobreturn}: Time to the interior}\label{sec:accbdd:sub:bddstrip}
First, we require a bound on the distribution tail of the time $Y$ spends in the $\varepsilon$-boundary strip around~$0$ and~$1$, respectively, on $[0,t]$. Having that estimate at hand, \mofe{we will prove } Proposition \ref{prop:bddprobreturn}.

It will be convenient to use the following notation. For a stochastic process $Z$ with \cadlag paths and $[c,d]\subset[0,1]$, we write 
$\Gamma_{Z}(t,[c,d])\coloneqq\int_{[0,t]} \1_{\{Z_s\in [c,d]\}}\dd s$, i.e. the time~$Z$ spends in $[c,d]$ up to time~$t$.

\begin{lemme}\label{lem:bddstripocc}
Let $(\Lambda,\mu,\sigma)\in \Theta$ and $b\in\{0,1\}$. If $C_b(\Lambda,\mu,\sigma)<0$, then there exist  $\varepsilon<1/3\wedge\maxsupp{\Lambda}$ and $c_1,c_2>0$ such that for any $y\in (0,1)$ and $t\geq 0$, \[\P_y\Big(\Gamma_{Y}(t,[b(1-\varepsilon),\varepsilon+b(1-\varepsilon)])>t/2\Big)\leq c_1 \lvert b-y\rvert^{-1/4}e^{-c_2t}.\]
\end{lemme}
\begin{proof}
We prove only the case $b=0$; the proof for the case $b=1$ is analogous.
	Recall that for $A\subset[0,1]$, $T_{Y}A= \inf\{t\geq 0: Y_t\in A\}$.  
	Let $a\in (0,1/2)$ and $\gamma>0$ be such that~$\E_y[e^{\gamma T_Y(a,1]}]\leq 2((a/y)\vee 1)^{1/4}$ for any $y\in (0,1)$; 
	existence $a$ and $\gamma$ is ensured by Proposition~\ref{lem:escapetimeY}. 
	Let $a'\in(0,a)$ and fix $y\in (0,1)$. 
	We will construct a process~$M$ starting at $M_0=y$, 
	being distributed as~$Y$, 
	and whose value is reset to $a'$ any time it surpasses the value $a$. 
	Let $(\Omega,\mathcal{F},\P)$ be a probability space on which an i.i.d. family $(S^{(i)}, N^{(i)})_{i\in \N_0}$ of Poisson random measures is defined 
	such that for all $i\in \N_0$, $S^{(i)}\sim S$ and $N^{(i)}\sim N$. 
	For all $i\in \N$, let $Y^{(i)}$ and $\overline{Y}^{(i)}$ be the strong solutions to~\eqref{eq:SDE_Y} with $S$ and $N$ replaced by $S^{(i)}$ and $N^{(i)}$, respectively, 
	and with $Y_0^{(i)}=\ind{i>0}a'+\ind{i=0}y$ and $\overline{Y}_0^{(i)}=\ind{i>0}\overline{Y}_{\tau_i - \tau_{i-1}}^{(i-1)}+\ind{i=0}y$, 
	where $\tau_0\coloneqq 0$ and 
	for $i\in \N$, $\tau_{i+1}\coloneqq \inf\{t\geq \tau_{i}: Y^{(i)}_{t-\tau_{i}}>a\}.$ 
	Set for $t\geq 0$, 
	$M_t\coloneqq \sum_{k=0}^{\infty}  \ind{\tau_k \leq t < \tau_{k+1}} Y^{(k)}_{t-\tau_k}$ and $\fray_t \coloneqq \sum_{k=0}^{\infty}  \ind{\tau_k \leq t < \tau_{k+1}} \overline{Y}^{(k)}_{t-\tau_k}$. 
	Using the strong Markov property at $\tau_i$, 
	we deduce that $\fray$ is equal in law to $Y$ under $\P_y$ and that $(M_t)_{t\geq 0}$ has the desired properties. 
	Moreover, from the strong Markov property and the monotonicity property \eqref{eq:monotinictyY} of $Y$, 
	it follows easily that $M_t\leq \fray_t$ for all~$t\geq 0$. 
	
	\smallskip
	
	Set $m\coloneqq -2/\log(\E_{a'}[e^{-T_Y(a,1]}])$ and 
	let $\Gamma^{(i)}\coloneqq \int_{[\tau_i,\tau_{i+1})}1_{\{M_s\leq\varepsilon\}}\dd s$, 
	where $\varepsilon\ll1$ will be specified later. 
	Then,
	\mofe{\begin{align*}
		\P_y(\textstyle \Gamma_{Y}(t,[0,\varepsilon])>\frac{t}{2})&\leq \P(\textstyle\Gamma_{M}(t,[0,\varepsilon])>\frac{t}{2})\leq \P(\Gamma_{M}(\tau_{\lfloor m t\rfloor},[0,\varepsilon])>\frac{t}{2})+\P(\tau_{\lfloor m t\rfloor}< t)\\
		&\leq \P(\textstyle \sum_{i=0}^{\lfloor mt \rfloor -1} \Gamma^{(i)}>\frac{t}{2})+\P(\tau_{\lfloor m t\rfloor}< t).
	\end{align*}}
We now show that \begin{align*}
		\text{(i) }\P(\tau_{\lfloor m t\rfloor}< t)\leq \E_{a'}\left[e^{-T_Y(a,1]}\right]^{-2}e^{-t} \text{ and (ii) } \P\left(\textstyle \sum_{i=0}^{\lfloor mt \rfloor -1} \Gamma^{(i)}>\frac{t}{2}\right)\leq 2y^{-1/4}e^{-\gamma t/4},
	\end{align*}
	which clearly implies the result. So let us first prove (i). By the Markov property of $M$ at times~$\tau_i$, it follows that $(\tau_{i+1}-\tau_i)_{i \geq 1}$ are i.i.d. and have the same law as $T_Y(a,1]$ under $\P_{a'}$. Thus, using Chernoff's inequality, \begin{align*}
		\P(\tau_{\lfloor mt \rfloor} <t)\leq e^{t}\E\left[e^{-\sum_{i=1}^{\lfloor mt \rfloor -1}\tau_{i+1}-\tau_i}\right]=e^{t}\E_{a'}\left[e^{-T_Y(a,1]}\right]^{{\lfloor mt \rfloor -1}}&\leq e^{t}\E_{a'}\left[e^{-T_Y(a,1]}\right]^{{ mt-2}}\\
		&= \E_{a'}\left[e^{-T_Y(a,1]}\right]^{{-2}}e^{-t}.
	\end{align*}
	This proves (i). Next, we prove (ii). Note first that $\Gamma^{(i)}$, $i\in \N_0$, are all independent and for $i\geq 1$, distributed as $\Gamma_Y(T_Y(a,1],[0,\varepsilon])$ under $\P_{a'}$. Thus, using Chernoff's inequality and Proposition~\ref{lem:escapetimeY},
	\mofe{\begin{align*}
		&\P\left(\Gamma^{(0)}+\textstyle \sum_{i=1}^{\lfloor mt \rfloor -1} \Gamma^{(i)}>\frac{t}{2}\right)\leq e^{-\frac{\gamma t}{2}}\,{\E\left[e^{\gamma\Gamma^{(0)}}\right]\,\prod_{i=1}^{\lfloor mt\rfloor -1}\E\left[e^{\gamma \Gamma^{(i)}}\right]}\\
		&\leq e^{-\frac{\gamma t}{2}}\E_y\left[e^{\gamma T_Y(a,1]}\right]\E_{a'}\left[e^{\gamma \Gamma_Y(T_Y(a,1],[0,\varepsilon] )}\right]^{mt} \!\!\!\!\leq 2\Big(\frac{a}{y}\vee 1\Big)^{\frac14}e^{-\frac{\gamma t}{2}}\,\E_{a'}\left[e^{\gamma \Gamma_Y(T_Y(a,1],[0,\varepsilon])}\right]^{mt}.
	\end{align*}}
	Note that $\Gamma_Y(T_Y(a,1],[0,\varepsilon])\leq T_Y(a,1]$. Thus, by Proposition~\ref{lem:escapetimeY}, $\E_{a'}[e^{\gamma \Gamma_Y(T_Y(a,1],[0,\varepsilon])}]$ \mofe{is finite}. By dominated convergence theorem, we conclude that $\E_{a'}[e^{\gamma \Gamma_Y(T_Y(a,1],[0,\varepsilon])}]\to1$ as $\varepsilon\to 0$. Therefore, we can choose $\varepsilon$ small enough such that $m\log(\E_{a'}[e^{\gamma \Gamma_Y(T_Y(a,1],[0,\varepsilon])}])<\frac{\gamma}{4}$. For this choice of $\varepsilon$, $e^{-\gamma t/2}\,\E_{a'}[e^{\gamma \Gamma_Y(T_Y(a,1],[0,\varepsilon])}]^{mt}\leq e^{-\gamma t/4}$ and (ii) follows.
\end{proof}
We can now prove Proposition \ref{prop:bddprobreturn}.
\begin{proof}[Proof of Proposition \ref{prop:bddprobreturn}]
	Fix $\hat{y},\check{y}\in (0,1)$ such that $\hat{y}\leq \check{y}$. 
	By Lemma~\ref{lem:bddstripocc} with~$b=0$, 
	there is $\hat{\varepsilon}$ and $\hat{c}_1,\hat{c}_2$ (independent of the choice of $\hat{y}, \check{y}$)
	such that $\P_{\hat{y}}(\Gamma_{\widehat{Y}}(t,[0,\hat{\varepsilon}])>t/2))\leq \hat{c}_1 \hat{y}^{-1/4} e^{-\hat{c}_2 t }$. 
	Similarly, by Lemma~\ref{lem:bddstripocc} with $b=1$, 
	there is $\check{\varepsilon}$ and $\check{c}_1,\check{c}_2$ (independent of $\hat{y}, \check{y}$)
	such that $\P_{\check{y}}(\Gamma_{\widecheck{Y}}(t,[1-\check{\varepsilon},1])>t/2)\leq \check{c}_1 (1-\check{y})^{-1/4} e^{-\check{c}_2 t }$. 
	Set $c_1\coloneqq \hat{c}_1\vee \check{c}_1$, 
	$c_2\coloneqq \hat{c}_2\wedge \check{c}_2$, 
	and \modifbis{$\rin\coloneqq\hat{\varepsilon}\wedge\check{\varepsilon}$}. 
	We claim that $\P_{\hat{y},\check{y}}(\waitint{}>t)\leq c_1e^{-c_2t}(\hat{y}^{-1/4}+(1-\check{y})^{-1/4})$. Assuming the claim is true, 
	we obtain for $\lambda=c_2/2$, 
	\mofe{\begin{align*}
		\E_{\hat{y},\check{y}}[e^{\lambda \waitint{}}]&=1+\!\!\int_0^\infty\!\!\lambda e^{\lambda s} \P_{\hat{y},\check{y}}(\waitint{}>s)\dd s
		\leq 1+\!\!\int_0^\infty\!\!\lambda c_1e^{-(c_2-\lambda)s}(\hat{y}^{-1/4}+(1-\check{y})^{-1/4} ) \dd s\\
		&\leq (1+c_1)(\hat{y}^{-1/4}+(1-\check{y})^{-1/4} ),
	\end{align*}}
	which proves the proposition. 
	It remains to prove the claim. 
	By the monotonicity property \eqref{eq:monotinictyY} of $Y$, 
	if there is a time~$s\in[0,t]$ such that $\widehat{Y}_s>\hat{\varepsilon}$ and $\widecheck{Y}_s<1-\check{\varepsilon}$, then $\waitint{}\leq t$. 
	Hence, 
	$$\big\{\Gamma_{\widehat{Y}}(t,[0,\hat{\varepsilon}])<t/2\big\} \cap \big\{\Gamma_{\widecheck{Y}}(t,[1-\check{\varepsilon},1])<t/2\big\}\subset \big\{ \waitint{}\leq t\big\}.$$
	Considering the \mofe{complementary } events and exploiting the properties of $\hat{c}_1,\hat{c}_2,\check{c}_1,\check{c}_2,c_1,c_2$ yields 
	\begin{align*}
		\P_{\hat{y},\check{y}}\big(\waitint{}> t\big) & \leq \P_{\hat{y}}\Big( \Gamma_{\widehat{Y}}(t,[0,\hat{\varepsilon}])> t/2 \Big) +\P_{\check{y}} \Big( \Gamma_{\widecheck{Y}}(t,[1-\check{\varepsilon},1])> t/2 \Big) \\
		&\leq \hat{c}_1 \hat{y}^{-1/4} e^{-\hat{c}_2 t } + \check{c}_1 (1-\check{y})^{-1/4} e^{-\check{c}_2 t }\leq  c_1e^{-c_2t}(\hat{y}^{-1/4}+(1-\check{y})^{-1/4}),
	\end{align*}
	which achieves the proof of the claim.
\end{proof}
\subsection{Proof of Proposition \ref{prop:nthboundaryvisit}: Accumulated time after $n$ steps} \label{procedure}
We now prove Proposition~\ref{prop:nthboundaryvisit}. The proof uses an estimate, which will be provided afterwards.

\begin{proof}[Proof of Proposition~\ref{prop:nthboundaryvisit}]
The result is trivially true for $n=0$. For $n\geq 1$, using that $\waitbds{n}-\waitint{n-1}\leq \tau$ and the strong Markov property, we get for all $0<\hat{y}\leq\check{y}<1$, 
	\begin{align*} 
		&{\E_{\hat{y},\check{y}}}\big[e^{\lambda \waitbds{n}}\big]={\E_{\hat{y},\check{y}}}\Big[e^{\lambda (\waitbds{n}-\waitint{n-1})}e^{\lambda (\waitint{n-1}-\waitbds{n-1})}e^{\lambda \waitbds{n-1}} \Big]\\ \leq& e^{\lambda \tau} {\E_{\hat{y},\check{y}}}\Big[e^{\lambda \waitbds{n-1}}{\E_{\hat{y},\check{y}}}\Big[e^{\lambda(\waitint{n-1}-\waitbds{n-1})}\mid \mathcal{F}_{\waitbds{n-1}}\Big]\Big]
		=e^{\lambda \tau} {\E_{\hat{y},\check{y}}}\Big[e^{\lambda \waitbds{n-1}}{\E_{\widehat{Y}_{\waitbds{n-1}},\widecheck{Y}_{\waitbds{n-1}}}\Big[e^{\lambda\waitint{\star}}\Big]}\Big],
	\end{align*}
where $\waitint{\star}$ is defined as $\wht$, 
but from an independent pair of coupled trajectories of~$Y$ 
started in $\widehat{Y}_{\waitbds{n-1}}$ and $\widecheck{Y}_{\waitbds{n-1}}$, respectively. 
Using Proposition~\ref{prop:bddprobreturn}, 
we obtain
\begin{eqnarray} 
{\E_{\hat{y},\check{y}}}\big[e^{\lambda \waitbds{n}}\big] \leq e^{\lambda \tau} C {\E_{\hat{y},\check{y}}}\Big[ e^{\lambda \waitbds{n-1}} \Big(\widehat{Y}_{\waitbds{n-1}}^{-1/4} +\big(1-\widecheck{Y}_{\waitbds{n-1}}\big)^{-1/4} \Big)\Big]. \label{eq:forlaterlem}
\end{eqnarray} 
Note that for $n=1$, 
this yields ${\E_{\hat{y},\check{y}}}[e^{\lambda \waitbds{1}}] \leq e^{\lambda \tau} C (\hat{y}^{-1/4} +(1-\check{y})^{-1/4} )$ \mofe{and the result holds in this case. Therefore, we assume now that } $n\geq 2$.

Recall that $J_N$ is the set of jump times of~$N$. 
If $\waitbds{n-1}\in J_N$, 
let $(\waitbds{n-1},R_{n-1},U_{n-1})\in N$ be the corresponding jump. 
We claim that 
\begin{eqnarray}
	&\widehat{Y}_{\waitbds{n-1}}^{-1/4}+\Big(1-\widecheck{Y}_{\waitbds{n-1}}\Big)^{-1/4}\leq 2\Big(\frac{\rbs}{2}\Big)^{-1/4}+U_{n-1}^{-1/4} \1_{\{\waitbds{n-1}\in J_N,\widehat{Y}_{\waitbds{n-1}}<\rbs\}}\nonumber \\
	&\hspace{5cm}\qquad+(1-U_{n-1})^{-1/4}\1_{\{\waitbds{n-1}\in J_N,\widecheck{Y}_{\waitbds{n-1}}>1-\rbs\}},\label{eq:wctdecompclaim1}\\
	&{\E_{\hat{y},\check{y}}}\Big[e^{\lambda\waitbds{n-1}}U_{n-1}^{-1/4}\1_{\{\waitbds{n-1}\in J_N, \widehat{Y}_{\waitbds{n-1}}<\rbs\}} \Big]\leq \frac{4}{3}\rbs^{-1/4}{\E_{\hat{y},\check{y}}}\Big[e^{\lambda\waitbds{n-1}}\Big], \label{eq:wctdecompclaim2} \\
	&{\E_{\hat{y},\check{y}}}\Big[e^{\lambda\waitbds{n-1}}(1-U_{n-1})^{-1/4}\1_{\{\waitbds{n-1}\in J_N, \widecheck{Y}_{\waitbds{n-1}}>1-\rbs\}} \Big]\leq \frac{4}{3}\rbs^{-1/4}{\E_{\hat{y},\check{y}}}\Big[e^{\lambda\waitbds{n-1}}\Big]. \label{eq:wctdecompclaim3}
\end{eqnarray}
Assume~\eqref{eq:wctdecompclaim1}--\eqref{eq:wctdecompclaim3} are true. Then, using \eqref{eq:wctdecompclaim1}--\eqref{eq:wctdecompclaim3}, ${\E_{\hat{y},\check{y}}}[e^{\lambda \waitbds{n}}]\leq C \tilde{C} e^{\lambda\tau} {\E_{\hat{y},\check{y}}}[e^{\lambda \waitbds{n-1}}]$, where $\tilde{C}\coloneqq (2(\rbs/2)^{-1/4}+8\rbs^{-1/4}/3)$. \mofe{Iterating this inequality}, we obtain $${\E_{\hat{y},\check{y}}}[e^{\lambda \waitbds{n}}]\leq (\hat{y}^{-1/4}+(1-\check{y})^{-1/4})\tilde{C}^{n-1}(Ce^{\lambda\tau})^{n}.$$ Since $\tilde{C}\geq 1$, the result follows by setting {$K\coloneqq C\tilde{C}e^{\lambda }$}. It remains to prove \eqref{eq:wctdecompclaim1}--\eqref{eq:wctdecompclaim3}.

First we deal with~\eqref{eq:wctdecompclaim1}. Let $J_{S+}$ (resp. $J_{S-}$) be the set of jump times of~$S$ with positive (resp. negative) $r$-component. We consider four cases.
\begin{enumerate} \item $\waitbds{n-1}\notin J_N\cup J_{S+}\cup J_{S-}$: We have $\widehat{Y}_{\waitbds{n-1}}=\widehat{Y}_{\waitbds{n-1}-}\geq \rbs$ and $1-\widecheck{Y}_{\waitbds{n-1}}=1-\widecheck{Y}_{\waitbds{n-1}-}\geq \rbs$, and \eqref{eq:wctdecompclaim1} holds.
	\item $\waitbds{n-1}\in J_N$: Here, $$\widehat{Y}_{\waitbds{n-1}}=\median{\mru{\widehat{Y}}{\waitbds{n-1}}{R_{n-1}},U_{n-1},\mro{\widehat{Y}}{\waitbds{n-1}}{R_{n-1}} }.$$ Note that $\mro{\widehat{Y}}{\waitbds{n-1}}{R_{n-1}}> \widehat{Y}_{\waitbds{n-1}-}\geq \rbs$. Therefore, if $\widehat{Y}_{\waitbds{n-1}}<\rbs$, $U_{n-1}< \rbs <1-\rbs$ and $\widehat{Y}_{\waitbds{n-1}}= \mru{\widehat{Y}}{\waitbds{n-1}}{R_{n-1}}\vee U_{n-1}\geq U_{n-1}$. Since $\mru{\widecheck{Y}}{\waitbds{n-1}}{R_{n-1}}<\widecheck{Y}_{\waitbds{n-1}-}\leq 1-\rbs$, we also have $\widecheck{Y}_{\waitbds{n-1}}<1-\rbs$. If $\widecheck{Y}_{\waitbds{n-1}}>1-\rbs$, one shows similarly that $\widecheck{Y}_{\waitbds{n-1}}\leq U_{n-1}$ and $\widehat{Y}_{\waitbds{n-1}}\geq \rbs$. Hence, \eqref{eq:wctdecompclaim1} holds also in this case. 
	\item $\waitbds{n-1}\in J_{S+}$: By Lemma~\ref{lem:estimatesel}, \eqref{eq:estimateselpos}, $\widehat{Y}_{\waitbds{n-1}}\geq \widehat{Y}_{\waitbds{n-1}-}/2\geq \rbs/2$ and $1-\widecheck{Y}_{\waitbds{n-1}}\geq 1-\widecheck{Y}_{\waitbds{n-1}-}\geq \rbs$, and \eqref{eq:wctdecompclaim1} holds.
	\item $\waitbds{n-1}\in J_{S-}$: By Lemma~\ref{lem:estimatesel}, \eqref{eq:estimateselneg}, $\widehat{Y}_{\waitbds{n-1}}\geq \widehat{Y}_{\waitbds{n-1}-}\geq \rbs$ and $1-\widecheck{Y}_{\waitbds{n-1}}\geq (1-\widecheck{Y}_{\waitbds{n-1}-})/2\geq \rbs/2$. This completes the proof of~\eqref{eq:wctdecompclaim1}.
\end{enumerate}

Next, we prove~\eqref{eq:wctdecompclaim2}. 
To this end, use the strong Markov property at time $\waitint{n-2}$ to get
\begin{align*}
	&\E_{\hat{y},\check{y}}\big[e^{\lambda\waitbds{n-1}}U_{n-1}^{-1/4}\1_{\{\waitbds{n-1}\in J_N, \widehat{Y}_{\waitbds{n-1}}<\rbs\}} \big]\\&={\E_{\hat{y},\check{y}}}\Big[e^{\lambda\waitint{n-2}}{\E_{\hat{y},\check{y}}}\Big[ e^{\lambda(\waitbds{n-1}-\waitint{n-2}) }U_{n-1}^{-1/4}\1_{\{\waitbds{n-1}\in J_N, \widehat{Y}_{\waitbds{n-1}}<\rbs\}}\mid \mathcal{F}_{\waitint{n-2}}\Big]\Big]\\
	&={\E_{\hat{y},\check{y}}}\Big[e^{\lambda\waitint{n-2}}{\E_{\widehat{Y}_{\waitint{n-2}},\widecheck{Y}_{\waitint{n-2}}}\Big[ e^{\lambda\wctn{\star}} U_{\wctn{\star}}^{-1/4}\1_{\{\wctn{\star}\in J_N\cap[0,\tau], \hat{\bar{Y}}_{\wctn{\star}}<\rbs\}}\Big]\Big]},
\end{align*}
where $\wctn{\star}$ is defined as $\waitbdsu$ 
but from an independent pair of coupled trajectories $\hat{\bar{Y}}$ and $\check{\bar{Y}}$ of~$Y$ 
with $\hat{\bar{Y}}_0=\widehat{Y}_{\waitint{n-2}}$ and $\check{\bar{Y}}_0=\widecheck{Y}_{\waitint{n-2}}$. 
Using Lemma~\ref{lem:boundjumpevent} (below) with $\gamma=1/4$ and the Markov property, 
we obtain 
\begin{align*}
	&\E_{\hat{y},\check{y}}\Big[e^{\lambda\waitbds{n-1}}U_{n-1}^{-1/4}\1_{\{\waitbds{n-1}\in J_N, \widehat{Y}_{\waitbds{n-1}}<\rbs\}} \Big]\leq\frac{4}{3}\rbs^{-1/4}{\E_{\hat{y},\check{y}}}\Big[e^{\lambda \waitint{n-2}}{\E_{\widehat{Y}_{\waitint{n-2}},\widecheck{Y}_{\waitint{n-2}}}}\Big[e^{\lambda(\tau\wedge \wctn{\star})}\Big]\Big]\\
	&=\frac{4}{3}\rbs^{-1/4}{\E_{\hat{y},\check{y}}}\Big[e^{\lambda \waitint{n-2}}{\E_{\hat{y},\check{y}}}\Big[e^{\lambda(\waitbds{n-1}-\waitint{n-2})}\mid \mathcal{F}_{\waitint{n-2}}\Big]\Big]= \frac{4}{3}\rbs^{-1/4}{\E_{\hat{y},\check{y}}}\Big[e^{\lambda \waitbds{n-1}}\Big].
\end{align*}
This proves~\eqref{eq:wctdecompclaim2}. The proof for \eqref{eq:wctdecompclaim3} is analogous.\end{proof}
\begin{lemme}\label{lem:boundjumpevent}
	Let $(\Lambda,\mu,\sel)\in\Theta$, $\lambda>0$, $\gamma\in (0,1)$, $\rin < (1/3)\wedge \maxsupp{\Lambda}$, $\tau>0$, $\rbs\in (0,\rin)$, and $\rin \leq \hat{y}\leq \check{y}\leq 1-\rin$. Then, 
	\begin{eqnarray}
		\E_{\hat{y},\check{y}}\Big[ e^{\lambda\waitbdsu} U_{\waitbdsu}^{-\gamma}\1_{\{\waitbdsu\in J_N\cap[0,\tau], \widehat{Y}_{\waitbdsu}<\rbs\}}\Big]\leq \frac{\rbs^{-\gamma}}{1-\gamma}\E_{\hat{y},\check{y}}\big[e^{\lambda(\tau\wedge \waitbdsu)}\big], \label{eq:boundjumpevent1}\\
		\E_{\hat{y},\check{y}}\Big[ e^{\lambda\waitbdsu} (1-U_{\waitbdsu})^{-\gamma}\1_{\{\waitbdsu\in J_N\cap[0,\tau], \widecheck{Y}_{\waitbdsu}>1-\rbs\}}\Big]\leq \frac{\rbs^{-\gamma}}{1-\gamma}\E_{\hat{y},\check{y}}\big[e^{\lambda(\tau\wedge \waitbdsu)}\big]. \label{eq:boundjumpevent2}
	\end{eqnarray}
\end{lemme}
\begin{proof}
	We only prove~\eqref{eq:boundjumpevent1}; 
	the proof of \eqref{eq:boundjumpevent2} is analogous. 
We define a sequence of stopping times that, on the event $\{\widehat{Y}_{\waitbdsu}<\rbs\}$, contains $\waitbdsu$. For $\eta\in (0,1)$, set $L_0^{0,\eta}\coloneqq 0 $ and 
	\begin{align*}
		L_n^{0,\eta}&\coloneqq \inf\{t> L_{n-1}^{0,\eta}: t\in J_N \text{ with }R_t>\eta \quad\text{ and }\quad \mru{\widehat{Y}}{t}{R_t}<\rbs\}.
	\end{align*} Note that $(L_n^{0,\eta})_{n\geq 0}$ is well-defined 
	since the set of jumping times of~$N$ with $r$-component larger than $\eta$ is discrete. 
	For $T\in J_N$ (the jump times of~$N$), 
	let $(T,R_T,U_T)$ be the corresponding jump. 
	Note that if $\waitbdsu\in J_N$ with $R_{\waitbdsu}>\eta$ and $\widehat{Y}_{\waitbdsu}<\rbs$, 
	then $\mru{\widehat{Y}}{\waitbdsu}{R_{\waitbdsu}}<\rbs$. 
	This follows from the definition of $\widehat{Y}_{\waitbdsu}$ and $\mro{\widehat{Y}}{\waitbdsu}{R_{\waitbdsu}}>\widehat{Y}_{\waitbdsu-}\geq \rbs$. 
	In particular, $\waitbdsu=L_n^{0,\eta}$ for some {$n\geq 1$ and $\eta>0$ 
	if $\waitbdsu\in J_N$ and $\widehat{Y}_{\waitbdsu}<\rbs$.}

	Fix $n\geq 1$. 
	If $L_n^{0,\eta}\leq \waitbdsu$, 
	then $\widehat{Y}_{L_n^{0,\eta}-}\geq \rbs$, 
	and thus $\mro{\widehat{Y}}{L_n^{0,\eta}}{R_{L_n^{0,\eta}}}>\rbs.$ 
	In addition, by the definition of $L_n^{0,\eta}$, 
	we have $\mru{\widehat{Y}}{L_n^{0,\eta}}{R_{L_n^{0,\eta}}} <\rbs$. 
	Hence, if $L_n^{0,\eta}\leq \waitbdsu$, by the definition of $\widehat{Y}_{L_n^{0,\eta}}$, 
	we have $U_{L_n^{0,\eta}}<\rbs$ if and only if $\widehat{Y}_{L_n^{0,\eta}}<\rbs$; 
	and the latter implies $L_n^{0,\eta}=\waitbdsu$. 
	Moreover, $U_{L_n^{0,\eta}}$ is independent of $\mathcal{F}_{L_n^{0,\eta}-}$. 
	Using this, we get 
	\begin{align*}
		&\E_{\hat{y},\check{y}}\Big[ e^{\lambda\waitbdsu} U_{\waitbdsu}^{-\gamma} \ind{\waitbdsu\in J_n\cap[0,\tau],R_{\waitbdsu}>\eta,\widehat{Y}_{\waitbdsu}<\rbs}\Big] =\sum_{n\geq 1} \E_{\hat{y},\check{y}}\Big[e^{\lambda L_{n}^{0,\eta}}U_{L_{n}^{0,\eta}}^{-\gamma} \ind{\tau\geq \waitbdsu=L_{n}^{0,\eta}, \widehat{Y}_{L_{n}^{0,\eta}}<\rbs}\Big]\\
		&=\sum_{n\geq 1} \E_{\hat{y},\check{y}}\Big[e^{\lambda L_{n}^{0,\eta}} \ind{\tau\wedge \waitbdsu\geq L_{n}^{0,\eta}}\E_{\hat{y},\check{y}}\big[U_{L_{n}^{0,\eta}}^{-\gamma} \ind{U_{L_{n}^{0,\eta}}<\rbs}\mid \mathcal{F}_{L_n^{0,\eta}-}\big]\Big] \\
		&=\sum_{n\geq 1} \E_{\hat{y},\check{y}}\Big[e^{\lambda L_{n}^{0,\eta}} \ind{\tau\wedge \waitbdsu\geq L_{n}^{0,\eta}}\frac{\rbs^{-\gamma}}{1-\gamma} \E_{\hat{y},\check{y}}\big[\ind{U_{L_{n}^{0,\eta}}<\rbs}\mid \mathcal{F}_{L_n^{0,\eta}-}\big]\Big]\\
		&=\frac{\rbs^{-\gamma}}{1-\gamma} \sum_{n\geq 1} \E_{\hat{y},\check{y}}\Big[e^{\lambda L_{n}^{0,\eta}} \ind{\tau\geq \waitbdsu=L_{n}^{0,\eta}, \widehat{Y}_{L_{n}^{0,\eta}}<\rbs}\Big]\leq \frac{\rbs^{-\gamma}}{1-\gamma}\E_{\hat{y},\check{y}}[e^{\lambda(\tau\wedge \waitbdsu)}].
	\end{align*}
	Letting $\eta\to 0$ we obtain~\eqref{eq:boundjumpevent1} by monotone convergence. 
\end{proof}

\section{The case \texorpdfstring{$\Theta_0\cup\Theta_1$}{Theta0 and Theta1}: Proof of Theorem \ref{thm:stuckatboundary}}\label{sec:asextinction}
In this section, we prove Theorem \ref{thm:stuckatboundary}. 
We first lay out the proof idea in the case $\Theta_1$ stating three propositions, 
which we then directly use to prove the theorem. 
The propositions are then proved in the remainder of the section. 
The case $\Theta_0$ follows in a straightforward way from the result in $\Theta_1$.

The main idea is to show that~$Y$ gets trapped at the boundary strip around~$0$ sufficiently fast. 
\mofe{To prove this we will: } 
1) \mofe{Prove that } the boundary strip can indeed trap $Y$ if started close enough to the boundary, and estimate what \mofe{one can interpret as an \emph{escape time} } in case the trapping fails. 
2) \mofe{Prove that } the process rapidly reaches the boundary strip from outside the strip. 
3) Iteratively combine 1) and 2) to estimate the trapping time of $Y$. 

For step 1), \mofe{the fact that the } boundary can trap $Y$ is justified below (see the discussion after \eqref{eq:Yaux}) based on the results from Section \ref{sec:comparisonlevy}. 
\mofe{The estimation of the escape time has essentially already been } carried out in Section \ref{sec:asextinction:sub:prepext}.

\mofe{The following proposition takes care of the step 2)}.
\begin{prop}[Time to boundary strip]\label{prop:boundtimeto0extinction}
	Let $(\Lambda,\mu,\sel)\in\Theta$ and assume $C_1(\Lambda,\mu,\sigma)<0$. Let $a\in (1/2,1)$ and $\gamma>0$ as provided by Proposition~\ref{lem:escapetimeY}--(2). 
	For each $\varepsilon\in (0,1-a)$, there are constants $\zeta,M>0$ such that for any $y\in (0,1)$, $$\E_y[e^{\zeta T_Y[0,\varepsilon)}]\leq M(1-y)^{-1/4}.$$
\end{prop}

Finally, step 3) \mofe{builds on Lemma \ref{unificationlemma}. For this we will define two } interlaced sequences of stopping times and study, in Propositions \ref{prop:controltimetobdd} and \ref{prop:controltimetobddstrong}, the amount of time accumulated after $k$ sojourns in the trapping region, which will allow us to check the assumptions of Lemma~\ref{unificationlemma}. The proofs of those propositions rely on the results from the two previous steps.

We now construct a new process $\fray$ that has the same distribution as~$Y$ 
and that alternates between the boundary strip around~$0$ and the interior 
until it ultimately becomes trapped at the boundary strip.
Fix $b> 0$ such that $\E[\hat{L}_1^b]>0$ (existence of $b$ follows from Lemma \ref{lem:existence_suit_approx} with $m=0$). \modifbis{We choose $\varepsilon\in (0,e^{-b})$ such that Proposition~\ref{prop:boundtimeto0extinction} holds for some $M$ and $\zeta$ for the boundary strip $[0,\varepsilon/2)$, that is, $\E_y[e^{\zeta T_Y[0,\varepsilon/2)}]\leq M(1-y)^{-1/4}$}.
Let $(\Omega,\mathcal{F},\P)$ be a probability space 
on which we define an i.i.d. sequence of Poisson random measures $(N^{n,i},S^{n,i})_{n\in \N, i\in \{1,2\}}$ 
with $N^{n,i}\sim N$ and $S^{n,i}\sim S$, $n\in\N$, $i\in\{1,2\}$. 
For $n\in \N$, let $\hL^{n}$ be the copy of the L{\'e}vy process $\hat{L}^b$ with $N$ and $S$ replaced by $N^{n,2}$ and $S^{n,2}$, respectively. 
Similarly, let $Y^{n,i}$ be the solution of SDE~\eqref{eq:SDE_Y} with $N$ and $S$ replaced by $N^{n,i}$ and $S^{n,i}$, respectively; 
their initial values and the time frames where they will be considered are specified iteratively as follows.
Define $\wctny{0}\coloneqq 0$, $\whtnl{0}\coloneqq 0$, $Y_0^{1,1}=Y_0^{(1)}=y\in (0,1)$, and for $n\in \N$, 
\begin{equation}\label{eq:definition_fray}
	\begin{split}
	&\wctny{n}\coloneqq \inf \Big\{t\geq \whtnl{n-1}:\ Y_{t-\whtnl{n-1}}^{n,1} <\varepsilon/2\Big\}, 
	\qquad\qquad\qquad\,\,\, Y_0^{n,2}=\hL_0^{(n)}\coloneqq Y_{\wctny{n}-\whtnl{n-1}}^{n,1},\\
	&\whtnl{n}\coloneqq \inf\Big\{t\geq \wctny{n}: \ \hL_0^{(n)}\exp\big(-\hL_{t-\wctny{n}}^{n}\big)>\varepsilon\Big\}, \qquad Y_0^{n+1,1}=Y_0^{(n+1)}\coloneqq Y_{\whtnl{n}-\wctny{n}}^{n,2}.
\end{split}
\end{equation}
By the definition of $\whtnl{n}$ and Lemma \ref{lem:existence_suit_approx}, for $n\geq 1$, 
\begin{eqnarray}
{\P\Big(\whtnl{n}=\infty \mid \whtnl{1}<\infty,...,\whtnl{n-1}<\infty \Big)}&=\P\Big(\inf_{t\in[0,\infty)} \hL_t^{n}>\log(\hL_0^{(n)}/\varepsilon) \Big)\nonumber \\
&\geq \P\Big(\inf_{t\in [0,\infty)} \hat{L}^b_t>-\log(2) \Big)>0. \label{minoprobastaybelow}
\end{eqnarray}
Define $\fray=(\fray_t)_{t\geq 0}$ via
\begin{equation}
	\fray_t\coloneqq  \begin{cases}
		Y_{t-\whtnl{n-1}}^{n,1},&\text{if }t\in [\whtnl{n-1},\wctny{n}),\,n\in\N,\\
		Y_{t-\wctny{n}}^{n,2},& \text{if }t\in [\wctny{n}, \whtnl{n}),n\in\N.\end{cases} \label{eq:Yaux}
\end{equation}
\begin{figure}[t]
	
	\scalebox{3.5}{
 }
	\caption{Illustration of path segments related to~$\fray$.}
	\label{fig:gettingstuck}
\end{figure}

See Fig.~\ref{fig:gettingstuck} for an illustration. From the strong Markov property at $\whtnl{n}$ and $\wctny{n}$, we deduce that $\fray$ is equal in law to $Y$ under $\P_y$. 
Let $\Nstuck\coloneqq\min \{n\geq 1:\ \whtnl{n}=\infty\}$. 
Note from \eqref{minoprobastaybelow} that $\Nstuck<\infty$ a.s. and, from Lemma~\ref{lem:levysandwich}, that $\fray_t\leq \varepsilon$ for all $t\geq \wctny{\Nstuck}$.

We now see that Proposition \ref{prop:boundtimeto0extinction} allows us to control the random variables $\wctny{k}-\whtnl{k-1}$, while Lemma~\ref{lem:lowerlevyhittingtime} and Corollary~\ref{coro:lowerlevyhittingtime} allow us to control the random variables $(\whtnl{k}-\wctny{k}) \ind{\whtnl{k}<\infty}$. Combining \mofe{these } results leads to the following two propositions that control the time it takes $\fray$ to visit $[0,\varepsilon/2)$ for the $k$th time, depending on the integrability condition.
\begin{prop}\label{prop:controltimetobdd}
	Assume that $(\Lambda,\mu,\sigma)\in \Theta_1$ and that $\rW_\gamma(r^{-2}\Lambda(\dd r))<\infty$ and $\rW_\gamma(\bar{\mu})<\infty$ for some $\gamma>0$. 
	Then for any $\alpha\in (0,\gamma)$ there is a constant $C\geq 1$ 
	such that for all $k\geq 1$, $t>0$ and $y\in(0,1)$, if $Y_0^{(1)}=y$, 
	we have $$\P(\wctny{k}>t, \Nstuck \geq k)\leq Ck^{1+\alpha}(1-y)^{-1/4}t^{-\alpha}.$$
	In particular, $\wctny{k}<\infty$ a.s. on $\{\Nstuck\geq k\}$.
\end{prop}
\begin{prop}\label{prop:controltimetobddstrong}
	Assume that $(\Lambda,\mu,\sigma)\in \Theta_1$ and that $s_\gamma(r^{-2}\Lambda(\dd r))<\infty$ and $s_\gamma(\bar{\mu})<\infty$ for some $\gamma>0$. 
	Then there are $\lambda>0$ and $C\geq 1$ such that for all $y\in(0,1)$ and $k\geq 1$, 
	we have for $Y_0^{(1)}=y$ $$\E[e^{\lambda \wctny{k}}\ind{\Nstuck\geq k}]\leq (1-y)^{-1/4}C^k.$$ 
	\modifbis{In particular, $\wctny{k}<\infty$ a.s. on $\{\Nstuck\geq k\}$}.
\end{prop}

Finally, we piece everything together and prove Theorem~\ref{thm:stuckatboundary}.
\begin{proof}[Proof of Theorem \ref{thm:stuckatboundary}]
	Without loss of generality, we prove the result only for small~$\varepsilon$, 
	because if $\varepsilon' \geq \varepsilon$, 
	then $\P_y(Y_t>\varepsilon') \leq \P_y(Y_t>\varepsilon)$. 
	So if \eqref{survYexpodecay} (resp. \eqref{survYpoldecay}) is true for some small $\varepsilon$, 
	then it is true for all $\varepsilon' \in [\varepsilon,1)$. 
	Fix $b> 0$ such that $\E[\hat{L}_1^b]>0$ (\mofe{the } existence of $b$ follows from Lemma \ref{lem:existence_suit_approx} with $m=0$). 
	By Proposition~\ref{prop:boundtimeto0extinction}, we can choose $\varepsilon\in (0,e^{-b})$ so that for some $\zeta,M>0$, 
	\begin{equation}\label{ezm}
		\E_y[e^{\zeta T_Y[0,\varepsilon/2)}]\leq M(1-y)^{-1/4}.
	\end{equation}
	Let $y\in (0,1)$ (resp. $y\in (0,1)$ and $\alpha\in (0,\gamma))$. 
	Under this choice of $\varepsilon$ and $y$, 
	let $(\wctny{n})_{n\geq0}$ and $(\whtnl{n})_{n\geq 0}$ be defined as in~\eqref{eq:definition_fray} 
	and recall that if $\wctny{\Nstuck}<\infty$, 
	then $\fray_t\leq \varepsilon $ for all $t\geq \wctny{\Nstuck}$, \modifbis{where $(\fray_t)_{t\geq 0}$ is defined in \eqref{eq:Yaux}}. 
	We claim that there exists $K_1,K_2> 0$ (resp. $K>0$), independent of our choice of $y$, 
	such that for any $t\geq 0$, 
	\begin{equation}\label{eq:stuckboundaryaux}
		\P(\wctny{\Nstuck}>t)\leq (1-y)^{-1/4} K_1e^{-K_2t} \quad \Big(\text{resp. }\P(\wctny{\Nstuck}>t)\leq (1-y)^{-1/4} Kt^{-\alpha}\Big).
	\end{equation}
	If the claim is true, then \eqref{survYpoldecay} and \eqref{survYexpodecay} follow, because \[\P_y(Y_t>\varepsilon)=\P(\fray_t>\varepsilon)\leq \P(\wctny{\Nstuck}>t).\]

	To verify the claim, we check the assumptions of Lemma \ref{unificationlemma} 
	for $T\coloneqq \wctny{\Nstuck}$, $(T^1_n)_{n \geq 0}:=(\wctny{n})_{n\geq0}$ and $(T^2_n)_{n \geq 0}:=(\whtnl{n})_{n\geq 0}$. 
	Condition (i) is clearly satisfied.
	Condition (iii) is satisfied by the definition of $\Nstuck$
	and by $T\coloneqq\wctny{\Nstuck}$. 
	We then see from \eqref{minoprobastaybelow} that Condition (ii) holds with $c:=\P(\inf_{t\in [0,\infty)} \hat{L}^b_t>-\log(2))$. 
	If for some $\gamma>0$, $s_\gamma(r^{-2}\Lambda(\dd r))<\infty$ and $s_\gamma(\bar{\mu})<\infty$, then by Proposition~\ref{prop:controltimetobddstrong}, 
	there is $\lambda>0$ and $C\geq 1$ (not depending on the choice of $y$) 
	such that for all $k \geq 1$, 
	$\E[e^{\lambda \wctny{k}}\ind{\Nstuck \geq k}]\leq (1-y)^{-1/4}C^k$. 
	Without loss of generality we assume $C>1$. 
	Thus, Condition (iv) is satisfied with $a:=(1-y)^{-1/4}$ and $\lambda$ as above and $K:=C$.
	If for some $\gamma>0$, $\rW_\gamma(r^{-2}\Lambda(\dd r))<\infty$ and $\rW_\gamma(\bar{\mu})<\infty$, 
	let $\alpha'\in (\alpha,\gamma)$ and let $C'\geq 1$ be such that $\P(\wctny{k}>t,N_0\geq k)\leq C'k^{1+\alpha'}(1-y)^{-1/4}t^{-\alpha'}$; the existence of such~$C'$, that does not depend on the choice of $y$, is ensured by Proposition~\ref{prop:controltimetobdd}. 
	Using this, we get that Condition (iv)' is satisfied with $\ell:=\alpha'$ and $a:=C'(1-y)^{-1/4}$. 
    Thus, in both cases, we can apply Lemma \ref{unificationlemma}. 
	In particular, $\wctny{\Nstuck}<\infty$ almost surely so, by the discussion in the beginning of the proof, 
	we have $\mathbb{P}_y$-almost surely $Y_t \in [0,\varepsilon]$ for $t$ sufficiently large. 
	Since $\varepsilon$ is arbitrary in $(0,e^{-b})$, the $\mathbb{P}_y$-almost sure convergence of $Y_t$ follows, ending the proof. 
\end{proof}
\subsection{Proof of Proposition~\ref{prop:boundtimeto0extinction}} \label{btttgctob}
The derivation of the bound for the exponential moments of the first entrance time into $[0,\varepsilon)$ is based on the following strategy. First, we establish a Doeblin-type condition to reach $[0,\varepsilon)$ before some finite time and before coming close to the boundary at $1$.
Second, we establish that from the boundary at~$1$ the process sufficiently fast returns to a state where Doeblin's condition may be used. Iterating these two steps then yields an exponential bound. We begin by establishing Doeblin's condition.
\begin{lemme}[Doeblin type condition]\label{lem:extinctiondoeblinlike}
	Let $(\Lambda,\mu,\sel)\in\Theta$. For any $\varepsilon\in (0,1/2)$ there is $t^\circ>0$ and $c\in (0,1)$ such that for any $y\in [\varepsilon,1-\varepsilon]$, \begin{equation}
		\P_y(T_Y[0,\varepsilon)<t^\circ\wedge T_Y(1-\varepsilon/2,1])\geq c.
	\end{equation}
\end{lemme}
The proof of the lemma is rather technical and based on ideas similar to the ones of Sections~\ref{sec:accbdd:sub:controlinside} and \ref{sec:accbdd:sub:representation}. 
First, we recall and introduce some notation. 
The following parameters play a similar role as the parameters in~\eqref{eq:mergeparametersboth} and \eqref{para3} 
and since their use will be only local, we use the same notation, 
even if the values are not the same. 
Fix $\varepsilon\in (0,1/2)$ and $\xi>0$ such that $2\xi<\max\{\maxsupp{\Lambda},(1-(1-\varepsilon)^2),1/2\}.$ For $n\in \Z$, set 

$$b_n\coloneqq 
	\left(1-\frac{(1/(1-2\xi))^{\frac{n}{2}}}{2}\right)\1_{\{n<\frac{2\log(2)}{\log(1/(1-2\xi))}\}},
	\quad \delta_n\coloneqq (b_{n+1}-b_{n+2})\wedge(b_{n-1}-b_n),
$$
and $t_n\coloneqq 1\wedge (\delta_n^2/4{\capco{\star}}^2)$, where {$\capco{\star}$ }is the constant given in Lemma~\ref{lem:capdeviation}. Let $n_0\coloneqq \max\{n\in \Z:\ b_n\neq 0\}$. In particular, $\lim_{n\to-\infty}b_n=1$, $b_n$ is strictly decreasing and positive while $n\leq n_0$. For $y\in (0,1)$, define $N_0(y)\coloneqq \lfloor 2\log(2(1-y))/\log(1/(1-2\xi))\rfloor,$ which can be negative. This implies that $b_{N_0(y)+1}< y \leq b_{N_0(y)}$. The  condition on~$2\xi$ ensures that $b_{n_0}\in (0,\varepsilon)$ and $b_{N_0(1-\varepsilon)-1}<1-\varepsilon/2$. More precisely, the first property holds because $2\xi< 1-(1-\varepsilon)^2$, and the second because $2\xi<1/2$. The choice of the parameters is motivated by their use in the following proof, see also~Fig.~\ref{fig:reaching0boundary}.

\begin{proof}[Proof of Lemma~\ref{lem:extinctiondoeblinlike}]
	
\begin{figure}[t]
	
	\scalebox{3.5}{	
}
\caption{A trajectory of~$Y$ that travels on its way to the $\varepsilon$-boundary strip around~$0$ through the intervals $(b_{N_0(y)+1},b_{N_0(y)}),\ldots, (b_{n_0}+1,b_{n_0})$.}
\label{fig:reaching0boundary}
\end{figure}

Let $\varepsilon\in (0,1/2)$ and fix $y\in [\varepsilon,1-\varepsilon]$. From the discussion above $y\in(b_{N_0(y)+1},b_{N_0(y)}]$. 
Note that $N_0(1-\varepsilon)\leq N_0(y)<n_0$. Set $t^\circ\coloneqq \sum_{i=N_0(1-\varepsilon)}^{n_0-1}t_i$. Define $\turtk_{N_0(y)-1}\coloneqq 0$ and, for $k\in \{N_0(y),\ldots, n_0-1\}$, 
$\turtk_{k}\coloneqq \inf\{t>\turtk_{k-1}:\, (t,r,u)\in N,\ r>2\xi\}$ and \[B_k\coloneqq \Big\{\turtk_k<(\turtk_{k-1}+t_k)\wedge T_Y(b_{k-1},1],\ Y_{\turtk_k}\in (b_{k+2},b_{k+1})\Big\}.\] 
	Using then that $b_{N_0(y)-1}\leq b_{N_0(1-\varepsilon)-1}<1-\varepsilon/2$ and $b_{n_0}\in (0,\varepsilon)$, we obtain \[\bigcap_{k=N_0(y)}^{n_0-1}B_k\subset \{T_Y[0,b_{n_0})<t^\circ\wedge T_Y(b_{N_0(y)-1},1]\}\subset \{ T_Y[0,\varepsilon)<t^\circ \wedge T_Y(1-\varepsilon/2,1]\}.\]
	Thus, $$\P_y(T_Y[0,\varepsilon)<t^\circ\wedge T_Y(1-\varepsilon/2,1])\geq \P_y(B_{N_0(y)}) \prod_{k=N_0(y)+1}^{n_0-1} \P_y(B_k\mid \cap_{i=N_0(y)+1}^{k-1} B_i).$$
	Using the strong Markov property and Lemma~\ref{lem:specialeventbound} below, we obtain, for each $k\in \{N_0(y)+1,\ldots, n_0-1\}$, {$\P_y(B_k\mid \cap_{i=N_0(y)+1}^{k-1} B_i)\geq p_k$ and $\P_y(B_{N_0(y)})\geq p_{N_0(y)}$, where the positive constants~$p_k$ are as in Lemma~\ref{lem:specialeventbound} and do not depend on the choice of $y$}. The result follows setting $c\coloneqq \prod_{k=N_0(1-\varepsilon)}^{n_0-1}p_k$.
\end{proof}

It remains to prove the bounds that we have used in the proof of Lemma~\ref{lem:extinctiondoeblinlike}.
\begin{lemme}\label{lem:specialeventbound}
	Let $(\Lambda,\mu,\sel)\in\Theta$. For any $n\leq n_0-1$ there exists $p_n>0$ such that for any $y\in [b_{n+1},b_n]$, \begin{equation}
		\P_y(\turtk<t_n\wedge T_Y(b_{n-1},1],\ Y_{\turtk}\in (b_{n+2},b_{n+1}))\geq p_n.
	\end{equation}
\end{lemme}
\begin{proof}
	We proceed similarly as in the proof of Lemma~\ref{lem:aux_mergingeventestimate}. Note first that for $s\in [0,\turtk)$, $Y_s=Y_s^{2\xi}$. Note also that, {if $b_{n+2}\leq Y_{\turtk-} \leq b_{n-1}$}, \[\mru{Y}{\turtk}{\rurtk}<\mru{Y}{\turtk}{2\xi}\leq \frac{b_{n-1}-2\xi}{1-2\xi}=b_{n+1},\quad\text{and}\quad \mro{Y}{\turtk}{\rurtk}>Y_{\turtk-}\geq b_{n+2}.\] Thus, since ${Y_{\turtk}}=\median{\mru{Y}{\turtk}{\rurtk},\uurtk, \mro{Y}{\turtk}{\rurtk}}$, if {moreover $\uurtk\in (b_{n+2},b_{n+1})$ and $\turtk\leq T_Y(b_{n-1},1]$, then $Y_{\turtk}=\uurtk\in (b_{n+2},b_{n+1})$ and $\turtk<T_Y(b_{n-1},1]$}. Therefore,
	\begin{align*}
	&\{ Y_{\turtk}\in (b_{n+2},b_{n+1}),\ \turtk<t_n\wedge T_Y(b_{n-1},1]\}\\
	&\qquad\qquad\qquad\supseteq\{ b_{n+2}\leq Y_{\turtk-} \leq b_{n-1},\ \turtk\leq T_Y(b_{n-1},1],\, \turtk< t_n,\uurtk\in (b_{n+2},b_{n+1})\}\\
	&\qquad\qquad\qquad\supseteq \{ \forall s\in [0,\turtk):\ b_{n+2}\leq Y_s \leq b_{n-1},\ \turtk<t_n,\uurtk\in (b_{n+2},b_{n+1})\}\\
	&\qquad\qquad\qquad\supseteq\Big\{\sup_{s\in [0,t_n]} \lvert Y_s^{2\xi}-y\rvert< \delta_n,\ \turtk<t_n,\ \uurtk\in (b_{n+2},b_{n+1}) \Big\}.
	\end{align*}
	Therefore, combining the independence of $Y^{2\xi}$, $\turtk$, and $\uurtk$, with Lemma~\ref{lem:capdeviation} and the definition of $t_n$ we obtain \begin{align*}
		&\P_y(\turtk<t_n \wedge T_Y(b_{n-1},1],\ Y_{\turtk}\in (b_{n+2},b_{n+1}))\\
		&\qquad\qquad\qquad\geq \P_y\Big(\sup_{s\in [0,t_n]} \lvert Y_s^{2\xi}-y\rvert< \delta_n,\ \turtk<t_n, \ \uurtk\in (b_{n+2},b_{n+1})\Big)\\
		&\qquad\qquad\qquad=\Big( 1-{\capco{\star}}\frac{\sqrt{t_n}}{\delta_n} \Big) \Big(1- \exp\big(-t_n\textstyle\int_{(2\xi,1)} r^{-2}\Lambda(\dd r)\big)\Big) (b_{n+1}-b_{n+2})\\
		&\qquad\qquad\qquad\geq \Big(1- \exp\big(-t_n\textstyle\int_{(2\xi,1)} r^{-2}\Lambda(\dd r)\big)\Big) (b_{n+1}-b_{n+2})/2\eqqcolon p_n>0,
	\end{align*}
achieving the proof.	
\end{proof}
The next lemma bounds the exponential moments of the time to the $n$th visit of~$Y$ to $(1-\varepsilon/2,1]$.
\begin{lemme}\label{lem:expmomentsreallyclose1}
	Let $(\Lambda,\mu,\sel)\in\Theta$. Assume $C_1(\Lambda,\mu,\sigma)<0$. Let $a\in (1/2,1)$ and $\gamma>0$ be as provided by Proposition~\ref{lem:escapetimeY}--(2). Let  $\tau>0$ and $\varepsilon\in (0,1/2)$ be such that $1-\varepsilon>a$. Define $\wctns{0} \coloneqq 0$, and for $n\in \N_0$, $$\whtns{n} \coloneqq\inf\{t\geq \wctns{n}:\ Y_t<1-\varepsilon\},\quad \wctns{n+1}\coloneqq (\whtns{n}+\tau)\wedge \inf\{t\geq \whtns{n}:\ Y_t>1-\varepsilon/2\}.$$
	Then there is $K\geq 1$ such that for all $n\geq 0$ and $y\in(0,1)$, \[\E_y[e^{\gamma \wctns{n}}]\leq (1-y)^{-1/4}K^n.\]
\end{lemme}
\begin{proof}
	The case $n=0$ is trivially true. For $n\geq 1$, we proceed similarly to Proposition~\ref{prop:nthboundaryvisit}. More precisely, arguing as in \eqref{eq:forlaterlem}, but using Proposition~\ref{lem:escapetimeY} instead of Proposition~\ref{prop:bddprobreturn}, we obtain
	\begin{align*} 
		\E_y\big[e^{\gamma \wctns{n}}\big]&	\leq  2e^{\gamma \tau}\E_y\bigg[ e^{\gamma \wctns{n-1} } \Big( \frac{1-a}{1-Y_{\wctns{n-1}}}\vee 1\Big)^{1/4}\bigg]\leq 2 e^{\gamma \tau}\E_y\bigg[ e^{\gamma \wctns{n-1} } (1-Y_{\wctns{n-1}})^{-1/4}\bigg] .
	\end{align*}
	For $n=1$, this corresponds already to the claim. For $n>1$, repeating the argument behind \eqref{eq:wctdecompclaim1} and~\eqref{eq:wctdecompclaim3}, we deduce that $${\E_y}[e^{\gamma \wctns{n-1}}(1-Y_{\wctns{n-1}})^{-1/4}]\leq \tilde{C} {\E_y}[e^{\gamma \wctns{n-1}}],$$ where $\tilde{C}\coloneqq (\varepsilon/2)^{-1/4}(2^{1/4}+4/3)$. By induction, we obtain $\E_y\big[e^{\gamma \wctns{n}}\big]\leq (2e^{\gamma\tau}\tilde{C})^{n-1} (1-y)^{-1/4}.$ Thus, setting $K\coloneqq (2e^{\gamma\tau}\tilde{C})$ yields the result.
\end{proof}

\begin{proof}[Proof of Proposition~\ref{prop:boundtimeto0extinction}]
	Let $a\in (1/2,1)$, $\gamma>0$ as in Lemma \ref{lem:expmomentsreallyclose1}. Fix $\varepsilon\in (0,1-a)$ and $\tau, c$ so that $\P_y(T_Y[0,\varepsilon)<\tau \wedge T_Y(1-\varepsilon/2,1])\geq c$ for any $y\in [\varepsilon,1-\varepsilon]$; the existence of such $\tau, c$ is ensured by Lemma~\ref{lem:extinctiondoeblinlike}. \modifbis{For this choice of $\tau$, we fix $K$ as given by } Lemma~\ref{lem:expmomentsreallyclose1}. We then fix $y \in (0,1)$ and check that the assumptions of Lemma~\ref{unificationlemma} are satisfied for $T:=T_Y[0,\varepsilon)$, $(T^1_n)_{n \geq 0}:=(\wctns{n})_{n \geq 0}$ and $(T^2_n)_{n \geq 0}:=(\whtns{n})_{n \geq 0}$, with $(\wctns{n})_{n \geq 0}$ and $(\whtns{n})_{n \geq 0}$ as in Lemma~\ref{lem:expmomentsreallyclose1}. Condition (i) is clearly satisfied. Lemma \ref{lem:extinctiondoeblinlike} ensures that Condition (ii) holds with $c$ as above. Lemma \ref{lem:expmomentsreallyclose1} shows that, almost surely, all stopping times $T^1_n$ are finite, and therefore all stopping times $T^2_n$ as well so Condition (iii) is satisfied. That lemma also shows that Condition (iv) is satisfied with $a:=(1-y)^{-1/4}$, $\lambda:=\gamma$ and $K$ as above. Then, the conclusion of the proposition follows from Lemma \ref{unificationlemma}.
\end{proof}

\subsection{Path of $Y$ to the boundary strip and the probability to get stuck there}\label{sec:asextinction:sub:returning_stuck}

We first prove Propositions~\ref{prop:controltimetobdd} and \ref{prop:controltimetobddstrong}. The latter will require another estimate that we subsequently prove.

\begin{proof}[Proof of Proposition~\ref{prop:controltimetobdd}]
	Recall that we have chosen $\varepsilon \in (0,e^{-b})$, $\zeta>0,M>0$ from Proposition~\ref{prop:boundtimeto0extinction}, that is, such that $\E_y[e^{\zeta T_Y[0,\varepsilon/2)}]\leq M(1-y)^{-1/4}.$
	For $a>0$ and $k\geq 0$, set $u_k(a)\coloneqq \P(\wctny{k}>ka, \Nstuck\geq k)$. Note that $u_0(a)=0$. We have for any $k\geq 1$ and $a\geq 0$, 
	\begin{align}
		u_k(a)=&\ \P\Big(\big(\wctny{k}-\whtnl{k-1}\big)+\big(\whtnl{k-1}-\wctny{k-1}\big)+\wctny{k-1}>ka,\ \whtnl{k-1}<\infty, \Nstuck\geq k-1 \Big) \nonumber \\
		\leq&\ u_{k-1}(a)+\P\big(\wctny{k}-\whtnl{k-1}>a/2,\whtnl{k-1}<\infty,\ \Nstuck\geq k-1\big)\label{eq:auxuka}\\
		&+ \P\big(a/2<\whtnl{k-1}-\wctny{k-1}<\infty,\Nstuck\geq k-1\big)\nonumber.	
	\end{align}
	\mofe{According to } Lemma~\ref{lem:lowerlevyhittingtime}--(1), we can choose \mofe{constants } $\alpha\in (0,\gamma)$ and $K>0$ such that $\P({t<T_{\hat{L}^b}(-\infty,-x)<\infty} )\leq K t^{-\alpha}$.
	We claim that \begin{eqnarray} &\P\Big(\wctny{k}-\whtnl{k-1}>a/2,\whtnl{k-1}<\infty,\ \Nstuck\geq k-1\Big) \leq (1-y)^{-1/4} M\tilde{M} e^{-\zeta a/2}, \label{eq:auxuka1}\\
		& \P\Big({\frac{a}{2}<\whtnl{k-1}-\wctny{k-1}<\infty},\Nstuck\geq k-1\Big)	\leq (1-y)^{-1/4} 2^\alpha K a^{-\alpha}, \label{eq:auxuka2} \end{eqnarray}
	for some $\tilde{M}\geq 1$ ({not depending on the choice of $y$}). Assume the claim is true. Note that we can choose $C$ such that $M\tilde{M} e^{-\zeta a/2}+2^\alpha K a^{-\alpha}\leq Ca^{-\alpha}$. Thus, \eqref{eq:auxuka} together with the claim yields, $u_k(a)\leq u_{k-1}(a)+C(1-y)^{-1/4} a^{-\alpha}$. Iterating this inequality yields $u_k(a)=kC (1-y)^{-1/4} a^{-\alpha}$. Therefore, $\P(\wctny{k}>t,\Nstuck\geq k)=u_k(t/k)\leq Ck^{1+\alpha}(1-y)^{-1/4} t^{-\alpha}$ as desired. It remains to prove the claim, i.e.~\eqref{eq:auxuka1} and~\eqref{eq:auxuka2}.

	Start with~\eqref{eq:auxuka1} and note that the left-hand side is 
	\begin{equation}
		\begin{split}
			&\E\Big[\P\big(\wctny{k}-\whtnl{k-1}>\frac{a}{2}\mid \mathcal{F}_{\whtnl{k-1}} \big)\ind{\whtnl{k-1}<\infty,\Nstuck\geq k-1} \Big]\\
			&=\E\Big[e^{-\zeta \frac{a}{2}} \E\big[e^{\zeta(\wctny{k}-\whtnl{k-1}) }\mid \mathcal{F}_{\whtnl{k-1}} \big] \ind{\whtnl{k-1}<\infty,\ \Nstuck\geq k-1} \Big]\\
		&\leq \E\Big[e^{-\frac{\zeta a}{2}} \E_{Y_0^{(k)}}\big[e^{\zeta T_Y[0,\frac{\varepsilon}{2})}\big]\ind{\whtnl{k-1}<\infty,\Nstuck\geq k-1}\Big]\\
		&\leq M e^{-\frac{\zeta a}{2}} \E\big[(1-Y_0^{(k)})^{-\frac14} \ind{\whtnl{k-1}<\infty, \Nstuck\geq k-1} \big],
		\end{split}\label{eq:auxuka3}
	\end{equation}
	where we used Chernoff's inequality, the strong Markov property, and Proposition~\ref{prop:boundtimeto0extinction}. If $k=1$, then {\eqref{eq:auxuka1} holds provided we choose $\tilde{M}\geq 1$}. If $k\geq 2$, use Lemma~\ref{lem:auxlemmacloseT} (with $\lambda=0$) and that {$(1-y)^{-1/4}\geq 1$ to show that \eqref{eq:auxuka1} holds with $\tilde{M}= 5(1-\varepsilon)^{-1/4}$}.

	Now, we prove~\eqref{eq:auxuka2}. Note that for $k=1$, since $\whtnl{0}=\wctny{0}=0$, the left-hand side of \eqref{eq:auxuka2} equals $0$. For $k\geq 2$, use the definition of $\whtnl{k-1}$ and Lemma~\ref{lem:lowerlevyhittingtime}--(1), to get 
	\begin{align*}
		&\P\Big({\frac{a}{2}<\whtnl{k-1}-\wctny{k-1}<\infty},\Nstuck\geq k-1\Big)	\\
		&\leq \P\Big({\frac{a}{2}<T_{\hat{\hL}^b}\Big(-\infty,\log\Big(\frac{\hL_0^{(k-1)}}{\varepsilon}\Big)\Big)<\infty}\Big) \leq K2^{\alpha}a^{-\alpha}\leq (1-y)^{-1/4} 2^\alpha K a^{-\alpha}.
	\end{align*}
	This completes the proof of~\eqref{eq:auxuka2}.	
\end{proof}
\begin{proof}[Proof of Proposition~\ref{prop:controltimetobddstrong}]
	Recall that we have chosen $\varepsilon \in (0,e^{-b})$, $\zeta>0,M>0$ from Proposition~\ref{prop:boundtimeto0extinction}, that is, such that $\E_y[e^{\zeta T_Y[0,\varepsilon/2)}]\leq M(1-y)^{-1/4}.$
	Thanks to Corollary~\ref{coro:lowerlevyhittingtime}, 
	there exist $\alpha$ and $K$ such that  $\E[e^{\alpha T_{\hat{L}^b}(-\infty,x)} \ind{T_{\hat{L}^b}(-\infty,x)<\infty} ]\leq K$. 
	Fix $\lambda\in (0,\alpha\wedge \zeta)$. 
	By the strong Markov property at $\whtnl{k-1}$ and Proposition~\ref{prop:boundtimeto0extinction}, we have for $k\geq 1$, 
	\mofe{\begin{align*}
		\E&\Big[e^{\lambda \wctny{k}}\ind{\Nstuck\geq k}\Big]=\E\Big[ e^{\lambda\whtnl{k-1}}   \E\big[e^{\lambda(\wctny{k}-\whtnl{k-1})}\mid \mathcal{F}_{\whtnl{k-1}} \big] \ind{\whtnl{k-1}<\infty,\Nstuck\geq k-1}\Big]\\
		&=\E\Big[ e^{\lambda\whtnl{k-1}}  \E_{Y_0^{(k)}}\big[e^{\lambda T_Y[0,\varepsilon/2)} \big]\ind{\whtnl{k-1}<\infty,\Nstuck\geq k-1}  \Big]\\
		&\leq M \E\Big[e^{\lambda \whtnl{k-1}}  (1-Y_0^{(k)})^{-1/4} \ind{\whtnl{k-1}<\infty,\ \Nstuck\geq k-1}\Big]\\
		&= M\E\Big[ \E\big[ e^{\lambda(\whtnl{k-1}-\wctny{k-1})}(1-Y_0^{(k)})^{-1/4} \ind{\whtnl{k-1}<\infty }\mid \mathcal{F}_{\wctny{k-1}}\big]e^{\lambda \wctny{k-1} }\ind{\Nstuck\geq k-1}\Big].
	\end{align*}}
	If $k=1$, the claim is true since the inner (conditional) expectation is just $(1-y)^{-1/4}$. For $k\geq2$, we will show in Lemma~\ref{lem:auxlemmacloseT} below that the inner (conditional) expectation can be bounded by some $\tilde{K}\geq 1$. Thus, we obtain $\E[e^{\lambda \wctny{k}}\ind{\Nstuck\geq k}]\leq M\tilde{K}\E[e^{\lambda \wctny{k-1}}\ind{\Nstuck\geq k-1}]$ for {this }$\tilde{K}\geq 1$. It then follows by induction that $$\E\Big[e^{\lambda \wctny{k}}\ind{\Nstuck\geq k}\Big]\leq (1-y)^{-1/4} M^k \tilde{K}^{k-1}\leq (1-y)^{-1/4} (M\tilde{K})^{k}.$$ Setting $C\coloneqq M\tilde{K}$ yields the desired result. 
\end{proof}
It remains to prove the bound we used in the proof of Proposition~\ref{prop:controltimetobddstrong}.

\begin{lemme}\label{lem:auxlemmacloseT}
	\begin{enumerate}
		\item Assume that $(\Lambda,\mu,\sigma)\in \Theta_1$. Let $\varepsilon \in (0,e^{-b}), \zeta>0,M>0$ as provided by  Proposition~\ref{prop:boundtimeto0extinction}. Then, for all $\lambda\in[0,\zeta)$ and $y\in(0,1)$, if $Y_0^{(1)}=y$, we have
		\begin{equation}\label{lem:auxlemmacloseTintermediary} 
			\begin{split}
				&\E\Big[ e^{\lambda(\whtnl{k-1}-\wctny{k-1})}(1-Y_0^{(k)})^{-1/4} \ind{\whtnl{k-1}<\infty }\mid \mathcal{F}_{\wctny{k-1}}\Big]  \\
			&\leq 5(1-\varepsilon)^{-1/4}\E\Big[e^{\lambda T_{\hL^{k-1}}(-\infty,\log( \hL_0^{(k-1)} /\varepsilon))}\ind{T_{\hL^{k-1}}(-\infty,\log(\hL_0^{(k-1)}/\varepsilon))<\infty} \Big].
			\end{split} 
		\end{equation}
		\item Assume in addition that $s_\gamma(r^{-2} \Lambda(\dd r))<\infty$ and $s_\gamma(\bar{\mu})<\infty$ for some $\gamma>0$. Let $\alpha,K>0$ as provided by Corollary~\ref{coro:lowerlevyhittingtime} with $m=0$. Then, for $\lambda\in [0,\alpha\wedge \zeta)$, there is $\tilde{K}\geq 1$ such that, for all $y\in (0,1)$, if $Y_0^{(1)}=y$, we have
		\begin{align*}
			\E\Big[ e^{\lambda(\whtnl{k-1}-\wctny{k-1})}(1-Y_0^{(k)})^{-1/4} \ind{\whtnl{k-1}<\infty }\mid \mathcal{F}_{\wctny{k-1}}\Big]\leq  \tilde{K}.
		\end{align*}
	\end{enumerate}
	
\end{lemme}
The expectations in \eqref{lem:auxlemmacloseTintermediary} are finite for $\lambda=0$ or, by Corollary~\ref{coro:lowerlevyhittingtime}, for some $\lambda>0$ under the assumption that $s_\gamma(r^{-2} \Lambda(\dd r))<\infty$ and $s_\gamma(\bar{\mu})<\infty$ for some $\gamma>0$. Otherwise, the expectations in \eqref{lem:auxlemmacloseTintermediary} may be infinite. 
\begin{proof}
	Note that if (1) holds, then (2) follows from Corollary~\ref{coro:lowerlevyhittingtime}. Thus we are left to prove (1). We proceed similarly as in the proof of Proposition~\ref{prop:nthboundaryvisit}. First recall that \begin{equation}
		\whtnl{k-1}-\wctny{k-1}= T_{\hL^{k-1}}(-\infty,\log(\hL_0^{(k-1)}/\varepsilon)).\label{eq:auxlemmacloseT_auxeq1}
	\end{equation} Let $J_N$, $J_{S-}$ and $J_{S+}$ be the sets of jump times of~$N$, and of $S$ with negative and positive $r$-component, respectively. For $T\in J_N$, let $(T,R_T,U_T)$ be the corresponding jump. Note that $Y_0^{(k)}$ cannot be reached via a jump in $J_{S+}$. Thus, by the definition of~$Y_0^{(k)}$, {Lemma~\ref{lem:estimatesel} and~\eqref{eq:estimateselneg}}, we obtain 
\begin{align*}Y_0^{(k)}&\leq \varepsilon \ind{\whtnl{k-1}\notin J_N\cup J_{S-}\cup J_{S+}}+\frac{1+\varepsilon}{2}\ind{\whtnl{k-1}\in J_{S-}}\\
&+U_{\whtnl{k-1}}\ind{\whtnl{k-1}\in J_N,\ Y_0^{(k)}>\varepsilon} + \varepsilon \ind{\whtnl{k-1}\in J_N, Y_0^{(k)}\leq \varepsilon }.\end{align*}
Now we go through all possible cases. First, consider $\{\whtnl{k-1}\notin J_N\cup J_{S-}\cup J_{S+}\}$.
	Using~\eqref{eq:auxlemmacloseT_auxeq1} {and the definition of $\hL^{k-1}$}, 
	\begin{align*}
		&\E\Big[ e^{\lambda(\whtnl{k-1}-\wctny{k-1})}(1-Y_0^{(k)})^{-1/4} \ind{\whtnl{k-1}<\infty,\whtnl{k-1}\notin J_N\cup J_{S-}\cup J_{S+}}\mid \mathcal{F}_{\wctny{k-1}}\Big]\\
		&\qquad\qquad\leq (1-\varepsilon)^{-1/4}\E_{L_0^{(k-1)}}\Big[ e^{\lambda {T_{\hat{L}^{b}}}(-\infty,\log( \hL_0^{(k-1)} /\varepsilon))}\ind{{T_{\hat{L}^{b}}}(-\infty,\log(\hL_0^{(k-1)}/\varepsilon))<\infty}\Big].
		\end{align*}
	A similar argument applies for the cases where $\{\whtnl{k-1}\in J_{S-}\}$ and $\{ \whtnl{k-1}\in J_N, Y_0^{(k)}\leq \varepsilon \}$. The case $\{\whtnl{k-1}\in J_N,\ Y_0^{(k)}>\varepsilon\}$ is slightly more delicate. First note that on $\{Y_0^{(k)}>\varepsilon\}$, since $\hL^{k-1}$ is a lower bound for $\log(1/Y^{2,k-1})$ (until the first time $Y^{2,k-1}$ exits $[0,e^{-b}]$), we have $$\whtnl{k-1}-\wctny{k-1}= {T_{Y^{k-1,2}}(\varepsilon,1]}=T_{\hL^{k-1}}(-\infty,\log(\hL_0^{(k-1)}/\varepsilon)).$$ Using this and the definitions of $Y^{k-1,2}$ and $\hL^{k-1}$,
	\begin{align*}
		&\E\Big[ e^{\lambda(\whtnl{k-1}-\wctny{k-1})}(1-Y_0^{(k)})^{-1/4} \ind{\whtnl{k-1}<\infty,\whtnl{k-1}\in J_N,\ Y_0^{(k)}>\varepsilon}\mid \mathcal{F}_{\wctny{k-1}}\Big]\\
		&\leq \E_{L_0^{(k-1)}}\Big[ e^{\lambda T_{Y}(\varepsilon,1]}(1-U_{T_{Y}(\varepsilon,1]})^{-1/4}\ind{T_{Y}(\varepsilon,1]<\infty,T_{Y}(\varepsilon,1]\in J_N, T_{Y}(\varepsilon,1]=T_{\hat{L}^{b}}(-\infty,\log( \hL_0^{(k-1)} /\varepsilon))}\Big]\\
	&\leq 2(1-\varepsilon)^{-1/4}\E_{L_0^{(k-1)}}\Big[ e^{\lambda T_{\hat{L}^{b}}(-\infty,\log( L_0^{(k-1)} /\varepsilon))} \ind{T_{\hat{L}^{b}}(-\infty,\log( L_0^{(k-1)} /\varepsilon))<\infty} \Big],
	\end{align*}
	where the last identity is obtained as in the proof of \eqref{eq:boundjumpevent2} with $T_{Y}(\varepsilon,1]$ instead of $\wct$.
 Combining all cases yields the desired result.
\end{proof}

	\subsection{A conjecture on the optimality of the polynomial decay}\label{rem:conjpolynomial} Note that the assumption $w_{\gamma}(r^{-2}\Lambda(\dd r))<\infty$ in Theorem~\ref{thm:survival_probability}--(1)(i) implies that, for all $x\in (0,1/2)$, \[\Lambda([1-x,1))<w_{\gamma}(r^{-2}\Lambda(\dd r)) \log(1/x)^{-(1+\gamma)}.\] For $(\Lambda,\mu,\sigma)\in \Theta_1$, we conjecture that the upper bound for $\P_x(X_t<1-t^{-\rho})$ in Theorem~\ref{thm:survival_probability}--(1)(i) could not be better than polynomial. More precisely, we conjecture that, if there is $c>0$ and $\xi>1$ such that, for all $x\in (0,1)$, \begin{equation}
		\Lambda([1-x,1))>c(\log(1/x))^{-\xi},\label{eq:clog}
	\end{equation} then there is $c_x>0$ such that $\P_x(X_t<1-t^{-\rho})\geq c_x t^{-\xi}$. (An analogous conjecture can be stated for $\P_x(X_t>t^{-\rho})$ when $(\Lambda,\mu,\sigma)\in \Theta_0$.) Note that \eqref{eq:clog} violates the strong integrability assumption from Theorem~\ref{thm:survival_probability}--(1)(ii), but can be compatible with the weak integrability assumption in part (1)(i). The conjecture is based on the following heuristic. First note that, since for $t$ large $\P_x(X_t<1/2)\leq \P_x(X_t<1-t^{-\rho})$ , it would be sufficient to prove that $c_xt^{-\xi}\leq \P_x(X_t<1/2)$. Fix $x\in (0,1)$. Recall that by Theorem \ref{thm:stuckatboundary}, $\P_{1/2}$-a.s. $\lim_{t\to\infty}Y_t=0$. When $Y$ is close to~$0$, $\log(1/Y)$ approximately behaves like a L{\'e}vy process with mean $C_0(\Lambda,\mu,\sel)$ (\modifbis{see Lemma \ref{lem:levysandwich} and Remark \ref{limrateofgrowth}}). More rigorously, it is not difficult to show (using among others Lemma~\ref{lem:levysandwich} and Chernoff's inequality) that there is $\tilde{c}>0$ such that $\lim_{t\to\infty} \P_{e^{-\tilde{c}t/2}}(e^{-\tilde{c}t}<Y_t)=1$. In particular, $\lim_{t\to\infty} \P_{1/2}(e^{-\tilde{c}t}<Y_t)=1$. Let~$c$ and~$\xi$ be such that \eqref{eq:clog} is satisfied. Then, with high $\P_{1/2}$-probability, $\Lambda([1-Y_{s-},1))>c\log(1/Y_{s-})^{-\xi}> c's^{-\xi}$ for some $c'>0$. $\Lambda([1-Y_{s-},1))$ is the rate of jumps of~$N$ having $r$-component larger than $1-Y_{s-}$. Therefore, the $\P_{1/2}$-probability of such a jump to occur for $s\in [t-1,t]$ can be expected, for large~$t$, to be larger than $c''t^{-\xi}$ for some $c''>0$. Given such a jump occurs, we can expect to have $Y_t>x$ with a $\P_{1/2}$-probability larger than some $c_x'>0$ independent of~$t$. In particular, for large~$t$, we should have $\P_{1/2}(Y_t>x)\geq c_xt^{-\xi}$ for some $c_x>0$. By Siegmund duality, this would imply $\P_{x}(X_t<1/2)\geq c_x t^{-\xi}$, yielding the expected result. 
 
\section{The case \texorpdfstring{$\Theta_3$}{Theta3}: Proof of Theorem~\ref{thm:coexistenceY}}\label{sect:coex}
In this section, we prove Theorem~\ref{thm:coexistenceY} using a similar approach as in the previous section. 
We construct a process $\fray$ with the same distribution as~$Y$ by assembling a sequence of path segments. 
These segments are alternately contained within the $\varepsilon/2$-interior and the $\varepsilon$-boundary strip, 
where $\varepsilon>0$.
Additionally, the distance to the boundary of each path segment within the $\varepsilon$-boundary strip is upper bounded by the exponential of a L\'evy process 
until the latter exceeds $\varepsilon$.

In the entire section we assume $(\Lambda,\mu,\sigma)\in \Theta_3$. Recall that $\bar{\mu}$ denotes the pushforward measure of $\mu$ under the map $r\mapsto-r$, and $\bar{\sigma}(y)\coloneqq-\sigma(1-y)$. Fix $\rho\in (0,C_0(\Lambda,\mu,\sigma) \wedge C_1(\Lambda,\mu,\sigma))$ and $m \in (\rho,C_0(\Lambda,\mu,\sigma) \wedge C_1(\Lambda,\mu,\sigma))$. Fix also $b> 0$ such that $\E[\hat{L}_1^b]\wedge\E[\hat{L}_1^{b,*}] >m$, where $\hat{L}^{b,*}$ is defined as $\hat{L}^b$, but using $(\Lambda,\bar{\mu},\bar{\sigma})$ instead of $(\Lambda,\mu,\sigma)$; existence of $b$ follows from Lemma~\ref{lem:existence_suit_approx}. Recall that $\hat{L}^{b}_t$ is a lower bound for $\log(Y_0/Y_t)$ when $Y$ is close to $0$ (see Lemma \ref{lem:levysandwich}). Similarly, $\hat{L}^{b,*}$ is a lower bound for $\log((1-Y_0)/(1-Y_t))$ when $Y$ is close to $1$. Finally fix $\varepsilon \in (0,e^{-b})$.
\smallskip
 
Let $(\Omega,\mathcal{F},\P)$ be a probability space on which {is defined }an i.i.d. sequence of Poisson random measures $(N^{n,i},S^{n,i})_{n\in \N, i\in \{1,2\}}$ with $N^{n,i}\sim N$ and $S^{n,i}\sim S$, $n\in\N$ and $i\in\{1,2\}$. For $n\in \N$, let $L^{0,n}$ and $L^{1,n}$ be copies of $\hat{L}^b$ and $\hat{L}^{b,*}$, respectively, with $N$ and $S$ replaced by $N^{n,2}$ and $S^{n,2}$, respectively. For $n\in\N$ and $i\in\{1,2\}$, let $Y^{n,i}$ be the solution of~\eqref{eq:SDE_Y} with $N$ and $S$ replaced by $N^{n,i}$ and $S^{n,i}$, respectively; their initial values and the time frames in which they will be considered are determined iteratively as follows. Set $Y^{1,1}_0\coloneqq y {\in (0,1)}$, $\whtnl{0}\coloneqq \wctny{0} \coloneqq0$ and for $n\geq 1$, 
\begin{align*}
&\wctny{n}\coloneqq \inf\big\{t\geq \whtnl{n-1}:Y_{t-\whtnl{n-1}}^{1,n}\in [0,\varepsilon/2)\cup (1-\varepsilon/2,1]\big\},\qquad Y^{2,n}_0\coloneqq L_0^{(n)}\coloneqq Y_{\wctny{n}-\whtnl{n-1}}^{1,n},\\
&\whtnl{n}\coloneqq \begin{cases} 
	\inf\big\{t\geq \wctny{n}: L_0^{(n)} \,{\exp(-L_{t-\wctny{n}}^{0,n}+m(t-\wctny{n}))}>\varepsilon\big\},&\text{if }L_0^{(n)}<\varepsilon/2,\\
	\inf\big\{t\geq \wctny{n}: (1-L_0^{(n)})\, {\exp(-L_{t-\wctny{n}}^{1,n}+m(t-\wctny{n}))}>\varepsilon\big\},&\text{if }1-L_0^{(n)}<\varepsilon/2, \end{cases}
\end{align*} 
and $Y_0^{1,n+1}\coloneqq Y_{\whtnl{n}-\wctny{n}}^{2,n}$. Set $$\fray_t=\begin{cases}
	Y_{t-\whtnl{n-1}}^{1,n}&\text{for }t\in [\whtnl{n-1},\wctny{n}),\\
	Y_{t-\wctny{n}}^{2,n}&\text{for }t\in [\wctny{n},\whtnl{n}).
\end{cases}$$

Clearly, by the Markov property, $\fray\sim Y$ under $\P_y$. The choice of~$b$ and Lemma~\ref{lem:existence_suit_approx} yield  \begin{eqnarray}
	&\P(\whtnl{n}=\infty \mid \whtnl{1}<\infty,...,\whtnl{n-1}<\infty) \nonumber \\
	\geq & \min_{i\in\{0,1\}}\P\Big(\inf_{t\in [0,\infty)}(L_t^{i,1}-mt)>-\log(2)\Big)>0,\label{eq:probabilitygettingstuck}
\end{eqnarray} so that, there is a.s. an $n\geq 1$ such that $\whtnl{n}=\infty$. {Define $\Nstuck\coloneqq \min\{n\geq 1:\ \whtnl{n}=\infty\}$}.

We first control the time it takes~$\fray$ to go from the boundary strip to the $\varepsilon$-interior analogously to Proposition~\ref{prop:controltimetobdd}.

\begin{lemme}\label{lem:expmomentreachint}
	Let $(\Lambda,\mu,\sel)\in\Theta_3$. There exist $\lambda > 0$ and $\check{M}>0$, such that, for any $k\in \N$ and $y\in(0,1)$, if $Y_0^{1,1}=y$, we have $$\E\left[e^{\lambda(\wctny{k}-\whtnl{k-1})}\mid \mathcal{F}_{\whtnl{k-1}}\right] \ind{\Nstuck\geq k} \leq \check{M} \ind{\Nstuck\geq k}.$$ In particular, we have $\wctny{k}<\infty$ a.s. on $\{\Nstuck\geq k\}$. 
\end{lemme}
\begin{proof}
	Consider first $k=1$. Using Lemma~\ref{lem:extinctiondoeblinlike}, we choose $\tau,c>0$ such that, for any $Y_0^{1,1}=y\in [\varepsilon/2,1-\varepsilon/2]$, we have $\P(\wctny{1}<\tau)\geq c$. Thus, using the Markov property at times $k\tau$, we get that, for any $\lambda\in (0,\log(1/(1-c))/\tau)$, 
	$$\E[e^{\lambda \wctny{1}}]\!\leq \!\E\Big[\sum_{k=0}^{\infty}\ind{\wctny{1}\in [k\tau,(k+1)\tau) }e^{\lambda(k+1)\tau} \Big]\!\leq\! \sum_{k=0}^{\infty} e^{\lambda(k+1)\tau}(1-c)^k=\frac{e^{\lambda\tau}}{1-e^{\lambda\tau}(1-c)}\eqqcolon \check{M}.$$ 
	
	Next, note that for $k\geq 2$, on $\{\Nstuck\geq k\}$, if $\wctny{k}-\whtnl{k-1}>0$, then $Y_0^{1,k}\in [\varepsilon/2,1-\varepsilon/2]$. Thus, by the Markov property at time $\whtnl{k-1}$, we have $\E[e^{\lambda(\wctny{k}-\whtnl{k-1})}\mid \mathcal{F}_{\whtnl{k-1}}]\leq \check{M}$. Moreover, if $\wctny{k}-\whtnl{k-1}=0$, the result holds, because $\check{M}>1$.
\end{proof}
The following lemma, analogously to Proposition~\ref{prop:controltimetobdd}, estimates $\P(\wctny{k}>t,\Nstuck\geq k)$ under the weak integrability conditions on the measures $\Lambda$ and $\mu$.
\begin{lemme}\label{lem:tyt}
	Assume that $(\Lambda,\mu,\sigma)\in \Theta_3$, $\rW_\gamma(r^{-2}\Lambda(\dd r))<\infty$, $\rW_\gamma(\mu)<\infty$ and $\rW_\gamma(\bar{\mu})<\infty$, for some $\gamma>0$. Then for any $\alpha\in (0,\gamma)$, there is a constant $C\geq 1$ such that, for all $k\geq 1$, $t>0$ and $y\in(0,1)$, if $Y_0^{1,1}=y$, we have $$\P(\wctny{k}>t, \Nstuck\geq k)\leq Ck^{1+\alpha}t^{-\alpha}.$$
\end{lemme}
\begin{proof} 
	We will proceed as in the proof of Proposition~\ref{prop:controltimetobdd}. For $a>0$ and $k\geq1$, we define $u_k(a)\coloneqq \P(\wctny{k}>ka,\Nstuck\geq k)$. In particular, $u_0(a)=0$. Recall from \eqref{eq:auxuka} that for $k\geq 1$,\begin{eqnarray}
		u_k(a)&\leq u_{k-1}(a)+\P\big(\wctny{k}-\whtnl{k-1}>a/2, \Nstuck\geq k\big) \nonumber \\
		&+\P\big({a/2<\whtnl{k-1}-\wctny{k-1}<\infty}, \Nstuck\geq k-1\big).\label{eq:ukaaux}
	\end{eqnarray}
	We first deal with the last term on the right-hand side of~\eqref{eq:ukaaux}. If $k=1$, this term is~$0$. Hence, assume $k\geq 2$. Fix $\alpha\in (0,\gamma)$. {For $i \in \{0,1\}$, let us denote by $\tilde L^{i,k}$ the L\'evy process $(L^{i,k}_t - mt)_{t \geq 0}$}. By Lemma~\ref{lem:lowerlevyhittingtime}--(1), we can choose $K$ and $\bar{K}$, {independent of $k$ and of the choice of $y$}, such that for all $t,x\geq 0$, $$\P({t<T_{\tilde L^{0,k-1}}(-\infty,-x)<\infty})\leq Kt^{-\alpha}\quad \text{and}\quad\P({t<T_{\tilde L^{1,k-1}}(-\infty,-x)<\infty})\leq \bar{K}t^{-\alpha}.$$ Then, \begin{align*}
		&\P({a/2<\whtnl{k-1}-\wctny{k-1}<\infty},\Nstuck\geq k-1)\\
		&\leq\ \E\Big[{\P\left({a/2<T_{\tilde L^{0,k-1}}\big(-\infty,-\log(L_0^{(k-1)}/\varepsilon)\big)<\infty}\mid L_0^{(k-1)}\right)} \ind{L_0^{(k-1)}<\varepsilon/2}\\
		& \qquad\quad+{\P\left({a/2<T_{\tilde L^{1,k-1}}\big(-\infty,-\log((1-L_0^{(k-1)})/\varepsilon)\big)<\infty}\mid L_0^{(k-1)}\right)} \ind{1-L_0^{(k-1)}<\varepsilon/2}\Big]\\
		&\qquad\qquad\leq (2/a)^\alpha (K+\bar{K}).
	\end{align*}
	Next, we deal with the second-last term on the right-hand side of~\eqref{eq:ukaaux} for $k\geq 1$. Proceeding as in~\eqref{eq:auxuka3} and using Lemma~\ref{lem:expmomentreachint}, we deduce that {for some $\zeta> 0$ and $\check{M}>0$, independent of the choice of $y$},
	\begin{align*}
		\P\Big(\wctny{k}-\whtnl{k-1}>a/2, \Nstuck\geq k\Big)&\leq e^{-\zeta a/2}\E\Big[\E\Big[e^{\zeta(\wctny{k}-\whtnl{k-1})}\mid \mathcal{F}_{\whtnl{k-1}}\Big] \ind{\Nstuck\geq k}\Big]\leq e^{-\zeta a/2} \check{M}.	
	\end{align*}
	Combining the estimates of the second-last and last term on the right-hand side of~\eqref{eq:ukaaux} yields
	$$u_k(a)\leq u_{k-1}(a)+\check{M}e^{-\zeta a/2}+2^{\alpha}(K+\bar{K})a^{-\alpha}\leq u_{k-1}(a)+Ca^{-\alpha},$$ where $C\geq 1$ is a constant chosen such that $Ca^{-\alpha}\geq \check{M}e^{-\zeta a/2}+2^{\alpha}(K+\bar{K})a^{-\alpha}$ for all $a> 0$. Thus, by induction, $u_k(a)\leq kCa^{-\alpha}$. Setting $a=t/k$ yields the statement of the lemma. 
\end{proof}

Next, we control $\wctny{k}$ under the stronger integrability conditions on the intensity measures. 
\begin{lemme}\label{lem:momentescapeintstrong}
	Assume that $(\Lambda,\mu,\sigma)\in \Theta_3$, $s_\gamma(r^{-2}\Lambda(\dd r))<\infty$, $s_\gamma(\mu)<\infty$ and $s_\gamma(\bar{\mu})<\infty$, for some $\gamma>0$. Then, there is $\lambda>0$ and $C\geq 1$ such that, for all $k\geq 1$ and $y\in(0,1)$, if $Y_0^{1,1}=y$, $$\E[e^{\lambda \wctny{k}}\ind{\Nstuck\geq k}]\leq C^k.$$
\end{lemme}
\begin{proof}
	{For $i \in \{0,1\}$, let us denote by $\tilde L^{i,k}$ the L\'evy process $(L^{i,k}_t - mt)_{t \geq 0}$}. By Corollary~\ref{coro:lowerlevyhittingtime}, we can choose $\alpha,\bar{\alpha},K,\bar{K}$, {independent of the choice of $y$}, such that for any $z>0$ and $k\in \N$, {$$\E\big[e^{\alpha T_{\tilde L^{0,k}}(-\infty,-z)} \ind{T_{\tilde L^{0,k}} (-\infty,-z)<\infty}\big]\leq K\quad \text{and}\quad \E\big[e^{\bar{\alpha} T_{\tilde L^{1,k}}(-\infty,-z)}\ind{T_{\tilde L^{1,k}}(-\infty,-z)<\infty}\big]\leq \bar{K}.$$} Set $\lambda\in (0,\alpha\wedge \bar{\alpha})$ that is also smaller than the $\lambda$ in Lemma~\ref{lem:expmomentreachint}. Then, on $\{\Nstuck\geq k\}$, 
	\begin{eqnarray}
		&\E\Big[ e^{\lambda (\whtnl{k}-\wctny{k})}\ind{\whtnl{k}<\infty } \mid \mathcal{F}_{\wctny{k} } \Big]\leq {\E\Big[ e^{\lambda (\whtnl{k}-\wctny{k})}\ind{T_{\tilde L^{0,k}} (-\infty,\log(L_{0}^{(k)}/\varepsilon))<\infty } \mid L_0^{(k)}\Big]}\ind{L_0^{(k)}<\varepsilon/2}\nonumber\\
		&\qquad\qquad+ {\E\Big[ e^{\lambda (\whtnl{k}-\wctny{k})}\ind{T_{\tilde L^{1,k}} (-\infty,\log((1-L_{0}^{(k)})/\varepsilon))<\infty } \mid L_0^{(k)}\Big]}\ind{L_0^{(k)}>1-\varepsilon/2}\leq K+\bar{K}.\label{eq:expmomentreachbddaux}
	\end{eqnarray}
	Then, using Lemma~\ref{lem:expmomentreachint} and~\eqref{eq:expmomentreachbddaux},
	\begin{align*}
		\E[e^{\lambda \wctny{k}}\ind{\Nstuck\geq k}]&=\E\Big[e^{\lambda \whtnl{k-1}}\ind{\Nstuck\geq k} \E\big[e^{\lambda(\wctny{k}-\whtnl{k-1})}\mid \mathcal{F}_{\whtnl{k-1}} \big]\Big] \\
		&\leq \check{M} \E\Big[e^{\lambda \wctny{k-1}}\ind{\Nstuck\geq k-1} \E\big[ e^{\lambda (\whtnl{k-1}-\wctny{k-1})}\ind{\whtnl{k-1}<\infty } \mid \mathcal{F}_{\wctny{k-1} } \big]\Big]\\
		&\leq \check{M}(K+\bar{K}) \E\Big[ e^{\lambda \wctny{k-1}}\ind{\Nstuck\geq k-1}\Big],
	\end{align*}
	{where $\check{M}>0$ is independent of the choice of $y$}. The result follows by induction.
\end{proof}

Now all the ingredients to prove Theorem \ref{thm:coexistenceY} are prepared.
\begin{proof}[Proof of Theorem \ref{thm:coexistenceY}]
	From \eqref{eq:probabilitygettingstuck} and the subsequent discussion follows that $\Nstuck < \infty$ almost surely. Lemma \ref{lem:expmomentreachint} implies that $\wctny{\Nstuck}<\infty$ almost surely. By Lemma~\ref{lem:levysandwich}, 
		$\fray_t \leq \varepsilon \exp(-m(t-\wctny{\Nstuck}))$ for all $t \geq \wctny{\Nstuck}$ if $L_0^{(\Nstuck)}<\varepsilon/2$, and $\fray_t \geq 1-\varepsilon \exp(-m(t-\wctny{\Nstuck}))$ for all $t \geq \wctny{\Nstuck}$ if $1-L_0^{(\Nstuck)}<\varepsilon/2$. Since $\exp(-m(t-\wctny{\Nstuck}))=\exp(-\rho t) \exp( m\wctny{\Nstuck}-(m-\rho) t)$ and $m \wctny{\Nstuck}/(m-\rho) \geq \wctny{\Nstuck}$, we deduce that either $\fray_t \in [0,e^{-\rho t}]$ for all $t \geq m \wctny{\Nstuck}/(m-\rho)$, or $\fray_t \in [1-e^{-\rho t},1]$ for all $t \geq m \wctny{\Nstuck}/(m-\rho)$. Since $\fray\sim Y$ under $\P_y$, the first statement in Theorem \ref{thm:coexistenceY} is true and $$\P_y(Y_t\in [e^{-\rho t},1-e^{-\rho t}])=\P(\fray_t\in [e^{-\rho t},1-e^{-\rho t}])\leq \P(\wctny{\Nstuck}>(m-\rho) t / m).$$
	Hence, to complete the proof, we have to bound $\P(\wctny{\Nstuck}>(m-\rho) t / m)$. To this end, we check that the assumptions of Lemma \ref{unificationlemma} are satisfied for $T:=\wctny{\Nstuck}$, $(T^1_n)_{n \geq 0}:=(\wctny{n})_{n\geq0}$ and $(T^2_n)_{n \geq 0}:=(\whtnl{n})_{n\geq 0}$. Condition (i) is clearly satisfied. Condition (iii) is satisfied by the definition of $\Nstuck$ after \eqref{eq:probabilitygettingstuck} and by $T:=\wctny{\Nstuck}$. We then see from \eqref{eq:probabilitygettingstuck} that Condition (ii) holds with $c:=\min_{i\in\{0,1\}}\P(\inf_{t\in [0,\infty)}(L_t^{i,1}-mt)>-\log(2))$. If for some $\gamma>0$, $s_\gamma(r^{-2}\Lambda(\dd r))<\infty$, $s_\gamma(\mu)<\infty$, and $s_\gamma(\bar{\mu})<\infty$, then by Lemma \ref{lem:momentescapeintstrong}, we can choose $\lambda>0$ and $C\geq 1$, {independent of the choice of $y$}, such that for all $k \geq 1$, $\E[e^{\lambda \wctny{k}}\ind{\Nstuck\geq k}]\leq C^k$. Without loss of generality we assume $C>1$. Thus, Condition (iv) is satisfied with $a:=1$, $\lambda$ as above and $K:=C$. If for some $\gamma>0$, $\rW_\gamma(r^{-2}\Lambda(\dd r))<\infty$, $\rW_\gamma(\mu)<\infty$, and $\rW_\gamma(\bar{\mu})<\infty$, then we fix $\alpha\in (0,\gamma)$. By Lemma~\ref{lem:tyt}, we can choose $\alpha'\in (\alpha,\gamma)$ and $C'\geq 1$, {independent of the choice of $y$, such that $\P(\wctny{k}>t,\Nstuck\geq k)\leq C'k^{1+\alpha'}t^{-\alpha'}$}. Using this, we get that Condition (iv)' is satisfied with $\ell:=\alpha'$ and $a:=C'$. 
In both cases, we can thus apply Lemma \ref{unificationlemma} to concludes the proof of Theorem \ref{thm:coexistenceY}.
\end{proof}	

\begin{appendix}

\section{Existence, uniqueness and stochastic monotonicity}\label{sec:ExistenceUniqueness}
In this section, we show the existence of a pathwise unique strong solution in the unit interval for the SDEs \mofe{\eqref{eq:SDEWFP} and~\eqref{eq:SDE_Y}}. We also prove that the corresponding processes are stochastically monotone. \mofe{To this end}, we invoke~\citep[Thm.~6.1]{li2012strong} and~\citep[Thm.~2.3]{DL12}. For the sake of completeness, we restate  those results in a \mofe{weaker form that is well-adapted to our setting. More precisely, consider $b\in\Cs^{1}([0,1])$, and for $i\in\{1,2\}$, let $D_i \subset \R^{n_i}$ be closed, $\mu_i\in\cM(D_i)$ and $g_i:[0,1]\times D_i\to\R$ be Borel measurable. For $i\in\{1,2\}$, let $N_i(\dd t,\dd u)$ be a Poisson measure on $[0,\infty)\times D_i$ with intensity $\dd t \times \mu_i(\dd u)$. Assume that $N_1$ and $N_2$ are independent.  Let $(\cF_t)_{t\geq 0}$ be the augmented natural filtration generated by ${N_1}$ and $N_2$. We are interested in SDEs of the form
\begin{equation}\label{eq:SDEgeneric}
	\dd Z_t = b(Z_{t})\dd t + \int_{D_1} g_1(Z_{t-}, u) N_1(\dd t,\dd u)+\int_{D_2} g_2(Z_{t-}, u) N_2(\dd t,\dd u). 
\end{equation}
By a \emph{strong solution} of the SDE \eqref{eq:SDEgeneric}
we mean a \cadlag and~$(\cF_t)_{t\geq 0}$-adapted real process~$(Z_t)_{t\geq 0}$ satisfying~\eqref{eq:SDEgeneric} almost surely for all~$t\geq 0$. We say that \emph{pathwise uniqueness} holds for~\eqref{eq:SDEgeneric} if for any two solutions~$Z^1$ and~$Z^2$, satisfying  $Z^1_0=Z^2_0$ , we have $Z^1_t=Z^2_t$ almost surely for all~$t\geq 0$. We formulate now \citep[Thm. 6.1]{li2012strong} and \citep[Thm.~2.3]{DL12} in the form we will use it in this section.}
\begin{theo}(\citep[Thm. 6.1]{li2012strong},\citep[Thm.~2.3]{DL12}) \label{thm:lipu} Assume that, for $i\in\{1,2\}$, there is $D_{i}^*\subseteq D_i$ with $\mu_i(D_i\setminus D_{i}^*)<\infty$ such that the following conditions are satisfied. \begin{enumerate}
		\item There is a constant $C>0$ such that, for any $z,y\in[0,1]$, 
		\begin{align*}
	& \sum_{i=1}^2\int_{D_i^*} \lvert g_i(z,u)-g_i(y,u)\rvert \mu_i(\dd u)\leq C \lvert z-y\rvert,
\end{align*}
\item There is a constant $K\geq 0$ such that for all $z\in[0,1]$,
			\begin{equation*}
	\sum_{i=1}^2\int_{D_i^*} \lvert g_i(z,u)\rvert \mu_i(\dd u)\leq K.
		\end{equation*}	
		\item For $i\in\{1,2\}$ and every $u\in D_i$ and $z\in[0,1]$, we have $z+g_i(z,u)\in [0,1]$. Moreover, $b(0)\geq0$, $b(1)\leq 0$, and for $i\in\{1,2\}$ and $u\in D_i$, ${g_i}(0,u)={g_i}(1,u)=0$.

	\end{enumerate}
	Then for any $z_0\in [0,1]$, there is a pathwise unique strong solution $(Z_t)_{t\geq0}$ to~\eqref{eq:SDEgeneric} such that $Z_0=z_0$. The solution is such that $Z_t\in[0,1]$ for all $t\geq 0$.

Furthermore, if for $i\in\{1,2\}$ and any $u\in D_i$ the function $z\mapsto z+{g_i}(z,u)$ is non-decreasing, then for any two solutions $Z$ and $Z^{'}$ of~\eqref{eq:SDEgeneric} satisfying $0\leq Z_0\leq Z_0^{'}\leq 1$ we have $\P(Z_t\leq Z_t^{'} \ \text{for all }t\geq 0)=1$; we say that the SDE satisfy the monotonicity property.
\end{theo}
\begin{remark}
To deduce the previous theorem from \citep[Thm. 6.1]{li2012strong} and \citep[Thm.~2.3]{DL12}, one has first to extend the functions $b,{g_i}$, $i\in\{1,2\}$, in the variable $z$ from $[0,1]$ to $\R$ by setting ${g_i}(z,\cdot)$ equal to $0$ for $z\notin[0,1]$, and ${b}(z)\coloneqq{b}(0)$ for $z<0$ and ${b}(z)\coloneqq {b}(1)$ for $z>1$. Note also that \citep[Thm. 6.1]{li2012strong} provides the existence of a non-negative solution $Z$; to show that $Z$ has values on $[0,1]$, one applies \citep[Thm. 6.1]{li2012strong} to the SDE satisfied by $1-Z$.
\smallskip

Condition (3) on $g_i$, $i\in\{1,2\}$, is not directly as in~\citep[Thm. 6.1]{li2012strong}. The condition there is that for all $z\in \R$, $z+{g_i}(z,\cdot)\in[0,1]$, which is clearly not satisfied here for $z\notin[0,1]$. However, the proof of their result relies on \citep[Prop. 2.1]{FuLi10} and the corresponding proof works without modifications under the alternative Condition (3) in Theorem~\ref{thm:lipu}.
\end{remark}
We now apply Theorem~\ref{thm:lipu} to the SDE~\eqref{eq:SDEWFP}. 
\begin{prop}[Existence, pathwise uniqueness and monotonicity for SDE~\eqref{eq:SDEWFP}]\label{prop:existuniqueX}
Let $(\Lambda,\mu,\sigma)\in \Theta$. For any~$x_0\in[0,1]$, there is a pathwise unique strong solution~$X$  to~\eqref{eq:SDEWFP} such that $X_0=x_0$ and $X_t\in [0,1]$ for all~$t\geq 0$.
Moreover, if $X$ and $X'$ are the solutions to~\eqref{eq:SDEWFP} such that $0\leq X_0=x_0\leq x_0'=X_0'\leq 1$, then
	\begin{equation}
	{\P_{x_0,x_0'}}(X_t\leq X_t'\ \text{ for all }t\geq 0)=1. \label{eq:monotinictyX}
	\end{equation}\end{prop}
\begin{proof}
Let $D_1=D_1^*=[0,1]^2$, $D_2=D_2^*=[-1,1]$. Consider the measures on $D_1$ and $D_2$ given respectively via  $$\mu_1(\dd r,\dd u)=\1_{\{r\in(0,1]\}}r^{-2}\Lambda(\dd r)\times \dd u,\quad\mu_2(\dd r)=\1_{\{r\in(-1,1)\}}\mu(\dd r).$$ Set $b(x)\coloneqq x(1-x)\sel(x)$, for $x\in[0,1]$. We also define ${g}_2(x,r)\coloneqq rx(1-x)$ for $x\in[0,1]$ and $r\in [-1,1]$. Similarly, we consider ${g}_1(x,r,u)\coloneqq r((1-x)\1_{\{x\geq u>0\}}-x \1_{\{x<u\}})$, for $x\in[0,1]$ and $(r,u)\in [0,1]^2$. Clearly, with this choice of parameters the SDE \eqref{eq:SDEgeneric} becomes~\eqref{eq:SDEWFP}. 
Now, we verify the conditions of Theorem~\ref{thm:lipu}. 
First, one can easily show that for $0\leq z<y\leq 1$, $|g_1(z,r,u)-g_1(y,r,u)|\leq 2r (|z-y|+ \1_{\{z\leq u\leq y\}})$. Similarly, we have $\lvert {g}_2(z,r)-g_2(y,r)\rvert\leq \lvert r\rvert \times \lvert z-y\rvert$. Using these two inequalities and the fact that $(\Lambda,\mu,\sigma)\in\Theta$ yields condition (1).
Moreover, we have $|{g}_1(z,r,u)|\leq r$ and
$|{g}_2(z,r)|\leq |r|$, which implies that condition (2) holds. Note also that $b(0)=b(1)=0$ and $g_1(0,r,u)=g_1(1,r,u)=0$ and $g_2(0,r)=g_2(1,r)=0$. Moreover, straightforward calculation shows that the functions
$z\mapsto r \1_{\{z\geq u>0\}}+(1-r)z$ and $z\mapsto z+rz(1-z)$ are non-decreasing. In particular, for all $z\in[0,1]$, we have $z+g_1(z,r,u)\in[0,1]$ and $z+g_2(z,r)\in[0,1]$. Therefore, condition (3) and the condition for monotonicity are satisfied. Hence, existence, pathwise uniqueness, and stochastic monotonicity of strong solutions of SDE \eqref{eq:SDEWFP} and the positive invariance of $[0,1]$ follow from Theorem~\ref{thm:lipu}. 
\end{proof}
The next result states that the solutions of SDE \eqref{eq:SDEWFP} enjoy the Feller property. 
\begin{prop}[Feller property for SDE \eqref{eq:SDEWFP}]\label{Feller}
Let $(\Lambda,\mu,\sigma)\in \Theta$. For $x\in[0,1]$, denote by $(X_t^x)_{t\geq 0}$ the solution of SDE \eqref{eq:SDEWFP} such that $X_0^x=x$. Define for $f\in\Cs([0,1])$, 
$$T_t f(x)\coloneqq\E[f(X_t^x)],\quad x\in[0,1].$$
Then, $(T_t)_{t\geq 0}$ is a Feller semigroup. 
In addition, the infinitesimal generator~$\cL$ of~$X$ acts on $f\in \Cs^2([0,1])$ via  
$\cL f(x)=(\cL_\sigma +\cL_{\mu}+\cL_{\Lambda})f(x)$, $x\in[0,1]$, where 
\begin{align*}
			\cL_\sigma f(x)&\coloneqq x(1-x)\sigma(x)f'(x),\quad\cL_{\mu}f(x)\coloneqq \int_{(-1,1)}(f(x+rx(1-x))-f(x))\mu(\dd r),\\
			\cL_{\Lambda}f(x)&\coloneqq \int_{(0,1]\times (0,1)}\big(f(x + r(\1_{\{u\leq x\}}(1-x)-\textbf{1}_{\{u>x\}} x ))-f(x)\big)\frac{\Lambda(\dd r)}{r^2}\dd u. 
	\end{align*}
\end{prop}
\begin{proof}
Thanks to the monotonicity property of SDE \eqref{eq:SDEWFP} and \mofe{It\^o's formula}, we have for $0\leq x\leq z\leq 1$, 
\modifbis{$$\E[\lvert X_t^z-X_t^x\rvert]=\E[X_t^z-X_t^x]=z-x+\int_0^t \left( \tilde \cL i(X_s^z)-\tilde \cL i(X_s^x)\right)\dd s,$$
where $i:[0,1]\to[0,1]$ denotes the identity function, and $\tilde \cL:=\cL_\sigma +\cL_{\mu}+\cL_{\Lambda}$, with $\cL_\sigma,\cL_{\mu},\cL_{\Lambda}$ as in the statement of the proposition. A straightforward calculation yields
$$\tilde \cL i(y)=y(1-y)\sigma(y)+y(1-y)\int_{(-1,1)} r \mu(\dd r).$$}
Since $\sigma\in\Cs^1([0,1])$ and $\int_{(-1,1)}|r|\mu(\dd r)<\infty$, we can conclude that there is $C>0$ such that
$$\E[\lvert X_t^z-X_t^x\rvert]\leq |z-x|+ C\int_0^t \E[\lvert X_s^z-X_s^x\rvert]\dd s.$$
Hence, using Gronwall's lemma we obtain $\E[\lvert X_t^z-X_t^x\rvert]\leq |z-x|e^{Ct}$. As a consequence, for any $f\in\Cs^1([0,1])$, we have for $z,x\in[0,1]$, 
$$\lvert T_t f(z)-T_t f(x)\rvert \leq \E[\lvert f(X_t^z)-f(X_t^x)\rvert]\leq \lVert f' \rVert_{\infty}\E[\lvert X_t^z-X_t^x\rvert]\leq \lVert f' \rVert_{\infty} e^{Ct}|z-x|.$$
Now, let $g\in\Cs([0,1])$ and let $(f_n)_{n\geq 1}$ be a sequence of functions in $\Cs^1([0,1])$ such that $f_n\to g$ uniformly in $[0,1]$ as $n\to\infty$. Then, using triangular inequality, we get
\begin{align*}
\lvert T_t g(z)-T_t g(x)\rvert&\leq \lvert T_t g (z)-T_t f_n(z)\rvert+ \lvert T_t f_n (z)-T_t f_n(x)\rvert+\lvert T_t f_n (x)-T_t g(x)\rvert\\
&\leq 2 \lVert f_n -g \rVert_{\infty}+ \lVert f_n' \rVert_{\infty} e^{Ct}|z-x|.
\end{align*}
Thus, for all $x\in[0,1]$ and $n\geq 1$, we have
$$\limsup_{z\to x}\lvert T_t g(z)-T_t g(x)\rvert\leq 2 \lVert f_n -g \rVert_{\infty}.$$
Letting $n\to\infty$ in the previous inequality, shows that the map $x\mapsto T_t g(x)$ is continuous, i.e. $T_t$ maps $\Cs([0,1])$ into $\Cs([0,1])$. It remains to prove that for any $g\in\Cs([0,1])$, $T_tg\to g$ uniformly as $t\to 0$. We prove this for any $g\in\Cs^2([0,1])$; the result for all $g\in\Cs([0,1])$ follows via an approximation argument similar to the one we have just used. 

Fix $g\in\Cs^2([0,1])$. Using Taylor expansions of order one and two around $x$, one can show that there are constants $C_1,C_2>0$ (depending on $\lVert g'\rVert_\infty$ and $\lVert g''\rVert_\infty$) such that 
$$\lvert g(x+rx(1-x))-g(x)\rvert\leq C_1\lvert r\rvert \quad\textrm{and}\quad \lvert xg(x+r(1-x))+(1-x)g(x-rx)-g(x)\rvert \leq C_2 r^2.$$ 
We conclude that there is a constant $C_3>0$ (depending on $\lVert g'\rVert_\infty$ and $\lVert g''\rVert_\infty$) such that, for all $x\in[0,1]$, $\lvert \cL g(x)\rvert \leq C_3.$ Using this and \mofe{It\^o's formula}, we obtain
\modifbis{$$\lvert T_t g(x)-g(x)\rvert=\lvert \E[g(X_t^x)]-{g(x)}\rvert\leq \int_0^t \lvert \E[\tilde \cL g(X_s^x)]\rvert \dd s\leq C_3 t.$$}
Therefore,
$\lVert T_t g-g \rVert_\infty\leq C_3 t\xrightarrow[t\to 0]{}0.$
\mofe{Hence, $T_t$ is Feller. The expression of the generator~$\cL$ is obtained via a standard application of It{\^o}'s formula (see e.g. \citep[Thm.~4.4.7]{Applebaum2009}).}
\end{proof}

We now apply Theorem~\ref{thm:lipu} to the SDE~\eqref{eq:SDE_Y}. 
\begin{prop}[Existence, pathwise uniqueness and monotonicity property for SDE~\eqref{eq:SDE_Y}]\label{lem:existuniqueY}
Let $(\Lambda,\mu,\sigma)\in \Theta$. For any~$y_0\in[0,1]$, there exists a pathwise unique strong solution~$Y\coloneqq (Y_t)_{t\geq 0}$ to~\eqref{eq:SDE_Y} such that $Y_0=y_0$ and $Y_t\in [0,1]$ for all~$t\geq 0$. 
Moreover, if $0\leq y_0\leq y_0'\leq 1$ and $Y$ and~$Y'$ are the solutions to~\eqref{eq:SDE_Y} such that $Y_0=y_0$ and $Y_0'=y_0'$, respectively, then 
	\begin{equation}
		{\P_{y_0,y_0'}}(Y_t\leq Y_t'\ \text{for all }t\geq 0)=1 .\label{eq:monotinictyY}
	\end{equation}
\end{prop}
\begin{proof}
For $y\in[0,1]$, set  $b(y)=-y(1-y)\sel(y)$. \mofe{Let $D_{1}=[0,1]^2$, $D_{2}=[-1,1]$. We consider the Poisson measures $N_1=N$ and $N_2=S$, and the functions $g_{1}(z,r,u)\coloneqq m_{r,u}(z)-z$ and $g_{2}(z,r)\coloneqq s_r(z)-z)$) so that SDE \eqref{eq:SDEgeneric} becomes \eqref{eq:SDE_Y}. Consider now the sets $D_{1}^*=[0,1/2]\times[0,1]\subset D_{1}$ and $D_{2}^*=[-1/2,1/2]\subset D_{2}$. Since 
$$\int_{D_{1}\setminus D_{1}^*}\frac{\Lambda(\dd r)}{r^2}\dd u<\infty\quad\textrm{and}\quad \int_{D_{2}\setminus D_{2}^*}\mu(\dd r)<\infty,$$
checking conditions (1) and (2) in Theorem~\ref{thm:lipu} amounts to prove that
$$\int_{D_{1}^*}\lvert m_{r,u}(z)-z\rvert \frac{\Lambda(\dd r)}{r^2}\dd u+\int_{D_{2}^*}\lvert s_r(z)-z \rvert \mu(dr)\leq C,$$
 and $$\int_{D_{1}^*}\lvert m_{r,u}(z)-z-(m_{r,u}(y)-y)\rvert \frac{\Lambda(\dd r)}{r^2}\dd u+\int_{D_{2}^*}\lvert s_r(z)-z-(s_r(y)-y) \rvert \mu(dr)\leq C\lvert z-y \rvert,$$
for some $C>0$}. \modifbis{Since we assume $(\Lambda,\mu,\sigma)\in \Theta$}, these inequalities follow directly from Lemmas \ref{hr2} and \ref{bsr}.  

Clearly for $r\in[-1,1]\setminus \{0\}$ and $z\in[0,1]$, $s_r(z)\in[0,1]$, $s_r(1)-1=s_r(0)=0$. Similarly, for $r,u\in(0,1)$ and $z\in[0,1]$, $m_{r,u}(z)\in[0,1]$, and $m_{r,u}(1)-1=m_{r,u}(0)=0$.
We also have $b(0)=b(1)=0$, and hence, condition (3) is satisfied.  
This already proves the existence and pathwise uniqueness of strong solutions of SDE \eqref{eq:SDE_Y} as well as the positive invariance of $[0,1]$. Since the functions $z\mapsto s_r(z)$ and $z\mapsto m_{r,u}(z)$ are non-decreasing, the monotonicity property also follows from Theorem~\ref{thm:lipu}.
\end{proof}
\begin{prop}\label{fullgenerator}
Let $(\Lambda,\mu,\sigma)\in \Theta$. Consider the linear operator $\cA$ acting on $f\in \Cs^2([0,1])$ via
\begin{align*}
\cA f(y)&= - y(1-y)\sel(y)f'(y) +\int_{(0,1]^2} \big( f(\mr_{r,u}(y)) -f(y)\big) \frac{\Lambda(\dd r)}{r^2}\dd u\\
&\quad+\int_{(-1,1)} \big (f(s_r(y))-f(y) \big )\mu(\dd r).
\end{align*}	
Then the solution $Y$ of \eqref{eq:SDE_Y} is the unique solution of the martingale problem $(\cA,\Cs^2([0,1]))$. Moreover, $Y$ is strongly Markovian. 
\end{prop}
\begin{proof}
A standard application of It{\^o}'s formula (see e.g. \citep[Thm.~4.4.7]{Applebaum2009}) shows that $Y$ is a solution of the martingale problem. Is it proven in Proposition~\ref{lem:existuniqueY} that $Y$ exists, is pathwise unique, and remains in~$[0,1]$. It follows from \citep[Thm.~2.3]{Kurtz2011} that every solution to the martingale problems is a weak solution to the SDE \eqref{eq:SDE_Y}. Because pathwise uniqueness implies weak uniqueness \citep[Thm.~1]{BLP15}, the martingale problems are well-posed, completing the proof of the first statement.

Note also that $\cA$ maps $\Cs^2{[0,1]}$ to $\Cs{[0,1]}$. This follows from the continuity of the maps $y\mapsto m_{r,u}(y)$ and $y\mapsto s_r(y)$, and the theorem of continuity for parameter-dependent integrals. Thus, using the well-posedness of the martingale problem $(\cA,\Cs^2([0,1]))$, \citep[Thm. 4.4.6]{EK86} and \citep[Rem. 2.5]{Kurtz1998}, we conclude that $y\mapsto\P_y(Y\in B)$ is measurable, for any  Borel set $B\subset\Db([0,\infty))$ of c\`adl\`ag functions. Hence, the strong Markov property for $Y$ is obtained applying \citep[Thm. 4.4.2]{EK86}.
\end{proof}

We end this section with a result that will help us get rid of the small and large jumps of $X$ and $Y$ in Appendix \ref{sec:proofsiegdual}. For this, fix $z\in[0,1]$, let  $(\Lambda,\mu,\sigma)\in \Theta$ and consider for $0\leq \varepsilon<\rho\leq 1$ the solution $Z^{\varepsilon,\rho}$ of
\begin{align*}
	\dd Z^{\varepsilon, \rho}_t  &=b(Z^{\varepsilon, \rho}_t)\dd t+\int_{(\varepsilon,\rho)\times (0,1)} \!\!g_1(r,u,Z^{\varepsilon, \rho}_{t-}) N(\dd t, \dd r,\dd u)+\int_{\{\varepsilon\leq|r|\leq \rho\}} \!\!g_2(r,Z^{\varepsilon, \rho}_{t-}) S(\dd t,\dd r), 
\end{align*} 
with $Z^{\varepsilon,\rho}_0=z$, where $b:\pm x\mapsto x(1-x)\sigma(x)$, and the functions $g_1$ and $g_2$ are such that for any $\rho<1$, there is a constant $C_\rho>0$ such that for any $r\in(0,\rho)$, $a\in\{-1,1\}$ and $z,w \in (0,1)$
\begin{align}
\int_0^1 |g_1(r,u,z)| \dd u \leq C_{\rho} r,& \qquad \int_0^1 |g_1(r,u,z)-g_1(r,u,w)| \dd u  \leq C_{\rho} |z-w| r, \label{hypg1} \\
|g_2(ar,z)| \leq C_{\rho} r,& \qquad\ |g_2(ar,z)-g_2(ar,w)|  \leq C_{\rho} |z-w|  r, \label{hypg2}
\end{align} 
A direct application of Theorem \ref{thm:lipu} shows that, under these conditions, $Z^{\varepsilon,\rho}$ is well-defined for any choice of $0\leq \varepsilon<\rho\leq 1$.
\begin{prop} \label{approxz}
For any $t \geq 0$ and $\rho \in [1/2,1)$, $Z^{\varepsilon, \rho}_t$ converges in distribution to $Z^{0, \rho}_t$ as $\varepsilon\to0$. For any $t \geq 0$, $Z^{0, \rho}_t$ converges in distribution to $Z^{0, 1}_t$ as $\rho\to 1$.  
\end{prop} 
\begin{proof}
Subtracting the SDEs defining $Z^{0,\rho}$ and $Z^{\varepsilon,\rho}$ and taking absolute values, we get that almost surely surely, 
\begin{align*}
&|Z^{0, \rho}_t - Z^{\varepsilon, \rho}_t| \leq  \!\int_{[0,t] \times (0,\varepsilon] \times (0,1)}\!\! \big|g_1(r,u,Z^{0, \rho}_{s-})\big|N(\dd s, \dd r, \dd u) +\! \int_{[0,t] \times [-\varepsilon,\varepsilon]}\!\! \big|g_2(Z^{0, \rho}_{s-})\big|S(\dd s, \dd r) \\
 &\qquad+ \int_{[0,t] \times (\varepsilon,\rho) \times (0,1)} \big|g_1(r,u,Z^{0, \rho}_{s-})-g_1(r,u,Z^{\varepsilon, \rho}_{s-})\big|N(\dd s, \dd r, \dd u) \\
 &\qquad+ \int_{[0,t] \times (-\rho,\rho)\setminus[-\varepsilon,\varepsilon]} \big|g_2(r,Z^{0, \rho}_{s-})-g_2(r,Z^{\varepsilon, \rho}_{s-})\big|S(\dd s, \dd r) + \int_0^t \big | b(Z^{0, \rho}_s) - b(Z^{\varepsilon, \rho}_s) \big | \dd s. 
\end{align*}
Taking expectations, using the compensation formula, \eqref{hypg1}, \eqref{hypg2} and that $b\in \Cs^1$,  we obtain 
\begin{align*}
\mathbb{E}_z \big [ |Z^{0, \rho}_t - Z^{\epsilon, \rho}_t| \big ] & \leq C_{\rho} \delta(\varepsilon) t + (C_{\rho} \delta(\rho)+ \lVert b' \rVert) \int_0^t \mathbb{E}_z \big [ \big | Z^{0, \rho}_s - Z^{\epsilon, \rho}_s \big | \big ] \dd s. 
\end{align*}
where $\delta(a)\coloneqq  ( \int_{(0,a]} r^{-1} \Lambda(\dd r) + \int_{[-a,a]} |r| \mu(\dd r))$.
Thus, using Gronwall's lemma we get
\begin{align*}
& \mathbb{E}_z \big [ |Z^{0, \rho}_t-Z^{\varepsilon, \rho}_t| \big ] \leq C_\rho \delta(\varepsilon) \,t\,e^{\left(C_\rho\delta(\rho)+\lVert b'\rVert\right)\,t}.
\end{align*}
Since $\delta(\varepsilon)\to 0$ as $\varepsilon\to 0$, the first result follows. The second convergence result follows noticing that
\begin{align}\label{unifconv}
\mathbb{P}_z(Z^{0, \rho}_t \neq Z^{0, 1}_t) & \leq \mathbb{P}(N([0,t] \times [\rho,1) \times (0,1))+S([0,t] \times (-1,-\rho] \cup [\rho,1))\geq 1)\nonumber \\
& =1-e^{-t(\mu((-1,-\rho] \cup [\rho,1))+\int_{[\rho,1)} r^{-2} \Lambda(\dd r))}\to 0\quad\textrm{as $\rho\to 1$}. 
\end{align}
\end{proof}
\section{Proof of the Siegmund duality} \label{sec:proofsiegdual}

In this section we show Theorem \ref{thm:siegmund_duality}, which states the Siegmund duality between the processes $X$ and $Y$. Because $X$ is stochastically monotone by Prop.~\ref{prop:existuniqueX} and the map $x\mapsto \P_x(X_t\geq y)$ is right-continuous for every~$y$ (follows from the Feller property of $X$ in Prop.~\ref{Feller}, see~\citep[p. 917]{Siegmund1976}), we conclude from \cite[Thm. 1]{Siegmund1976} that there is a unique Markov Process  $X^D=(X_t^D)_{t\geq 0}$ which is
Siegmund dual to $X$, i.e. $\P_x(X_t\geq y)=\P_y(x\geq X_t^D).$
Therefore, it remains to show that $X^D$ and $Y$ are equal in distribution.

One of the main issues for characterizing the Siegmund dual of a given process is that the duality function $(x,y)\mapsto \ind{x\geq y}$ usually does not belong to the domain of the generator (unless the state space is discrete). Kolokoltsov developed in \cite{Kolo11, Kolokoltsov2013} methods to circumvent this problem (see also \cite[Thm. 6.1]{CF21+} for a more direct approach for one-dimensional diffusions). One is based on a discretization argument, and the other one, on a careful analysis of the adjoint of the generator of the Siegmund dual process. In our case, those approaches are rather difficult to apply due to integrablity problems, and the lack of prior regularity of the semigroup associated to $X^D$. For this reason we adopt a different strategy. We first prove the duality relation if we remove from $X$ and $Y$ the jumps that are smaller than some $\varepsilon>0$ and larger than some $\rho<1$, and then use a limit argument as $\varepsilon\to 0$ and $\rho\to 1$ to conclude. More precisely, for $0<\varepsilon<\rho<1$, let $X^{\varepsilon,\rho}$ and $Y^{\varepsilon,\rho}$ be the solutions of the SDEs \eqref{eq:SDEWFP} and \eqref{eq:SDE_Y}, respectively, with $(N,S)$ being replaced by the Poisson measures $(N_{\varepsilon,\rho},S_{\varepsilon,\rho})$ defined via $N_{\varepsilon,\rho}(\dd t,\dd r,\dd u)=N(\dd t,\dd r,\dd u)\ind{r\in(\varepsilon,\rho)}$ and $S_{\varepsilon,\rho}(\dd t,\dd r)=S(\dd t, \dd r)\ind{|r|\in(\varepsilon,\rho)}$
(existence and uniqueness follows from Theorems \ref{prop:existuniqueX} and \ref{lem:existuniqueY}). The next result states the duality between $X^{\varepsilon,\rho}$ and $Y^{\varepsilon,\rho}$.
\begin{prop}\label{SDweak}
For all $0<\varepsilon<\rho<1$, $t\geq 0$, and $x,y\in[0,1]$, we have
\begin{equation*}
\P_x(X^{\varepsilon,\rho}_t\geq y)=\P_y(x\geq Y_t^{\varepsilon,\rho}).
	\end{equation*} 
\end{prop}
\begin{proof}
Let us first understand the (deterministic) evolution of $X^{\varepsilon,\rho}$ and $Y^{\varepsilon,\rho}$ between jumping times. For this, let $\theta(x)=\sigma(x)x(1-x)$, $x\in[0,1]$, and consider for $a,b\in[0,1]$ the ODEs 
\begin{align}
\frac{\dd }{\dd t} x_t(a) &= \theta(x_t(a)), \quad x_0(a)=a,\quad\textrm{and}\quad \frac{\dd}{\dd t} y_t(b) = -\theta(y_t(b)), \qquad y_0(b)=b. \label{deterministicduality2}
\end{align}
Since the function $x\mapsto \theta(x)$ is Lipschitz in $[0,1]$ and vanishes at $0$ and $1$, one can infer from Cauchy-Lipschitz Theorem that for any $a, b\in[0,1]$, the ODEs in \eqref{deterministicduality2} have unique global solutions, which remain in $[0,1]$. Thus, the functions $a \mapsto x_t(a)$ and $b \mapsto x_t(b)$ are well-defined on $[0,1]$, and they are increasing and of class $\mathcal{C}^1$ (this follows from classical properties of flows of ODEs). This yields, in particular, the existence of their inverses, $a\mapsto x_t^{-1}(a)$ and $b\mapsto y_t^{-1}(b)$. Note that  
\[ a= x_t(a) - \int_0^t \theta(x_s(a)) \dd s. \]
Now fix $b \in [0,1]$ and use this identity for $a:=x^{-1}_t(b)=x^{-1}_s(x^{-1}_{t-s}(b))$, to obtain 
\[ x^{-1}_t(b)=  b- \int_0^t \theta(x^{-1}_{t-s}(b)) \dd s=b - \int_0^t \theta(x^{-1}_{s}(b)) \dd s. \]
Therefore $t \mapsto x^{-1}_t(b)$ satisfies the same ODE as $t \mapsto y_t(b)$ so, by uniqueness, they are equal. This implies the following deterministic version of the Siegmund duality
\begin{equation}\label{detsieg}
x_t(a)\geq b \iff a\geq y_t(b),\qquad t\geq0, a,b\in[0,1].
\end{equation}
By construction, if $T<T'$  are consecutive jump times of $(N_{\varepsilon,\rho},S_{\varepsilon,\rho})$, we have for $t<T'-T$ 
\begin{equation}\label{betjumps}
X^{\varepsilon,\rho}_{T+t}=x_t(X^{\varepsilon,\rho}_{T})\quad\textrm{and}\quad Y^{\varepsilon,\rho}_{T+t}=y_t(Y^{\varepsilon,\rho}_{T}).
\end{equation}
Now, we consider the effect of the jumps. Set $\tilde{m}_{r,u}(x)\coloneqq  x + r(\textbf{1}_{\{u\leq x\}}(1-x)-\textbf{1}_{\{u>x\}} x )$ and $\tilde{s}_{r}(x)\coloneqq  x + rx(1-x)$, and recall that (see \eqref{eq:medianexplanation} and the subsequent explanation) 
\begin{align}
\tilde{m}_{r,u}(x) \geq y \Leftrightarrow x \geq \mr_{r,u}(y) \ \ \ \text{and} \ \ \ \tilde{s}_{r}(x) \geq y \Leftrightarrow x \geq s_{r}(y). \label{invatjumps}
\end{align}

Fix $t>0$ and $0<\varepsilon<\rho<1$. Let $K\coloneqq N_{\varepsilon,\rho}([0,t] \times (0,1)^2)+S_{\varepsilon,\rho}([0,t] \times (-1,1))$ and let $T_1 < \cdots < T_K$ be the time components of the $K$ jumps on $[0,t]$. For each $i \in \{1,\ldots,K\}$, we set $\phi_{}^i(\cdot) := \tilde{m}_{R_i,U_i}(\cdot)$ and $\psi_{}^i(\cdot) := \mr_{R_i,U_i}(\cdot)$ if $T_i$ is an arrival time of $N_{\varepsilon,\rho}$ with associated jump $(T_i,R_i,U_i)$, and we set $\phi_{}^i(\cdot) := \tilde{s}_{R_i}(\cdot)$ and $\psi_{}^i(\cdot) := s_{R_i}(\cdot)$ if $T_i$ is an arrival time of $S_{\varepsilon,\rho}$ with associated jump $(T_i,R_i)$. Combining \eqref{betjumps} with the effect of the jumps, we get
\begin{align}
X^{\varepsilon,\rho}_t &:=(x_{t-T_K} \circ \phi_{}^{K} \circ x_{T_K-T_{K-1}} \circ \cdots \circ \phi_{}^{1} \circ x_{T_1})(x),\nonumber \\
Y^{\varepsilon,\rho}_t &:=(y_{t-T_K} \circ \psi_{}^{K} \circ y_{T_K-T_{K-1}} \circ \cdots \circ \psi_{}^{1} \circ y_{T_1})(y).\label{exprxepsilon}  
\end{align}
It is implicit in \eqref{exprxepsilon} that $X^{\varepsilon,\rho}_t=x_{t}(x)$ and $Y^{\varepsilon,\rho}_t=y_{t}(y)$ if $K=0$. By the time-homogeneity of the Poisson processes $N$ and $S$ we have 
\begin{align}
Y^{\varepsilon,\rho}_t \overset{(d)}{=} \tilde Y^{\varepsilon,\rho}_t :=(y_{T_1} \circ \psi_{}^{1} \circ \ldots \circ y_{T_K-T_{K-1}} \circ \psi_{}^{K} \circ y_{t-T_K})(y). \label{expryepsilontilde}
\end{align}
Combining the expressions \eqref{exprxepsilon} and \eqref{expryepsilontilde} of $X^{\varepsilon,\rho}_t$ and $\tilde Y^{\varepsilon,\rho}_t$ with \eqref{detsieg} and \eqref{invatjumps}, we infer that for any choice of starting values  $x,y \in [0,1]$, we have almost surely $X^{\varepsilon,\rho}_t \geq y \Leftrightarrow x \geq \tilde Y^{\varepsilon,\rho}_t$. Since $Y^{\varepsilon,\rho}_t\overset{(d)}{=} \tilde{Y}^{\varepsilon,\rho}_t$, the result follows.
\end{proof}
The next lemma will help us pass to the limit when $\varepsilon\to 0$ and $\rho\to 1$ in the previous result.
\begin{lemme}\label{passlimit}
For any $t\geq 0$, we have convergence in distribution of $X^{\varepsilon,\rho}_t$ to $X_t$ and of $Y^{\varepsilon,\rho}_t$ to $Y_t$ as $\varepsilon\to 0$ and then $\rho\to 1$. Moreover, the map $y\mapsto \P_y(Y_t\in\cdot)$ is continuous for the topology of convergence in distribution.
\end{lemme}
\begin{proof}
According to Proposition \ref{approxz}, to prove the convergence results, we only need to show that the coefficients of the SDEs defining $X$ and $Y$ satisfy the assumptions \eqref{hypg1} and \eqref{hypg2}. Lemma \ref{hr2} and the estimates from the proof of Proposition \ref{prop:existuniqueX} yield \eqref{hypg1} for $Y$ and $X$, respectively. Lemma \ref{bsr} and the estimates from the proof of \ref{prop:existuniqueX} yield \eqref{hypg2} for $Y$ and $X$, respectively. For the proof of the continuity of $y\mapsto \P_y(Y_t\in\cdot)$, it is sufficient to show that for every $f\in\Cs^1$ the function $y\mapsto \E_y[f(Y_t)]$ is continuous. Proceeding analogously to the first part of the proof of Proposition \ref{Feller} and using Lemmas \ref{hr2} and \ref{bsr}, one can prove that for any $\rho\in(0,1)$, $y\mapsto \E_y[f(Y_t^{0,\rho})]$ is continuous. Hence, using \eqref{unifconv} we see that $y\mapsto \E_y[f(Y_t)]$ is the uniform limit as $\rho\to0$ of $y\mapsto \E_y[f(Y_t^{0,\rho})]$, and hence it is continuous, which ends the proof.
\end{proof}
 
Now, we can proceed with the proof of the Siegmund duality.
\begin{proof}[Proof of Theorem \ref{thm:siegmund_duality}]
Let $f,g \in C([0,1])$. Set $F(x):=\int_x^1f(y)\dd y$ and $G(x):=\int_0^xg(y)\dd y$. Note that if two processes $Z^1$ and $Z^2$ on $[0,1]$ are Siegmund dual then 
\begin{align}
\int_0^1 f(z) \mathbb{E}_z[G(Z^1_t)] \dd z & = \int_0^1  \int_0^1 f(z) g(w) \mathbb{P}_z (w \leq Z^1_t) \dd w \dd z \nonumber \\ 
& = \int_0^1  \int_0^1 f(z) g(w) \mathbb{P}_w (Z^2_t \leq z) \dd w \dd z = \int_0^1 \mathbb{E}_w[F(Z^2_t)] g(w) \dd w. \label{siegalt}
\end{align}
Therefore, it follows by Proposition \ref{SDweak} that
\begin{align}
\int_0^1 f(z) \mathbb{E}_z[G(X^{\varepsilon, \rho}_t)] \dd z = \int_0^1 \mathbb{E}_w[F(Y^{\varepsilon, \rho}_t)] g(w) \dd w, \label{siegalteps}
\end{align}
for all $f,g \in C([0,1]$, $\rho \in [1/2,1)$ and $\varepsilon \in (0,1/2)$. Letting $\varepsilon$ go to $0$ and then $\rho$ go to $1$ in \eqref{siegalteps} and using Lemma \ref{passlimit} we get 
\begin{align}
\int_0^1 f(z) \mathbb{E}_z[G(X_t)] \dd z = \int_0^1 \mathbb{E}_w[F(Y_t)] g(w) \dd w. \label{siegalt01}
\end{align} 
Applying \eqref{siegalt} to the pair $(X,X^D)$ and combining with \eqref{siegalt01} we obtain 
\begin{align}
\int_0^1 \mathbb{E}_y[F(X^D_t)] g(y) \dd y = \int_0^1 \mathbb{E}_y[F(Y_t)] g(y) \dd y. \label{identificationxdy}
\end{align} 
We deduce that, for almost every $y \in [0,1]$, the probability distributions $\mathbb{P}_y(X^D_t \in \cdot)$ and $\mathbb{P}_y(Y_t \in \cdot)$ are equal. By Lemma \ref{passlimit}, $y \mapsto \mathbb{P}_y(Y_t \in \cdot)$ is continuous on $[0,1]$. Moreover, by Proposition B.5 and the monotonicity of the Siegmund dual, the functions $y \mapsto \mathbb{P}_y(X^D_t \in \cdot)$ and $y \mapsto \mathbb{P}_y(Y_t \in \cdot)$ are non-decreasing (in the sense of stochastic monotonicity). We thus deduce that $\mathbb{P}_y(X^D_t \in \cdot)$ and $\mathbb{P}_y(Y_t \in \cdot)$ are equal for all $y \in [0,1]$. Hence, the result follows using the Siegmund duality between $X^D$ and $X$.
\end{proof}
%

\begin{proof}[Proof of Corollary \ref{2dual}]
Let $0\leq \hat{x}< \check{x}\leq 1$, $y\in [0,1]$ and $t\geq 0$. Since the solutions of SDE \eqref{eq:SDEWFP} are stochastically monotone (see Proposition~\ref{prop:existuniqueX}), 
$\P_{\hat{x},\check{x}}(\widehat{X}_t<y\leq \check{X}_t)={\P_{\check{x}}(y\leq \check{X}_t)-\P_{\hat{x}}(y\leq \widehat{X}_t).}$
Using Theorem~\ref{thm:siegmund_duality}, we obtain
$\P_{\hat{x},\check{x}}(\widehat{X}_t<y\leq \check{X}_t)=\P_y(Y_t \leq \check{x})-\P_y(Y_t \leq \hat{x})=\P_y(\hat{x}<Y_t \leq \check{x}),$
achieving the proof of the first identity. The proof of the second identity is analogous.
\end{proof}
\section{Some auxiliary results on L{\'e}vy processes}\label{sec:app:levy}
In this section, we collect some results on L{\'e}vy processes that are useful in Sections~\ref{sec:asextinction} and~\ref{sect:coex}. These results are standard, but we could not find a reference in the \mofe{required form}. For a real-valued L{\'e}vy-process $L\coloneqq (L_t)_{t\geq 0}$ that admits a Laplace transform, denote by $\psi_L$ the corresponding Laplace exponent, i.e. if $\lambda\in \R$ is in the domain of the Laplace transform and $t\geq 0$, then $\E[e^{\lambda L_t}]=e^{t\psi_L(\lambda)}$. Moreover, define $G_L\coloneqq \sup\{t\geq 0:\, L_t\vee L_{t-}=\sup_{s\in[0,\infty)} L_s\}$. For $L$, $(\gamma_L,Q_L,\Pi_L)$ is the corresponding L{\'e}vy triplet, i.e. $\gamma_L\in \R$, $Q_L\geq 0$, and $\Pi_L$ is a L{\'e}vy measure.
\begin{lemme}\label{lem:auxlevyexp}
	Let $L$ be a real-valued L{\'e}vy process. Assume there is $\lambda_0>0$ in the domain of the Laplace transform of~$L$ such that $\psi_L(\lambda_0)<0$. Then, a) $L_t\xrightarrow{t\to\infty} -\infty$ a.s.  (in particular, $G_L$ is well defined and finite a.s.) and b) for any $r\in (0,-\psi_L(\lambda_0))$, $\E[e^{r G_L}]<\infty$.
\end{lemme}
\begin{proof}
	For the first part, we use Rogozin's criterion, i.e. if $L$ is not identically zero, and if $\int_1^\infty t^{-1}\P(L_t\geq 0)\dd t<\infty$, then $\lim_{t\to\infty}L_t=-\infty$ almost surely~\cite[Thm.~VI.12(i)]{Bertoin1996}. Let $\lambda_0>0$ such that $\psi_L(\lambda_0)<0$. Then, \begin{equation}
		\P(L_t\geq 0)\leq \E[\ind{L_t\geq 0}e^{\lambda_0L_t}]\leq \E[e^{\lambda_0L_t}]=e^{t\psi_L(\lambda_0)}.\label{eq:rogoineq}
	\end{equation} In particular, $\int_1^\infty t^{-1}\P(L_t\geq 0)\, \dd t<\infty$ and by Rogozin's criterion~$L$ drifts to $-\infty$. For the second part, let $\tau(q)$ be an independent exponentially distributed random variable with parameter~$q>0$, \modifbis{$S_{\tau(q)}\coloneqq \sup_{s\in[0,\tau(q)]} L_s$ and $G_{\tau(q)}\coloneqq \sup\{t \leq \tau(q): L_t\vee L_{t-}=S_{\tau(q)}\}$. By~\citep[Lem.~VI.8]{Bertoin1996}, $G_{\tau(q)}$ is infinitely divisible, and letting $q\to0$, we deduce that $G_L$ also has this property}. It is known that for $\lambda>0$, $\E[e^{-\lambda G_{\tau(q)}}]=\exp(\int_0^\infty (\exp(-\lambda  t)-1)t^{-1}e^{-qt}\P(L_t \geq 0) \dd t)$ \citep[Eq.~VI.(5) or Thm.~VI.5(ii))]{Bertoin1996}. Hence, again letting $q\to0$ and using the L{\'e}vy-Khintchine formula, we deduce that the L{\'e}vy measure of $G_L$ is $\nu(\dd t)=t^{-1}\P(L_t\geq 0)\dd t$. For $r\in (0,-\psi_L(\lambda_0))$, by \citep[Thm.~25.3]{KenIti1999}, $G_L$ has exponential moments of order~$r$ if and only if $\int_1^\infty e^{rt}\nu(\dd t)<\infty$, which can be seen to be satisfied by $\nu$ using~\eqref{eq:rogoineq}.
\end{proof}
\begin{lemme}\label{lem:aux:levy}
	Let~$L$ be a real-valued L{\'e}vy process. \modifbis{Assume }$\E[L_1]<0$ and there is $\theta>0$ such that $\int_1^{\infty} x^{1+\theta}\Pi_L(\dd x)<\infty$. Then, for any $r\in (0,\theta)$, there exists $C>0$ such that for sufficiently large~$t$, $\P(L_t\geq 0)\leq C t^{-r}$.
\end{lemme}
\begin{proof}
	The idea is to decompose~$L$ into jumps that are small-positive-or-negative, medium-positive, and large-positive relative to time passed; then we deal with each term separately. Fix $r\in (0,\theta)$. Let $t>1$ and set $a\coloneqq t^{(1+r)/(1+\theta)}$. Consider the decomposition $L=L^{(1)}+L^{(a)}+L^{(2)}$, where $L^{(1)}$, $L^{(a)}$, and $L^{(2)}$ are independent L{\'e}vy processes with L{\'e}vy triplets $(\gamma_L,Q_L,\Pi_L(\, \cdot\, \cap (-\infty,1])$, $(0,0,\Pi_L(\, \cdot\, \cap (1,a) ) )$, and $(0,0,\Pi_L(\, \cdot\, \cap [a,\infty)). $
	Then, $\P(L_t\geq 0)\leq \P(L_t^{(2)}\neq 0)+ \P(L_t^{(1)}+L_t^{(a)}\geq 0)$. We now bound the first and the second term on the right-hand side of this inequality.
	For the first term, note that \begin{align*}
		\P(L_t^{(2)}\neq 0)&\leq 1-e^{-t\Pi_L([a,\infty))}\leq t\, \Pi_L([a,\infty))\leq ta^{-(1+\theta)}\int_{[a,\infty)}x^{1+\theta}\Pi_L(\dd x)\\
		&\leq t^{-r}\int_{[1,\infty)}x^{1+\theta}\Pi_L(\dd x).
	\end{align*}
	Next, we deal with the second term, i.e. $\P(L_t^{(1)}+L_t^{(a)}\geq 0)$. Note that, by construction, the L\'evy processes $L^{(1)}$ and $L^{(a)}$ have bounded positive jumps so they have Laplace transforms defined on $[0,\infty)$, thanks to \citep[Thm.~25.3]{KenIti1999}. By the L{\'e}vy-Khintchine formula, the Laplace exponent of $L^{(1)}$ and $L^{(a)}$, respectively, is \begin{align*}
		\psi_{L^{(1)}}(\lambda)&=\gamma_L\lambda+ \frac{Q_L}{2}\lambda^2+\int_{(-\infty,1]} (e^{\lambda x}-1-\lambda x\1_{[-1,1]}(x))\Pi_L(\dd x),\\ \psi_{L^{(a)}}(\lambda)&=\int_{(1,a)} (e^{\lambda x}-1)\Pi_L(\dd x).
	\end{align*}
	Fix $\varepsilon\in (0,\frac{\theta-r}{1+r})$. Using the classical estimate $\psi_{L^{(1)}}(\lambda) \sim \lambda \psi_{L^{(1)}}'(0) = \lambda \mathbb{E}[L^{(1)}_1]$ and Taylor's theorem (\modifbis{together with $\int_{(1,a)} (\lambda x)^2 \Pi_L(\dd x) \leq \lambda^2 a \int_{(1,a)} x \Pi_L(\dd x)$}) we get that as $a\to\infty$,  \begin{align*}
		\psi_{L^{(1)}}(a^{-(1+\varepsilon)})& \sim a^{-(1+\varepsilon)}\Big(\gamma_L+\int_{(-\infty,-1)}x\Pi_L(\dd x)\Big), \\ \psi_{L^{(a)}}(a^{-(1+\varepsilon)})&=a^{-(1+\varepsilon)}\int_{(1,a)}x\Pi_L(\dd x) + \modifbis{O(a^{-(1+2\varepsilon)})}.
	\end{align*}  In particular, \modifbis{$\limsup_{a \rightarrow \infty} a^{1+\varepsilon}(\psi_{L^{(1)}}(a^{-(1+\varepsilon)})+\psi_{L^{(a)}}(a^{-(1+\varepsilon)}))\leq \gamma_L+\int_{\R\setminus[-1,1] } x\Pi_L(\dd x)=\E[L_1]$}. Now, for sufficiently large~$t$, 
\begin{align*}
	\P(L_t^{(1)}+L_t^{(a)}\geq 0) &\leq \E\left[\ind{L_t^{(1)}+L_t^{(a)}\geq 0}\exp\left(a^{-(1+\varepsilon)}\left(L_t^{(1)}+L_t^{(a)}\right) \right)\right] \\
	&\leq \exp\left(t\psi_{L^{(1)}}(a^{-(1+\varepsilon)})+t\psi_{L^{(a)}}(a^{-(1+\varepsilon)})\right) \\
	&\leq \exp\left(ta^{-(1+\varepsilon)}\E[L_1]/2\right)= \exp\left(t^{(\theta-r-\varepsilon(1+r))/(1+\theta)}\E[L_1]/2\right), 
\end{align*}
where in the last row we used the definition of~$a$. Recall that $\E[L_1]<0$ by assumption. Since $ (\theta-r-\varepsilon(1+r))/(1+\theta)>0$ and because the exponential decay is faster than any polynomial decay, the result follows. Altogether, this completes the proof of the lemma.
\end{proof}

\begin{lemme}\label{lem:aux_driftlevy}
	Let~$L$ be a real-valued L{\'e}vy process. Suppose $\E[L_1]<0$, and there is $\theta>0$ such that $\int_1^{\infty} x^{1+\theta}\Pi_L(\dd x)<\infty$. Then, for any $\alpha\in(0,\theta)$, $\E[G_L^\alpha]<\infty$.
\end{lemme}
\begin{proof}
	Proceeding as in the proof of Lemma~\ref{lem:auxlevyexp}, we deduce that $G_L$ is an infinitely divisible random variable with L{\'e}vy measure $\nu(\dd t)\coloneqq t^{-1}\P(L_t\geq 0)\dd t$. Fix $\alpha\in (0,\theta)$. $G_L$ has finite $\alpha$th moment if and only if $\int_1^\infty t^\alpha\nu(\dd t)<\infty$~\citep[Thm.~25.3]{KenIti1999}. {The latter condition can be checked }using that, by Lemma~\ref{lem:aux:levy}, for $\alpha'\in (\alpha,\theta)$ there is $C>0$ such that $\P(L_t\geq 0)\leq Ct^{-\alpha'}$ for all large~$t$.
 \end{proof}
\section{Technical estimates}\label{sec:TechEst}
In this section we provide technical estimates that are useful throughout the paper.
\subsection{Inequalities related to the function $y\mapsto m_{r,u}(y)$}
\begin{lemme}\label{claim1}
Assume $\int_{(0,1)}r^{-1}\Lambda(\dd r)<\infty$. Then for any $\delta\in(0,1/2)$, and $b\geq \log(2)$, 
\begin{align}
& \sup_{y\leq e^{-b}}\int_{(0,\delta)} \left ( \int_{(0,1)} \left(\sqrt{\frac{y}{\mr_{r,u}(y)}}-1\right) du \right ) \frac{\Lambda(\dd r)}{r^2}\leq \int_{(0,\delta)}\frac{\Lambda(\dd r)}{r}<\infty, \label{majoplusmoins} \\
& \sup_{y\leq e^{-b}}\int_{(0,\delta)\times (0,1)} \left | \sqrt{\frac{y}{\mr_{r,u}(y)}}-1\right | \frac{\Lambda(\dd r)}{r^2}\dd u<\infty. \label{majova}
\end{align}
\end{lemme}
\begin{proof}
For $y\leq e^{-b}$ and $r,u \in (0,1)$ we have $\mr_{r,u}(y)\leq y/(1-r)$ so $\sqrt{y/\mr_{r,u}(y)}-1\geq \sqrt{1-r}-1\geq -r$. We thus get $\int_{(0,\delta)\times (0,1)} \left (\sqrt{\frac{y}{\mr_{r,u}(y)}}-1 \right )_{-} r^{-2}\Lambda(\dd r)\dd u \leq \int_{(0,\delta)}r^{-1}\Lambda(\dd r)$. 

Then, note that since $y\leq e^{-b}<1/2$ and $r\leq \delta< 1/2$, $y/(1-r)\leq 1$. A straightforward calculation then yields \begin{align*}
		\int_{(0,1)} \frac{1}{\sqrt{\mr_{r,u}(y)}}\dd u &=\textstyle 
		\ind{r\leq y}\left(\sqrt{\frac{y}{1-r}}-\sqrt{\frac{y-r}{1-r}}+\sqrt{\frac{1-r}{y}}\right)+\ind{r>y}\left(\sqrt{\frac{y}{1-r}}+\sqrt{\frac{1-r}{y}}\right)\\
		&=\frac{y+1-r-\sqrt{y(y-r)}\,\ind{r\leq y}}{\sqrt{y(1-r)}}.
	\end{align*}
	Using this, we obtain \begin{align*}
		\int_{(0,1)}\big(\sqrt{y/\mr_{r,u}(y)}-1\big)\dd u = \frac{y(1-\sqrt{1-r/y}\,\ind{r\leq y})-\sqrt{1-r}(1-\sqrt{1-r}) }{\sqrt{1-r}}\eqqcolon C(r,y).
	\end{align*}
	
It then suffices to show that $C(r,y)\leq r$ . For $r>y$, using $\sqrt{1-r}\leq 1-r/2$ and that $\delta\leq 1/2$, we obtain  \[C(r,y)\leq\frac{y-r\sqrt{1-r}/2}{\sqrt{1-r}}<r \frac{2-\sqrt{1-\delta}}{2\sqrt{1-\delta}}\leq r.\]
For $r\leq y$, using $1-\frac{r}{2}-\frac{r^2}{2}\leq \sqrt{1-r}\leq 1-r/2$, we obtain 
	\[\textstyle C(r,y)\leq \frac{1}{\sqrt{1-r}}\left(y\left(\frac{r}{2y}+\frac{r^2}{2y^2}\right)-\frac{r}{2}\sqrt{1-r}\right)\leq \frac{r}{\sqrt{1-r}}\left(1-\frac{\sqrt{1-r}}{2}\right)\leq r\left(\sqrt{2}-\frac{1}{2}\right)\leq r,\]
	where we used again that $\delta\leq 1/2$. We thus get \eqref{majoplusmoins}. Since, for any $f$, $|f|=f+2f_{-}$ the proof is complete.
\end{proof}

\begin{lemme}\label{lem:medianboundspoisson}
For all $r\in(0,1/2]$ and $y\in[0,1]$,
\begin{equation}\label{in:mru}
\left\lvert\int_{(0,1)}(m_{r,u}(y)-y)\dd u\right\rvert \leq 2r^2\quad\textrm{and}\quad\left\lvert\int_{(0,1)}(m_{r,u}(y)-y)^2\dd u\right\rvert \leq 4r^2.
\end{equation}
In particular, for any {$c\in (0,1)$ }and $y\in[0,1]$, we have
	\begin{eqnarray}
		&{\left\lvert\int_{(0,c)}\left(\int_{(0,1)}  (\mr_{r,u}(y) -y) \ \dd u\right)  \ \frac{\Lambda (\dd r)}{r^2}\right\rvert}\leq\ 4\Lambda((0,1)), \label{eq:boundfirstm} \\
		&{\left\lvert\int_{(0,1)}\left(\int_{(0,1)}  (\mr_{r,u}(y) -y)^2 \ \dd u\right)  \ \frac{\Lambda (\dd r)}{r^2}\right\rvert}\leq\ 4\Lambda((0,1)). \label{eq:boundsecondtm} 
	\end{eqnarray}
	Moreover, {let $(\Lambda,\mu,\sigma)\in \Theta$, $y \in [0,1]$, $c\in (0,1)$, and $Y^c$ be as defined in \eqref{Yc} with $Y^c_0=y$. Then for any $t\geq 0$, \begin{eqnarray}
		&\E_y\left[\left(\int_{[0,t]\times (0,c]\times(0,1)}\left(\mr_{r,u}(Y^c_{s-}) -Y^c_{s-}\right)  \tilde{N}(\dd s,\dd r, \dd u)\right)^2 \right]\leq  4\Lambda((0,1))\,t. \label{eq:quad} 
	\end{eqnarray}}
\end{lemme}
\begin{proof}
Let $r\in(0,1/2]$. Let us prove the first inequality in \eqref{in:mru}. For this, we split the integral according to the relative position of $u$ with respect to $(y-r)/(1-r)$ and $y/(1-r)$. A straightforward calculation yields that
 \[
	\int_{(0,1)}(m_{r,u}(y)-y)\dd u=\begin{cases}
		\frac{r^2}{(1-r)^2} (\frac1{2} - y), &\text{ if }r\leq y\leq 1-r,\\
		\frac{1}{(1-r)^2} (-\frac{y^2}{2}+r(1-r)y),&\text{ if }y\leq r, \\
		 \frac{1}{(1-r)^2} (\frac{(1-y)^2}{2}-r(1-r)(1-y)) ,&\text{ if }y\geq 1-r.
	\end{cases}\]
A simple analysis of these expressions yields that $\lvert\int_{(0,1)}(m_{r,u}(y)-y)\dd u\rvert\leq r^2/2(1-r)^2$ and the first inequality follows. 	
Proceeding in a similar way, we obtain
 \[
	\int_{(0,1)}(m_{r,u}(y)-y)^2\dd u=\begin{cases}
		\frac{r^2}{(1-r)^3} (y(1-y)(1+r) - \frac{2r}{3}), &\text{ if }r\leq y\leq 1-r,\\
		\frac{y^2}{(1-r)^3} (\frac{y}{3}-ry+r^2(1-r)),&\text{ if }y\leq r, \\
		 \frac{(1-y)^2}{(1-r)^3} (\frac{1-y}{3}-r(1-y)+r^2(1-r))  ,&\text{ if }y\geq 1-r.
	\end{cases}\]
Using these expressions, one can easily show that $\int_{(0,1)}(m_{r,u}(y)-y)^2\dd u\leq r^2/2(1-r)^3$, and the second inequality in \eqref{in:mru} follows. 

Splitting the integral on the left-hand side of \eqref{eq:boundfirstm} (resp. \eqref{eq:boundsecondtm}) for $r\in(0,1/2)$ and $r\in[1/2,1)$, using \eqref{in:mru} in the first case and, in the second case, that $m_{r,u}(y)\in[0,1]$ for all $r,u\in[1/2,1)\times (0,1)$ and $y\in[0,1]$, we get \eqref{eq:boundfirstm} (resp. \eqref{eq:boundsecondtm}). Finally, by properties of the stochastic integral~\citep[Thm. 4.2.3--(2)]{Applebaum2009} and~\eqref{eq:boundsecondtm}, \begin{align*}
	& \E_y\Big[\Big( \int_{[0,t]\times (0,c]\times (0,1)} (\mr_{r,u}(Y^c_{s-}) -Y^c_{s-})  \tilde{N}(\dd s,\dd r, \dd u)\Big)^2 \Big] \\
	= & \E_y\Big[\int_0^t \Big( \int_{(0,c]}\Big(\int_{(0,1)}  (\mr_{r,u}(Y^c_{s-}) -Y^c_{s-})^2 \ \dd u\Big)  \ \frac{\Lambda (\dd r)}{r^2} \Big ) \dd s\Big] \leq  4\Lambda((0,1))\,t.
	\end{align*}
Hence, we get \eqref{eq:quad}.
	\end{proof}

\begin{lemme}\label{hr2}
Define for $r\in(0,1)$ and $x,y\in(0,1)$,
\begin{equation*}
h_1(r,y)\coloneqq \int_0^1 \lvert m_{r,u}(y)-y\rvert\dd u\quad\textrm{and}\quad h_2(r,x,y)\coloneqq \int_0^1 \lvert m_{r,u}(y)-y-(m_{r,u}(x)-x)\rvert \dd u.
\end{equation*}
Then, there is a constant $K>0$ such that
\begin{equation}
h_1(r,y)\leq K r\quad\textrm{and}\quad h_2(r,x,y)\leq K\frac{r}{(1-r)^2}\lvert x-y\rvert.
    \end{equation}
\end{lemme}

\begin{proof}
For $r\in(0,1/2]$, the inequality for $h_1$ is a direct consequence of \eqref{in:mru} and Cauchy-Schwartz inequality; for $r\in(1/2,1)$ the inequality follows using that $m_{r,u}(y)\in[0,1]$. Let us now prove the inequality for $h_2$.
Define for $\alpha>0$ and $u,x,y\in[0,1]$,
\begin{align*}
d_\alpha^{(1)}(u,x,y)&\coloneqq\lvert\mathrm{median}\left(y-1,y,(u-y)\alpha\right)-\mathrm{median}\left(y-1,y,(u-x)\alpha\right)\rvert,\\
d_\alpha^{(2)}(u,x,y)&\coloneqq\lvert\mathrm{median}\left(y-1,y,(u-x)\alpha\right)-\mathrm{median}\left(x-1,x,(u-x)\alpha\right)\rvert.
\end{align*}
Using the definition of the function $m_{r,u}$, basic properties of the median and the triangular inequality, we obtain
$$h_2(r,x,y)\leq \frac{r}{1-r}\int_{(0,1)}\left(d_{\frac{1-r}{r}}^{(1)}(u,x,y)+d_{\frac{1-r}{r}}^{(2)}(u,x,y)\right)\dd u,$$
In what follows we will show that for any $\alpha>0$
\begin{equation}\label{dab}
\int_0^1 d_\alpha^{(1)}(u,x,y) \dd u\leq 3\lvert x-y\rvert\quad\textrm{and}\quad \int_0^1 d_\alpha^{(2)}(u,x,y) \dd u\leq \left(2+\frac{1}{\alpha}\right)\lvert x-y\rvert,
\end{equation}
from where the result follows.
Fix $\alpha\geq 1$ and assume w.l.o.g. that $0\leq y<x\leq 1$. Consider, for $z\in[0,1]$, the partition of $[0,1]$:
$$\textstyle A_1(z)\coloneqq \left[0,z+\frac{z-1}{\alpha}\right]\cap[0,1],\quad A_2(z)\coloneqq \left(z+\frac{z-1}{\alpha},z+\frac{z}{\alpha}\right]\cap[0,1],\quad A_3(z)\coloneqq \left(z+\frac{z}{\alpha},1\right].$$ 
In all this proof, we adopt the convention $(a,b]=[a,b]=\emptyset$ for $a>b$. Since $x>y$, we have for $u\in A_1(y)$, $(u-x)\alpha\leq (u-y)\alpha\leq y-1$ so that $\int_{A_1(y)}d_\alpha^{(1)}(u,x,y)\dd u=0.$
Now, split $A_2(y)$ into 
$$\textstyle A_{2}^1(y,x)\coloneqq A_2(y)\cap \left[0,x+\frac{y-1}{\alpha}\right]\quad\textrm{ and}\quad A_{2}^2(y,x)\coloneqq A_2(y)\cap \left(x+ \frac{y-1}{\alpha},1\right].$$
Note that for $u\in A_2^1(y,x)$, $(u-x)\alpha\leq y-1\leq (u-y)\alpha\leq y$. Thus, 
\begin{align*}
\int_{A_{2}^1(y,x)}d_\alpha^{(1)}(u,x,y)\dd u&=\int_{A_{2}^1(y,x)}|\alpha(u-y)-y+1|\dd u
\leq \alpha\int_{0}^{(x-y)\wedge \frac1{\alpha}}v\dd v\leq \frac{(x-y)}{2}.
\end{align*}
Similarly, for $u\in A_2^2(y,x)$, $y-1\leq (u-x)\alpha\leq (u-y)\alpha\leq y$, so that 
\begin{align*}
\int_{A_{2}^2(y,x)}d_\alpha^{(1)}(u,x,y)\dd u&=\int_{A_{2}^2(y,x)}|\alpha(x-y)|\dd u\leq \alpha(x-y)\textrm{Length}(A_2^2(y,x))\leq x-y.
\end{align*}
Now, we split $A_3(y)$ into
\begin{align*}
\textstyle A_{3}^1(y,x)&\coloneqq \textstyle A_3(y)\cap \left[0,x+\frac{y-1}{\alpha}\right],\quad A_{3}^2(y,x)\coloneqq A_3(y)\cap \left(x+\frac{y-1}{\alpha}, x+\frac{y}{\alpha}\right],\\
\textstyle A_{3}^3(y,x)&\coloneqq\left(x+ \frac{y}{\alpha},1\right].
\end{align*}
Note that for $u\in A_3^3(y,x)$, $y\leq (u-x)\alpha \leq (u-y)\alpha$, so that 
$\int_{A_3^3(y,x)}d_\alpha^{(1)}(u,x,y)=0.$ Moreover, for $u\in A_3^1(y,x)$, $(u-x)\alpha \leq y-1 < y \leq (u-y)\alpha$, so that
\begin{align*}
\int_{A_{3}^1(y,x)}d_\alpha^{(1)}(u,x,y)\dd u&=\int_{A_{3}^1(y,x)}\lvert y-(y-1)\rvert \dd u\leq\left(x-y-\frac{1}{\alpha}\right)\vee 0\leq x-y,
\end{align*}
and for $u\in A_3^2(y,x)$, $y-1\leq (u-x)\alpha \leq y \leq (u-y)\alpha$, so that 
\begin{align*}
\int_{A_{3}^2(y,x)}d_\alpha^{(1)}(u,x,y)\dd u&=\alpha\int_{A_{3}^2(y,x)}\lvert x+\frac{y}{\alpha}-u\rvert \dd u\leq \alpha\int_0^{(x-y)\wedge \frac{1}{\alpha}}v\dd v\leq \frac{x-y}{2}.
\end{align*}
Thus, $\int_0^1 d_\alpha^{(1)}(u,x,y)\dd u\leq 3\lvert x-y\rvert $. It remains to bound $\int_0^1d_\alpha^{(2)}(u,x,y)\dd u$. This time we partition $[0,1]$ as $A_1(x)\cup A_2(x)\cup A_3(x)$. Since $x>y$, we have for $u\in A_3(x)$, $x\leq (u-x)\alpha\leq (u-y)\alpha$, so that
$$\int_{A_3(x)} d_\alpha^{(2)}(u,x,y)\dd u=(x-y)\textrm{Length}(A_3(x))\leq x-y.$$
Now, split $A_2(x)$ into 
$$\textstyle \bar{A}_{2}^1(x,y)\coloneqq A_2(x)\cap \left(x+\frac{y-1}{\alpha},x+\frac{y}{\alpha}\right]\quad\textrm{ and}\quad \bar{A}_{2}^2(x,y)\coloneqq A_2(x)\cap \left[x+\frac{y}{\alpha},1\right].$$
For $u\in \bar{A}_2^1(x,y)$ we have $x-1\leq (u-x)\alpha\leq x$ and $y-1\leq (u-x)\alpha\leq y$, so that $\int_{\bar{A}_2^1(x,y)} d_\alpha^{(2)}(u,x,y)\dd u=0$. For $u\in \bar{A}_2^2(x,y)$ we have $x-1\leq (u-x)\alpha\leq x$ and $y\leq (u-x)\alpha$. Hence, $$\int_{\bar{A}_2^2(x,y)} d_\alpha^{(2)}(u,x,y)\dd u= \int_{\bar{A}_2^2(x,y)} |\alpha(u-x)-y|\dd u\leq\alpha\int_0^{\frac{x-y}{\alpha}}|v|\dd v=\frac{(x-y)^2}{2\alpha}. $$ 
Finally, split $A_1(x)$ into 
$$\textstyle \bar{A}_{1}^1(x,y)\coloneqq A_1(x)\cap \left[0,x+\frac{y-1}{\alpha}\right]\quad\textrm{ and}\quad \bar{A}_{1}^2(x,y)\coloneqq A_1(x)\cap \left(x+\frac{y-1}{\alpha},x+\frac{x-1}{\alpha}\right].$$
Since for $u\in\bar{A}_1^1(x,y)$, $(u-x)\alpha\leq y-1\leq x-1$, then $$\int_{\bar{A}_1^1(x,y)} d_\alpha^{(2)}(u,x,y)\dd u=\int_{\bar{A}_1^1(x,y)}\lvert x-y\rvert \leq x-y.$$ 
Note that $x-1\leq y$. In particular, for $u\in\bar{A}_1^2(x,y)$, $y-1 \leq (u-x)\alpha\leq x-1\leq y$, and hence
$$\int_{\bar{A}_1^2(x,y)} d_\alpha^{(2)}(u,x,y)\dd u= \alpha\int_{\bar{A}_1^2(x,y)} \lvert x+\frac{x-1}{\alpha}-u\rvert \dd u\leq\alpha\int_0^{\frac{x-y}{\alpha}}|v|\dd v=\frac{(x-y)^2}{2\alpha}. $$ 
Therefore, $\int_0^1 d_\alpha^{(2)}(u,x,y)\dd u\leq (2+1/\alpha)\lvert x-y\rvert $. This completes the proof.
\end{proof}

\subsection{Inequalities related to the function $y\mapsto s_r(y)$}	
\begin{lemme} \label{lem:estimatesel}
For any $y \in [0,1]$ and $r \in (0,1)$, we have 
\begin{equation}\label{eq:estimateselpos}
-\frac{r(1-y)}{1-r} \leq s_r(y)-y \leq 0\quad\textrm{and}\quad\frac{y}{2} \leq \frac{y}{1+r} \leq s_r(y) \leq y.
\end{equation}
Similarly, for all $y\in[0,1]$ and $r\in(-1,0)$, we have
\begin{equation}\label{eq:estimateselneg}
0\leq s_r(y)-y \leq \frac{-ry}{1+r}\quad\textrm{and}\quad y \leq s_r(y) \leq \frac{y-r}{1-r} \leq \frac{y+1}{2}.
\end{equation}
\end{lemme}
\begin{proof}
Note first that for $r\in(-1,0)$, $s_r(y)=1-s_{|r|}(1-y)$. Thus, it is enough to {prove \eqref{eq:estimateselpos}}.
Assume that $r\in(0,1)$. We start showing $s_{r}(y)\leq y$. Note that $$y\geq s_{r}(y)\Leftrightarrow \sqrt{(1+r)^2-4ry}\geq (1+r)-2yr\Leftrightarrow 0\geq r^2y(y-1),$$
which holds for all $y\in [0,1]$.
Since $\sqrt{1+x} \leq 1 + x/2$, {$x\geq 0$}, we infer that for any $y \in [0,1]$, 
$$\sqrt{(1-r)^2 + 4r(1-y)} \leq (1-r)\left(1+\frac{2r(1-y)}{(1-r)^2}\right),$$ and hence, $s_r(y) -y\geq -r(1-y)/(1-r)$. It remains to prove that $y/(1+r)\leq s_r(y)$. To this end, note that $\sqrt{1-x} \leq 1 - x/2$, $x\in[0,1]$. Thus, the remaining inequality follows using
\[s_r(y)= \frac{1+r}{2r} \left(1 - \sqrt{1 - \frac{4ry}{(1+r)^2}} \right).\]
\end{proof}	

\begin{lemme}\label{bsr}
There are constants $C_1,C_2>0$ such that, for all $r\in(-1,1)$ with $r\neq 0$ and all {$x,y\in[0,1]$},
\begin{equation} \label{bsreq}
\lvert s_{r}(y)-y\rvert \leq C_1  \frac{\lvert r\rvert}{1-|r|} \quad\textrm{and}\quad |s_{r}(x)-x-(s_r(y)-y)|\leq  C_2\frac{\lvert r\rvert}{1-|r|}\lvert x-y\rvert.
\end{equation}
\end{lemme}
\begin{proof}
The first inequality follows directly from Lemma \ref{lem:estimatesel}. For the second inequality, note that $|s_{r}(x)-x-(s_r(y)-y)|=|s_{|r|}(1-x)-(1-x)-(s_{|r|}(1-y)-(1-y))|$, for $r\in(-1,0)$ and {$x,y\in[0,1]$}.
Hence, it is enough to prove the inequality for $r\in(0,1)$ and $x,y\in[0,1]$.  
Note that for $r\in(0,1)$ and $x,y\in[0,1]$, 
\begin{align*}
\lvert s_r(y)-y-s_r(x) +x\rvert&=\frac{\lvert 2r(x-y)+\sqrt{(1+r)^2-4rx}-\sqrt{(1+r)^2-4ry}\rvert}{2r}\\
&=\frac{1}{2r}\left\lvert 2r(x-y)-\frac{4r(x-y)}{\sqrt{(1+r)^2-4rx}+\sqrt{(1+r)^2-4ry}}\right\rvert\\
&=\lvert x-y\rvert\frac{\lvert \sqrt{(1+r)^2-4rx}+\sqrt{(1+r)^2-4ry}-2\rvert}{\sqrt{(1+r)^2-4rx}+\sqrt{(1+r)^2-4ry}}\\
&\leq \frac{\lvert x-y\lvert}{2(1-r)}\left(\lvert\sqrt{(1+r)^2-4rx}-1\rvert+\lvert\sqrt{(1+r)^2-4ry}-1\rvert\right)\\
&= \frac{\lvert x-y\lvert}{2(1-r)}\left(\frac{\lvert r^2+2r(1-2x)\rvert}{1+\sqrt{(1+r)^2-4rx}}+\frac{\lvert r^2+2r(1-2y)\rvert}{1+\sqrt{(1+r)^2-4ry}}\right).
\end{align*}
Thus, we have
$|s_r(y)-y-s_r(x) +x|\leq 3|x-y||r|/(1-|r|),$
which concludes the proof. \end{proof}

\end{appendix}

\begin{acks}[Acknowledgments]
SH is funded by the Deutsche Forschungsgemeinschaft (DFG, German Research Foundation) -- Projektnummer 449823447. SH thanks Steve Evans and the Department of Statistics at UC Berkeley for their hospitality. FC gratefully acknowledges financial support by the Deutsche Forschungsgemeinschaft (DFG, German Research Foundation) -- SFB 1283/2 2021 -- 317210226. 
\end{acks}
\bibliographystyle{imsart-number} 
\bibliography{reference}       


\end{document}